\documentclass[11pt,letterpaper]{amsart}

\usepackage{amsmath}
\usepackage{amssymb}
\usepackage{enumerate}
\usepackage{txfonts}
\usepackage{mathrsfs}
\usepackage{hyperref,cite} 
\usepackage[all]{xy} 
\usepackage{xcolor} 
\usepackage{framed} 
\usepackage[enableskew,vcentermath]{youngtab}   
\usepackage{verbatim} 
\usepackage{mathdots}     
\usepackage{geometry}
\usepackage{color}
\newcommand{\bbF}{\mathbb{F}}
\newcommand{\bbZ}{\mathbb{Z}}
 
\newcommand{\bG}{\mathbf{G}}
\newcommand{\bH}{\mathbf{H}}
\newcommand{\bc}{\mathbf{c}}
\newcommand{\fS}{\mathfrak{S}}

\newcommand{\tG}{\tilde{G}}
\newcommand{\tR}{\tilde{R}}
\newcommand{\tH}{\tilde{H}}

\newcommand{\cF}{\mathcal{F}}
\newcommand{\cE}{\mathscr{E}}

\def\B{\operatorname{B}}

\def\C{\operatorname{C}}

\def\E{\operatorname{E}}

\newcommand{\Ad}{\operatorname{Ad}} 
\newcommand{\Aut}{\operatorname{Aut}}

\newcommand{\diag}{\operatorname{diag}}
\newcommand{\disc}{\operatorname{disc}} 
\newcommand{\GL}{\operatorname{GL}}

\newcommand{\GU}{\operatorname{GU}}
\newcommand{\IBr}{\operatorname{IBr}}
\newcommand{\Alp}{\operatorname{Alp}}
\newcommand{\rdz}{\operatorname{rdz}}
\newcommand{\dz}{\operatorname{dz}}  
\newcommand{\id}{\operatorname{id}}  

\newcommand{\Ind}{\operatorname{Ind}} 
 
\newcommand{\Int}{\operatorname{Int}} 
\newcommand{\Irr}{\operatorname{Irr}}

\newcommand{\rO}{\operatorname{O}}
 
\newcommand{\Out}{\operatorname{Out}}
\newcommand{\PGL}{\operatorname{PGL}} 
 
\newcommand{\PO}{\operatorname{PO}}

\newcommand{\rank}{\operatorname{rank}} 
\newcommand{\Res}{\operatorname{Res}}
\newcommand{\SL}{\operatorname{SL}}
\newcommand{\SO}{\operatorname{SO}}
\newcommand{\Sp}{\operatorname{Sp}}
\newcommand{\spann}{\operatorname{span}}
\newcommand{\Spin}{\operatorname{Spin}}

\newcommand{\Syl}{\operatorname{Syl}}
\newcommand{\ty}{\operatorname{ty}}

\newcommand{\tp}[1]{\mathsf{#1}}

\def\embed{\hookrightarrow}
 
\let\ka=\kappa
\let\ga=\gamma
\let\ti=\times
\let\eps=\epsilon
\let\la=\lambda
\let\al=\alpha
\let\Ga=\Gamma

\theoremstyle{theorem}
\newtheorem{mainthm}{Theorem}
\newtheorem{thm}{Theorem}[section]
\newtheorem{lem}[thm]{Lemma}
\newtheorem{prop}[thm]{Proposition}
\newtheorem{cor}[thm]{Corollary}

\newtheorem{ques}[thm]{Question}

\theoremstyle{definition}
\newtheorem{defn}[thm]{Definition}
\newtheorem{cons}[thm]{Construction}
\newtheorem{rmk}[thm]{Remark}

\numberwithin{equation}{section}
\newtheorem{rem}[thm]{Remark}

\begin{document}

\title[Radical subgroups]{Radical subgroups of finite reductive groups}

\author{Zhicheng Feng}
\address[Zhicheng Feng]{Shenzhen International Center for Mathematics and Department of Mathematics, Southern 
University of Science and Technology, Shenzhen 518055, China}
\makeatletter
\email{fengzc@sustech.edu.cn} 
\makeatother

\author{Jun Yu}
\address[Jun Yu]{Beijing International Center for Mathematical Research and School of Mathematical Sciences, 
Peking University, No. 5 Yiheyuan Road, Beijing 100871, China}  
\makeatletter
\email{junyu@bicmr.pku.edu.cn} 
\makeatother

\author{Jiping Zhang}
\address[Jiping Zhang]{School of Mathematical Sciences, Peking University, No. 5 Yiheyuan Road, Beijing 100871, China}
\makeatletter
\email{jzhang@pku.edu.cn}
\makeatother

\thanks{Z.~Feng gratefully acknowledges financial support by NSFC (12622101), J.~Yu by NSFC (11971036), and J.~Zhang by NSFC (12431001).
}

\begin{abstract}
Radical subgroups play an important role in both finite group theory and representation theory. This is the first in a 
series of papers of ours in classifying radical $p$-subgroups of finite reductive groups and in verifying the inductive 
blockwise Alperin weight condition for them, contributing to the program of proving the Alperin weight conjecture by 
verifying its inductive condition for finite simple groups. In this paper we present a uniform method for classifying 
radical subgroups of finite reductive groups. As applications, we verify the inductive blockwise 
Alperin weight condition for 2-blocks of classical groups and groups of type $\tp F_4$.   
\end{abstract}  



\maketitle
  
\tableofcontents


\section{Introduction}\label{S:Intro}  

Let $G$ be a finite group and $p$ be a prime. As usual, let $O_p(G)$ denote the largest normal $p$-subgroup of $G$. 
A \emph{radical $p$-subgroup} (or \emph{$p$-radical subgroup}) of $G$ is a subgroup $Q\le G$ with $Q=O_p(N_G(Q))$. 
Radical subgroups are essential in both group theory and representation theory of finite groups; for example, in the study of Brown's complex, Quillen's conjecture (cf. \cite{Qu78}), Alperin's weight conjecture (cf \cite{Al87}), and Dade's conjectures (cf. \cite{Da92,Da94}), etc. 

\vspace{0.5em}
\noindent {\bf Brown's complex and Quillen's conjecture.} 
To the poset of non-identity $p$-subgroups of a finite group $G$ ordered by inclusion, we can associate a simplicial 
complex, which was first introduced by Brown \cite{Br75} in his work on Euler characteristics and cohomology for 
discrete groups and is called Brown’s complex. Quillen's conjecture  asserts that Brown’s complex is 
contractible if and only if $O_p(G)>1$. This conjecture was proved by Quillen for solvable groups, but is still 
open in general. In 1984, Bouc \cite{Bo84a,Bo84b} considered the poset involving only non-identity radical 
$p$-subgroups of $G$, and proved that its associated simplicial complex is homotopy equivalent to Brown’s complex. 
Thus, the classification of radical $p$-subgroups is closely connected with $p$-local geometries (see also 
Yoshiara \cite{Yo98}). Starting with the work of Kn\"orr and Robinson \cite{KR89}, the above simplicial complexes 
are widely used in the study of local-global conjectures. 

\vspace{0.5em}
\noindent {\bf Local-global conjectures.}
In the representation theory of finite groups, local-global conjectures relate the representation theory of finite 
groups with that of their local subgroups. For example, the McKay conjecture, the Alperin weight conjecture and Dade's 
conjectures, etc., are deep local-global conjectures. They are central to the representation theory of finite groups. 
Starting with the reduction of the McKay conjecture by Isaacs--Malle--Navarro \cite{IMN07} in 2007, many important 
local-global conjectures have been reduced to questions on finite simple groups. Then, the representation theory of 
quasi-simple groups, especially finite reductive groups, turns more and more into focus. The classification of radical 
subgroups of finite reductive groups is thus essential for the study of some local-global conjectures, such as the Alperin 
weight conjecture and Dade's conjectures.

\vspace{0.5em}
\noindent {\bf Radical subgroups.}
By the work of Borel and Tits \cite{BT71}, radical subgroups of finite reductive groups in defining characteristic are 
known: they are just the unipotent radicals of parabolic subgroups. In the non-defining characteristic situation, 
Alperin--Fong \cite{AF90} initiated the classification of radical subgroups of finite reductive groups in 1990, 
which was pursued by An and his collaborators (see, for instance 
\cite{An92,An93a,An93b,An93c,An94,An94b,An95,An98,AD14, ADH14, ADH18, AH13}). However, a general classification of radical 
subgroups and weights is still unavailable to the present day, especially for groups of exceptional type. In their survey 
paper \cite[\S4]{KM19}, Kessar and Malle pointed out that ``for exceptional groups of larger rank that seems quite challenging 
at the moment".  

In 1989, Alperin \cite{al90} presented an analogy between finite group theory and Lie group theory and speculated a Lie approach 
to finite groups. Alperin called the normalizer of a $p$-radical subgroup a \emph{$p$-parabolic subgroup}. The notions of 
$p$-parabolic subgroups and Brown's complexes are analogies of parabolic subgroups and buildings in Lie theory. Note that 
radical subgroups of complex reductive groups are connected to the combinatorial structure of the root datum of a complex 
reductive group, and they play essential roles in representation theory via induced representations, etc. With the method in 
this paper, it comes the great hope of classifying radical subgroups of finite reductive groups in general. We would like to 
mention a question: do radical subgroups of finite reductive groups have similar structural and representation theoretical 
meaning, just as radical subgroups of complex reductive groups?    

\vspace{0.5em}
\noindent {\bf Method of classification.} 
In this paper, we present a uniform method of classifying radical subgroups of finite reductive groups in non-defining 
characteristic. Let us explain the method. Let $1\neq R$ be a radical $p$-subgroup in a finite reductive 
group $G(\mathbb{F}_{q})$ with $q$ a prime power coprime to $p$. For simplicity, assume that $G$ is of adjoint type 
(otherwise we consider the adjoint quotient of it instead). Then, $A(R):=\Omega_{1}(Z(R))=\{x\in Z(R)\mid x^p=1\}$ is an elementary abelian $p$-subgroup. 
By viewing $A(R)$ as a subgroup of $G(\bar{\mathbb{F}}_{q})$ and considering all $\mathbb{F}_{q}$-forms of it, we understand 
the conjugacy class of $A(R)$ in $G(\mathbb{F}_{q})$. By the ``process of reduction modulo $\ell$" (cf. \cite[Thm. A.12]{GR98}, 
due to Larsen), it is sufficient to classify elementary abelian $p$-subgroups of complex linear groups $G(\mathbb{C})$. It is 
also well-known that elementary abelian $p$-subgroups of a complex reductive group and those of its compact form are equivalent 
objects. Note that, the classification of elementary abelian $p$-subgroups of complex (or compact) simple groups are systematically 
investigated in \cite{Gr91, AGMV08, Yu13}. In \cite{Yu13}, the second author classified elementary abelian 
2-subgroups of exceptional compact simple Lie groups by dividing them into different families and defining numerical invariants 
to parametrize conjugacy classes of subgroups in each family.   

We observe that: $R$ is a radical $p$-subgroup of $Z_{G(\mathbb{F}_{q})}(A(R))$. Since $Z_{G(\mathbb{F}_{q})}(A(R))$ 
is always smaller than $G$, this gives a way of reduction for the classification of radical $p$-subgroups. However, the 
structure of $Z_{G(\mathbb{F}_{q})}(A(R))$ could be complicated. We have another important observation: for certain subgroup 
$B$ of $A(R)$, $R$ is also a radical 2-subgroup of $Z_{G(\mathbb{F}_{q})}(B)$. This brings great advantage in the reduction: 
the conjugacy class of $B$ has few possibilities and the structure of $Z_{G(\mathbb{F}_{q})}(B)$ is much simpler. With this 
reduction method, we give a new proof of the classification of radical $p$-subgroups of classical groups, which was first 
done in a series of papers of Alperin--Fong and An \cite{AF90,An92,An93a,An93b,An94} in the 1990s. We treat the classification 
of radical subgroups of classical groups in characteristic 2, which was missing in literature. We also correct some mistakes 
in the list of radical $p$-subgroups in the literature.  

Using our reduction method and based on the classification of elementary 2-subgroups in \cite{Yu13}, we classify radical 
2-subgroups of the Chevalley groups $\tp F_4(q)$ when $q$ is odd. The classification of radical $p$-subgroups (for all 
primes $p$) of exceptional groups of types $\tp E_6$, $\tp E_7$, $\tp E_8$ will be pursued in our ongoing works, 
along with the investigation of verifying the inductive blockwise Alperin weight condition for them.  

We shall mention a different approach taken by An--Dietrich--Litterick \cite{ADL23}. They aimed to first classify elementary abelian $p$-subgroups of finite reductive groups, and subsequently classify maximal local subgroups and radical $p$-subgroups. In 
particular, their progress in classifying elementary abelian $2$-subgroups of the finite exceptional groups $\tp F_4$ and $\tp E_6$--$\tp E_8$ is grounded in the prior classification for compact Lie groups due to the second author \cite{Yu13}.

A recent study by An \cite{An25} also investigates the weight subgroups of quasi-isolated 2-blocks of the Chevalley groups $\tp F_4(q)$. 
While preceding studies \cite{AD14,ADH14,ADH18,AH13} classified radical subgroups for certain exceptional groups by relying extensively on maximal subgroup structures, the approach in \cite{An25} is methodologically different. Instead,
 the arguments in \cite{An25} depend principally on the classification of elementary abelian $2$-subgroups, based on the work of the second author \cite{Yu13} which achieves a classification of elementary abelian $2$-subgroups of compact Lie groups.
The two works are independent and employ distinct methodological approaches. 
In comparison, \cite{An25} classifies the weight subgroups of quasi-isolated 2-blocks of the groups $\tp F_4(q)$ and proves the Alperin weight conjecture for these blocks. 
In our work, having independently classified the weights for quasi-isolated 2-blocks of the groups $\tp F_4(q)$, we proceed to provide a complete verification of the inductive BAW conditions for groups of type $\tp F_4$.

\vspace{0.5em}
\noindent {\bf Alperin weight conjecture.}
The Alperin weight conjecture, which was proposed by Alperin \cite{Al87} in the 1980s, predicts that the number 
of modular irreducible representations (up to isomorphism) of a finite group should equal the number of weights (up to conjugacy). 
Moreover, this equality should hold blockwisely. This conjecture has created great interest and its truth would have 
significant consequences. The motivation for the concept of weights of finite groups stems from connections and analogies with 
the theory of Lie groups. On the other hand, the investigation of weights of finite reductive groups points towards 
hidden theories in the local representation theory of finite groups. For example, the classification of weights of symmetric 
groups and general linear groups by Alperin--Fong \cite{AF90} highlights a well-known analogy between the representation 
theory of symmetric and general linear groups. 

The Alperin weight conjecture was reduced to a verification in finite simple groups by Navarro and Tiep \cite{NT11} for 
the non-blockwise version in 2011, and by Sp\"ath \cite{Sp13} for the blockwise version in 2013. They proved that if all 
non-abelian simple groups satisfy the inductive blockwise Alperin weight (BAW) condition, then the blockwise Alperin weight 
conjecture holds for all finite groups. So far, the inductive BAW condition has been verified for certain families of 
simple groups; for the recent developments we refer to the survey papers of Kessar--Malle \cite{KM19}, and of
the first and third authors \cite{FZ22}. However, it is still a serious challenge to establish the inductive BAW condition for finite reductive groups, especially for spin groups and for groups of exceptional type; in fact, more fundamentally, even the original Alperin weight conjecture remains open for most of these groups.

In this paper, we consider the spin groups at the prime 2 and prove the following theorem.

\begin{mainthm}\label{mainthm:BAW-BD}
The inductive blockwise Alperin weight condition holds for every 2-block of every simple group of Lie type $\tp B_n$, 
$\tp D_n$ or $^2\tp D_n$ without an exceptional Schur multiplier.
\end{mainthm}

Since the verification of the inductive BAW condition for groups of Lie type in non-defining characteristic was reduced to 
quasi-isolated blocks by Li and the first and third authors \cite{FLZ22}, we mainly treat the principal 2-blocks of the spin groups here. 
Moreover, together with \cite{FM22,FLZ23}, Theorem~\ref{mainthm:BAW-BD} completes the verification of the inductive blockwise Alperin weight condition for all simple classical groups except $\Omega_7(3)$.

The determination of radical subgroups is crucial in the verification of the inductive BAW condition.
For example, in the works \cite{AHL21, FLZ21, FLZ23,Sc16}, groups of types 
$\tp A_n$, $^2\tp A_n$,  $\tp F_4$, $\tp G_2$ and $^3\tp D_4$ are considered and the inductive BAW 
condition is established except for groups of type $\tp F_4$ at the prime 2. Also, there is still no substantive breakthrough 
for exceptional groups of types $\tp E_6$--$\tp E_{8}$ so far. What prevents people from extending to those situations:  
the missing classification of the radical subgroups!

After classifying the radical 2-subgroups of the Chevalley groups $\tp F_4(q)$ (with odd $q$), we are able to prove the following.

\begin{mainthm}\label{mainthm:BAW-F4}
The inductive blockwise Alperin weight condition holds for the Chevalley groups $\tp F_4(q)$ (with odd $q$) and the prime 2.
\end{mainthm}

This completes the proof of the inductive BAW condition for finite simple groups of type $\tp F_4$, since the odd primes situation 
has been dealt with by An--Hiss--L\"ubeck \cite{AHL21}. As a further important ingredient, we obtain a Jordan decomposition 
for the weights of quasi-isolated 2-blocks of the groups $\tp F_4(q)$ along the way; see Remark~\ref{rmk:Jordan-weights}. Many 
arguments in classifying weights for quasi-isolated blocks have variants in other exceptional types. This will be explained 
in our ongoing works.

\vspace{0.6em}
\noindent {\bf Structure of this paper.}
The organization of this paper is as follows. After collecting some preliminary results in Section~\ref{S:Preli}, we give 
a new proof of the classification of radical 2-subgroups of classical groups in Section~\ref{S:radical}. In particular, 
we correct errors in the list of radical 2-subgroups in \cite{An93a}, list the radical subgroups providing 2-weights, 
as well as determine the action of field automorphisms. Section \ref{S:F4} is devoted to classifying radical 2-subgroups 
of the Chevalley group $\tp F_4(q)$. The process in Sections \ref{S:radical} and \ref{S:F4} indicates how our strategy works 
for general finite reductive groups. In Section \ref{S:BAW}, we investigate the inductive BAW condition for 2-blocks of 
classical groups, determine weights of quasi-isolated 2-blocks of the groups $\tp F_4(q)$, and prove Theorems 
\ref{mainthm:BAW-BD} and \ref{mainthm:BAW-F4}. 



\section{Preliminaries}\label{S:Preli}

\noindent {\bf Notation.}
Let $G$ be a group. We denote by $\Aut(G)$ the automorphism group of $G$, and by $\Aut^0(G)$ the subgroup of $\Aut(G)$ acting 
trivially on $Z(G)$. Let $H$ be a subgroup of $G$. We denote by $N_G(H)$ (resp. $Z_G(H)$) the normalizer (resp. centralizer) 
of $H$ in $G$. For a finite abelian group $G$ and a prime $p$, write $G_{p}$ for the unique Sylow $p$-subgroup of $G$. 

Write $|X|$ or $\sharp{X}$ for the cardinality of a set $X$. Denote by $v_p$ the exponential valuation associated to the prime 
$p$, normalized so that $v_p(p)=1$. For a non-zero integer $n$, we write $n_p=p^{v_p(n)}$.   

\subsection{The Alperin weight conjecture}
Let $G$ be a finite group and $p$ be a prime. If $N\unlhd G$, sometimes we identify characters of $G$ with kernel containing 
$N$ with characters of $G/N$. Let $\dz(G)$ denote the set of irreducible characters $\chi$ of $G$ that are of $p$-defect zero 
(i.e., $v_p(\chi(1))=v_p(|G|)$).  

For $N\unlhd G$ and $\theta\in\Irr(N)$, we set 
\[\rdz(G\mid\theta)=\{\,\chi\in\Irr(G\mid\theta)\mid v_p(\chi(1)/\theta(1))=v_p(|G/N|)\,\}.\]

\begin{lem}\label{ext-rdz}
	Let $N\unlhd G$ be finite groups, and let $\theta\in\Irr(N)$.
\begin{enumerate}[\rm(1)]
\item Let $H$ be a normal subgroup of $G$ containing $N$. If $\chi\in \rdz(G\mid\theta)$ and $\vartheta\in\Irr(H\mid\chi)\cap\Irr(H\mid\theta)$, then $\vartheta\in\rdz(H\mid\theta)$.
\item If $O_p(G_\theta/N)\ne 1$, then $\rdz(G\mid\theta)=\emptyset$.
\end{enumerate}
\end{lem}

\begin{proof}	 
By Clifford theory, one has that $\chi(1)/\vartheta(1)$ divides the order of $G/H$, and thus (1) holds. For (2) we note that 
$\rdz(G\mid\theta)=\emptyset$ if and only if $\rdz(G_\theta\mid\theta)=\emptyset$, and thus we may assume that $G=G_\theta$. 
Suppose $O_p(G/N)\ne 1$ and let $M\unlhd G$ such that $N\subset M$ and $M/N=O_p(G/N)$. For any $\psi\in\Irr(M\mid\theta)$, 
we get that  $(\psi(1)/\theta(1))^2\mid\sharp{(M/N)}$ from Clifford theory (see, e.g., \cite[Thm. 19.3]{Hu98}). 
So $\rdz(M\mid\theta)=\emptyset$ which implies that $\rdz(G\mid\theta)=\emptyset$.
\end{proof} 

A consequence of Lemma~\ref{ext-rdz} is that $\dz(G)\ne \emptyset$ occurs only when $O_p(G)=1$.

A \emph{($p$-)weight} of a finite group $G$ means a pair $(R,\varphi)$ where $R$ is a $p$-subgroup of $G$ and $\varphi\in\dz(N_G(R)/R)$ 
(i.e., $\varphi\in\rdz(N_G(R)\mid 1_R)$). We often identify characters in $\Irr(N_G(R)/R)$ with their inflations to $N_G(R)$. 
Denote the conjugacy class containing $(R,\varphi)$ by $\overline{(R,\varphi)}$.
For a $p$-subgroup $R$ of $G$, we call $R$ a \emph{weight subgroup} of $G$ if $R$ provides a weight $(R,\varphi)$  of $G$. 
By Lemma~\ref{ext-rdz}, weight subgroups are necessarily radical subgroups. 

Each weight  may be assigned to a unique block. Let $B$ be a $p$-block of $G$. A weight $(R,\varphi)$ of $G$ is called 
a \emph{$B$-weight} if $b^G=B$, where $b$ is the block of $N_G(R)$ containing the character $\varphi$ and $b^G$ denotes 
the Brauer induced block of $b$. If a $p$-subgroup $R$ provides a $B$-weight of $G$, then $R$ is called a 
\emph{$B$-weight subgroup} of $G$.  Denote the set of conjugacy classes of $B$-weights of $G$ by $\Alp(B)$. The blockwise 
version of the Alperin weight conjecture (cf. \cite{Al87}) asserts that  
\[|\IBr(B)|=|\Alp(B)|\ \textrm{for every $p$-block $B$ of every finite group}.\] 

Let $G$ be a finite group. A \emph{Brauer pair} of $G$ is a pair $(R,b)$ where $R$ is a $p$-subgroup of $G$ and $b$ is a 
block of $Z_G(R)$. A Brauer pair $(R,b)$ is called \emph{self-centralizing} if $Z(R)$ is the defect group of the block $b$. 
The set of Brauer pairs of $G$ is partially ordered; we can define the inclusion relation ``$\le$'' on the set of Brauer 
pairs (see, for example, \cite[\S18 and \S41]{Th95}). Let $B$ be a block of $G$. Then a Brauer pair $(R,b)$ is called a 
\emph{$B$-Brauer pair} of $G$ if $(1,B)\le (R,b)$.  

By Clifford theory, weights can be constructed in the following manner (see for instance \cite[p.~373--374]{Al87} or \cite[p.~3]{AF90}).

\begin{cons}\label{construction-of-weights}
Let $B$ be a $p$-block of a finite group $G$. For a radical $p$-subgroup $R$ of $G$ and a self-centralizing $B$-Brauer 
pair $(R,b)$, we denote by $\theta$ the canonical character of $b$. As $\theta$ can be regarded as a character of 
$Z_G(R)/Z(R)\cong RZ_G(R)/R$, we denote the inflation of $\theta\in\Irr(RZ_G(R)/R)$ to $RZ_G(R)$ by $\vartheta$. Then 
for every $\phi\in\rdz(N_G(R,b)\mid\vartheta)$, the pair $(R,\Ind^{N_G(R)}_{N_G(R,b)}(\phi))$ is a $B$-weight of $G$. 
As a consequence of the Clifford correspondence, distinct characters $\phi$ yield distinct weights. A complete set of 
representatives of the $G$-conjugacy classes of $B$-weights of $G$ are then obtained by letting $(R,b)$ run through 
a complete set of representatives of the $G$-conjugacy classes of self-centralizing $B$-Brauer pairs of $G$ such that 
$R$ is a radical $p$-subgroup of $G$, and for each such $(R,b)$ letting $\phi$ run through $\rdz(N_G(R,b)\mid\vartheta)$. 
\end{cons} 

Let $B_0(G)$ denote the principal $p$-block of $G$. A $B_0(G)$-weight of $G$ is also called a \emph{principal $p$-weight} 
of $G$, while a $B_0(G)$-weight subgroup of $G$ is called a \emph{principal $p$-weight subgroup} of $G$. 

\begin{lem}\label{lem:B_0-wei}
	Let $G$ be a finite group, $R$ a $p$-subgroup of $G$ and $\varphi\in\Irr(N_G(R))$.
	Then $(R,\varphi)$ is a principal weight of $G$ if and only if $Z(R)\in\Syl_p(Z_G(R))$, $R Z_G(R)\subset\ker(\varphi)$ and $\varphi\in\dz(N_G(R)/RZ_G(R))$.
	\end{lem}
	
	\begin{proof}
	Let $\theta\in\Irr(Z_G(R)\mid\varphi)$ and $b$ be the $p$-block of $Z_G(R)$ containing $\theta$. 
	
	First assume that $Z(R)\in\Syl_p(Z_G(R))$, $R Z_G(R)\subset\ker(\varphi)$ and $\varphi\in\dz(N_G(R)/RZ_G(R))$. Then $\varphi\in\dz(N_G(R)/R)$, $\theta=1_{Z_G(R)}$ and $b$ is the principal $p$-block of $Z_G(R)$. So $(R,b)$ is a $B_0(G)$-Brauer pair of $G$ by Brauer's Third Main Theorem (see e.g. \cite[Thm.~(6.7)]{Na98}), which implies that $(R,\varphi)$ is a principal weight of $G$.
	
	Conversely, we suppose that $(R,\varphi)$ is a principal weight of $G$.
	Then $(R,b)$ is a self-centralizing $B_0(G)$-Brauer pair of $G$ and $\theta$ is the canonical character of $b$. In addition, Brauer's Third Main Theorem implies that $b$ is the principal block of $Z_G(R)$.
	Therefore, $Z(R)\in\Syl_p(Z_G(R))$, $\theta=1_{Z_G(R)}$, and thus $\varphi\in\dz(N_G(R)/RZ_G(R))$.
	\end{proof}	
	
	Let $N\unlhd G$ be finite groups. In \cite[\S2]{BS22}, Brough and Sp\"ath defined a relationship ``cover" for the weights between $G$ and $N$ and established Clifford theory for weights.
	Let $(R,\varphi)$ be a weight of $G$ and $(S,\psi)$ a weight of $N$, then we say $\overline{(R,\varphi)}$ \emph{covers} $\overline{(S,\psi)}$ if $(R,\varphi)$ covers $(S,\psi)^g$ for some $g\in G$. Denote by $\Alp(N\mid\overline{(R,\varphi)})$ the elements of $\Alp(N)$ covered by $\overline{(R,\varphi)}$.

	\subsection{Properties of radical subgroups and weight subgroups}

	For a finite group $G$ and a prime $p$, we denote by $\Re_{p}(G)$ the set of radical $p$-subgroups of $G$, and by $\Re_{w}(G)$ the set of $p$-weight subgroups of $G$.   
	
	\begin{lem}\label{L:R1}
	Let $H\subset G$ be finite groups. If $R$ is a radical $p$-subgroup of $G$, then it is a radical $p$-subgroup of $H$ whenever 
	$H\supset N_{G}(R)$. 
	\end{lem}
	
	\begin{proof}
	Since $N_{G}(R)\subset H\subset G$, we have $R\subset N_{H}(R)=N_{G}(R)$. Then, $R$ is a radical $p$-subgroup of $H$ as well.
	\end{proof}
	
	\begin{lem}[{\cite[Lemma 2.1]{OU95}}]\label{L:R2}
	If $R$ is a radical $p$-subgroup of a finite group $G$, then $O_{p}(G)\subset R$. 
	\end{lem} 
	
	\begin{lem}[{\cite[Lemma 2.1]{OU95}, \cite[Lemma 2.2]{Fe19} and \cite[Corollary~2.12]{BS22}}]\label{L:R3} 
	Let $H$ be a normal subgroup of $G$. Then the map \[R\mapsto R\cap H\] gives a surjection $\Re_{p}(G)\rightarrow\Re_{p}(H)$. 
	Moreover, this also gives a surjection $\Re_{w}(G)\rightarrow\Re_{w}(H)$.
	\end{lem}
	
	\begin{lem}[{\cite[Lemma 2.2]{OU95} and \cite[Lemma 2.3]{NT11}}]\label{L:R6}
	If $G=G_{1}\times G_{2}$, then the map $(R_{1},R_{2})\mapsto R_{1}\times R_{2}$ gives a bijection 
	$\Re_{p}(G_{1})\times\Re_{p}(G_{2})\rightarrow\Re_{p}(G)$.   
	\end{lem} 
	
	\begin{lem}\label{L:R4}
	Suppose that $N$ is a nilpotent normal subgroup of a finite group $G$ such that $O_{p'}(N)\le Z(G)$. Then the map $R\mapsto RN/N$ 
	gives a bijection $\Re_{p}(G)\rightarrow\Re_{p}(G/N)$. 
	\end{lem}
	
	\begin{proof}
	Choose a Sylow $p$-subgroup $N_{p}$ of $N$. Since $N$ is nilpotent, then $N_{p}$ is a normal subgroup of $N$. By considering the 
	surjections $G\rightarrow G/N_{p}$ and $G/N_{p}\rightarrow G/N$ separately, we may assume that $N$ is a $p$-subgroup or a central
	$p'$-subgroup. 
	
	Suppose that $N$ is a normal $p$-subgroup of $G$. By Lemma \ref{L:R2}, all $R\in\Re_{p}(G)$ contain $N$. Then, the conclusion is 
	clear. 
	
	Suppose that $N$ is a central $p'$-subgroup of $G$. Then, the map $R\mapsto RN/N$ gives a bijection between the set of $p$-subgroups 
	of $G$ and the set of $p$-subgroups of $G/N$. Now let $R$ be a $p$-subgroup of $G$ and put $R'=RN/N\subset G/N$. It suffices to 
	show that: \[R\in\Re_{p}(G)\Leftrightarrow R'\in\Re_{p}(G/N).\] This follows from the statement $N_{G}(RN)=N_{G}(R)$.    
	\end{proof}
	
	We mention that Lemma~\ref{L:R4} is also proved in \cite[Prop.~2.7]{Da94} when $N$ is a $p$-group, and is proved in 
\cite[Lemma~2.3 (c)]{NT11} when $N$ is a $p'$-group. 
	
	\begin{lem}\label{L:R5}
	Let $R$ be a radical $p$-subgroup of a finite group $G$. If $A$ is a subgroup of $Z(R)$ which is stable under the conjugation 
	action of all elements $n\in N_{G}(R)$, then $R$ is a radical $p$-subgroup of both $N_{G}(A)$ and $Z_{G}(A)$.  
	\end{lem}  
	
	\begin{proof}
	By assumption, we have $N_{G}(R)\subset N_{G}(A)$. By Lemma \ref{L:R1}, $R$ is a radical $p$-subgroup of $N_{G}(A)$. We have 
	$R\subset Z_{G}(A)$ and $Z_{G}(A)$ is a normal subgroup of $N_{G}(A)$. By Lemma \ref{L:R3}, $R$ is also a radical $p$-subgroup 
	of $Z_{G}(A)$.    
	\end{proof}
	
	\begin{lem}\label{L:R8}
	Let $R$ be a radical $p$-subgroup of a finite group $G$. Let $H$ be a subgroup of $G$ which contains $R$ and is normalized by 
	$N_{G}(R)$. Then, $R$ is a radical $p$-subgroup of $H$.  
	\end{lem}
	
	\begin{proof}
	Since $N_{G}(R)$ normalizes $H$ and $R$, then it normalizes $N_{H}(R)$. Thus, $O_{p}(N_{H}(R))$ is a normal subgroup of 
	$N_{G}(R)$. Hence, $O_{p}(N_{H}(R))\subset O_{p}(N_{G}(R))$. As $R$ is a radical $p$-subgroup of $G$, we have $O_{p}(N_{G}(R))=R$. 
	Therefore, \[R\subset O_{p}(N_{H}(R))\subset O_{p}(N_{G}(R))=R.\] Hence, $O_{p}(N_{H}(R))=R$ and $R$ is a radical $p$-subgroup 
	of $H$.
	\end{proof}
	
Let $c\in\mathbb{Z}_{\geq 1}$. Then $\fS_{2^{c}}$ has an elementary abelian 2-subgroup $A_{c}$ generated by the permutations 
\[x\mapsto \phi_{i}(x),\ 1\leq x\leq 2^{c}\] ($1\leq i\leq c$), where $\phi_{i}(x)\equiv x+2^{i-1}\pmod{2^{i}}$ and 
$\lceil\frac{\phi_{i}(x)}{2^{i}}\rceil=\lceil\frac{x}{2^{i}}\rceil$.  
	
	\begin{lem}\label{L:R7} 
	Let $G=H^{2^{c}}\rtimes A_{c}=H\wr A_{c}$ and $R$ be a radical 2-subgroup of $G$ meeting every coset of $H^{2^{c}}$ in $G$. 
	Then $R$ is conjugate to some $R'^{2^{c}}\rtimes A_{c}$, where $R'\in\Re_2(H)$.   
	\end{lem}  
	
	\begin{proof}
	By Lemma \ref{L:R3}, $R\cap H^{2^{c}}\in\Re_2(H^{2^{c}})$. By Lemma \ref{L:R6}, $R\cap H^{2^{c}}=\prod_{1\leq i\leq 2^{c}}R_{i}$, 
	where each $R_{i}\in\Re_2(H)$ ($1\leq i\leq 2^{c}$). As it is assumed that $R$ intersects with each coset of $H^{2^{c}}$ in $G$, 
	such $R_{i}$ must be pairwise conjugate. Replacing $R$ by a conjugate one if necessary, we assume that all $R_{i}=R'$ where 
	$R'\in\Re_2(H)$. Then, \[R\subset N_{G}(R'^{2^{c}})=(N_{H}(R')^{2^{c}})\rtimes A_{c},\] $R\cap N_{H}(R')^{2^{c}}=R'^{2^{c}}$ and 
	$R/R'^{2^{c}}\cong A_{c}$. Considering $R/R'^{2^{c}}$ as a subgroup of $((N_{H}(R')/R')^{2^{c}})\rtimes A_{c}$, one can show 
	that it is conjugate to $A_{c}$. Then, $R$ is conjugate to $R'^{2^{c}}\rtimes A_{c}$. 
	\end{proof}
	
	The above lemma still holds while $A_{c}$ is substituted by any other transitive abelian 2-subgroup of $\fS_{2^{c}}$. If we 
	consider radical $p$-subgroups, it holds true if $2^{c}$ is substituted by $p^{c}$ and $A_{c}$ is substituted by a transitive 
	abelian $p$-subgroup of $\fS_{p^{c}}$. The same proof is valid for these slight generalizations. 
	
	\begin{lem}\label{L:R9}
	Let $R$ be a radical $p$-subgroup of $G$. Put $D=Z_{G}(R)$. Then, $R$ is a radical $p$-subgroup of $Z_{G}(D)$. 
	\end{lem}
	
	\begin{proof}
	It is clear that $N_{G}(R)\subset N_{G}(D)$. By Lemma \ref{L:R1}, $R$ is a radical $p$-subgroup of $N_{G}(D)$. We have 
	$R\subset Z_{G}(D)$ and $Z_{G}(D)$ is a normal subgroup of $N_{G}(D)$. By Lemma \ref{L:R3}, $R$ is a radical $p$-subgroup 
	of $Z_{G}(D)$. 
	\end{proof}
	
	\begin{lem}\label{L:R10}
	Let $G$ be a finite group and $R$ be a $p$-subgroup of $G$ with $Z(R)\in \Syl_p(Z_G(R))$. If $O_p(N_G(R)/RZ_G(R))=1$, then $R$ is a radical $p$-subgroup of $G$.
	\end{lem}	
	\begin{proof}
	Suppose that $R$ is not radical, which means $R\subsetneq O_p(N_G(R))$.
	By $Z(R)\in \Syl_p(Z_G(R))$, one has $RZ_G(R)\subsetneq O_p(N_G(R))Z_G(R)$, which implies 
	$O_p(N_G(R)/RZ_G(R))\ne1$.
	This is a contradiction and thus $R$ must be radical.
	\end{proof}

	\begin{lem}\label{lem:weight-quo}
		Let $G$ be a finite group and $N$ be a normal $p$-subgroup of $G$. 
		If $(R,\varphi)$ is a $p$-weight of $G$, then $N\subset R$. The character $\varphi$ gives naturally a character $\bar\varphi$ of $N_G(R)/N$.
			Then the map $(R,\varphi)\to (R/N,\bar\varphi)$ induces a bijection $\Alp(G)\to\Alp(G/N)$. 
	
		In particular, if $N$ is a $p$-subgroup of $Z(G)$, then this bijection also preserves blocks.	
		\end{lem}	
		
		\begin{proof}
	This follows by Lemma \ref{L:R2}.
			\end{proof}
	
	\begin{lem}\label{lem:center-weisub}
		Let $G$ be a finite group, and $B$ a block of $G$.	If $R$ is a $B$-weight subgroup of $G$, then there exists a defect group $D$ of $B$ such that $Z(D)\subset Z(R)\subset R\subset D$.
		\end{lem}
		
		\begin{proof}
		Let $(R,\varphi)$ be a $B$-weight of $G$, $\theta\in\Irr(Z_G(R)\mid\varphi)$ and $b$ be the block of $Z_G(R)$ containing $\theta$. Then $(R,b)$ is a $B$-Brauer pair of $G$.
		Following \cite[Thm.~(9.24)]{Na98}, one sees that there exists a defect group $D$ of $B$ such that $Z(D)\subset Z(R)\subset R\subset D$.
		\end{proof}
	
	\begin{lem}\label{lem:weight-subgp-sub}
	Let $G$ be a finite group, $H$ be a subgroup of $G$ and $R$ be a $p$-subgroup of $H$. Suppose that $N_H(R)\unlhd N_G(R)$.
	If $R$ is a $p$-weight subgroup of $G$, then it is also a $p$-weight subgroup of $H$.
	\end{lem}
	\begin{proof}
	This follows by Lemma \ref{ext-rdz}.
	\end{proof}	
	
	\begin{rmk}\label{rmk:voe}
	Keep the hypothesis and setup of Lemma~\ref{lem:weight-subgp-sub}. If, moreover, $R$ is a principal weight subgroup of $G$, then it is also a principal weight subgroup of $H$.
	In fact, in the proof of Lemma~\ref{lem:weight-subgp-sub}, we have that $\varphi$ lies in the principal block of $N_G(R)$ by Brauer's Third Main Theorem. From this, for every $\vartheta\in\Irr(N_H(R)\mid\varphi)$, $\vartheta$ lies in the principal block of $N_H(R)$, which implies that $(R,\vartheta)$ is a principal weight of $H$.
	\end{rmk}
	
	\begin{rmk}\label{rmk:voe2}
	Let $R$ be a principal weight subgroup of $G$. Suppose that $B$ is a subgroup of $Z(R)$ such that $B\unlhd N_G(Z(R))$. Write $H=Z_G(B)$.
	Then $N_H(R)\unlhd N_G(R)$, and by Remark \ref{rmk:voe}, $R$ is also a principal weight subgroup of $H$.
	\end{rmk}

			\begin{rmk}\label{rmk:extension}
				Let $\tG$ be a finite group, and $G$ be a normal subgroup of $\tG$ such that $\tG/G$ is solvable. 
			 If $R$ is a principal weight subgroup of $G$, then there exists a principal weight subgroup $\tilde R$ of $\tG$ such that $\tilde R\cap G=R$.
				In fact, by induction we may assume that $\tG/G$ is cyclic.
				Let $(R,\varphi)$ be a principal weight of $G$ and let $(\tR,\tilde\varphi)$ be a $p$-weight of $\tG$ covering  $(R,\varphi)$. Suppose that $(\tilde R,\tilde\varphi)$ is a $\tilde B$-weight of $\tG$, where $\tilde B$ is a block of $\tG$.
			Then by \cite[Lemma~2.3]{KS15}, $\tilde B$ covers $B_0(G)$.
			So $B_0(\tG)=\la\otimes\tilde B$, where $\la\in\Irr(\tG/G)$ is of $p'$-order, and thus $(\tR,\Res_{N_{\tG}(\tR)}^{\tG}
			(\la)\tilde\varphi)$ is a principal weight of $\tG$ covering $(R,\varphi)$ by \cite[Lemma~2.15]{BS22}.
			Thus the map $\tR\mapsto\tR\cap G$ gives a surjection from the principal weight subgroups of $\tG$ onto the principal weight subgroups of $G$.
			\end{rmk}
			
			For arbitrary group $X$ and a subgroup $Y$ of $X$, we denote by $[X/Y]$ a complete set of representatives of left $Y$-cosets in $X$. 
	
			\begin{lem}\label{lem:act-wei}
				Suppose that $G\unlhd \tG$ are finite groups, and $(\tR,\tilde\varphi)$ is a $p$-weight of $\tG$. 
			 Let $A$ be a subgroup of $\Aut(\tG)$ such that $G$ is $A$-stable and $(\tR,\tilde\varphi)$ is $A$-invariant. Assume further that there exists an $A$-invariant element in $\Irr(N_G(\tR)\mid\tilde\varphi)$. 
			 Then there exists an $A$-invariant element in $\Alp(G\mid \overline{(\tR,\tilde\varphi)})$.
			
			 If moreover $A$ acts trivially on $\tG/G$, then every element of $\Alp(G\mid \overline{(\tR,\tilde\varphi)})$ is $A$-invariant. 
			\end{lem}	
			
			\begin{proof}
				Let $(R,\varphi_0)$ be a weight of $G$ covered by $(\tR,\tilde\varphi)$.
			The Dade--Glauberman--Nagao correspondence (cf. \cite[\S5]{NS14}) induces a bijection 
            \[\Phi\colon\{\varphi_0^g\mid g\in N_{\tG}(\tR)\} \to \Irr(N_G(\tR)\mid\tilde\varphi).\]
			By the equivariance of the Dade--Glauberman--Nagao correspondence, there exists an $A$-invariant element $\varphi$ of the $N_{\tG}(\tR)$-orbit containing $\varphi_0$.
			Thus $(R,\varphi)$ is an $A$-invariant weight of $G$ covered by $(\tR,\tilde\varphi)$.
			By \cite[Lemma~2.7]{BS22}, the elements in $\Alp(G\mid \overline{(\tR,\tilde\varphi)})$ form a $\tG$-orbit.
			If $A$ acts trivially on $\tG/G$, then every element of $\Alp(G\mid \overline{(\tR,\tilde\varphi)})$ is $A$-invariant. 
				\end{proof}

	\begin{lem}\label{lem:stab-cover}
		Suppose that $G\unlhd \tG$ are finite groups such that $\tG/G$ is a nilpotent group. Let $(\tR,\tilde\varphi)$ be a weight of $\tG$, and let $(R,\varphi)$ be a weight of $G$ covered by $(\tR,\tilde\varphi)$ so that $R=\tR\cap G$.  Then $N_G(R)\tR/N_G(R)$ is the (unique normal) Sylow $p$-subgroup of $N_{\tG}(R)_\varphi /N_G(R)$.
	\end{lem}
	
	\begin{proof}
	Write $N=N_G(R)$, $\tilde N=N_{\tG}(R)_\varphi$ and $M=N\tR$.
	Then $M\subset \tilde N$, and $\tR/R$ is a defect group of the unique block of $M/R$ covering the block of $N/R$ containing $\varphi$.
	Suppose that $\varphi'\in\Irr(N_{N/R}(\tR/R))$ is the Dade--Glauberman--Nagao correspondent of $\varphi$.	
	Write $K=N_{\tG}(\tR)$ and $C=N_G(\tR)\tR$.
	We regard $\varphi'$ as a character of $C/\tR$.
	Since $(\tR,\tilde\varphi)$ covers $(R,\varphi)$, we have
	$\varphi'\in\dz(C/\tR)$ and
	$\tilde\varphi\in\dz(K/\tR\mid\varphi')$.
	By Lemma~\ref{ext-rdz}, $O_p(K_{\varphi'}/C)=1$.
	By the equivariance of the Dade--Glauberman--Nagao correspondence
	(see \cite[Thm.~3.8]{Sp13} or the proof of \cite[Thm.~2.10]{BS22}),
	\(	K_{\varphi'}=K_\varphi=N_{\tilde N}(\tR).\)
	The unique block of $M/R$ covering the block of $N/R$ containing
	$\varphi$ is $N_{\tilde N}(M)$-stable.
	Since its defect groups are conjugate in $M/R$, we have
	$N_{\tilde N}(M)=M N_{\tilde N}(\tR)$.
	As $M\cap N_{\tilde N}(\tR)=C$, it follows that
	\(	N_{\tilde N}(M)/M\cong K_{\varphi'}/C.\)
	Therefore $O_p(N_{\tilde N}(M)/M)=1$.
	If $p$ divides the order of $N_{\tilde N}(M)/M$, then $O_p(N_{\tilde N}(M)/M)\ne 1$ since $\tG/G$ is nilpotent.
	So \(N_{\tilde N}(M)/M\cong N_{\tilde N/N}(M/N)/(M/N)\) is a $p'$-group, that is, $M/N$ is a Sylow $p$-subgroup of $N_{\tilde N/N}(M/N)$. However, this forces that $M/N$ is a Sylow $p$-subgroup of $\tilde N/N$.
	Hence by the nilpotency again, $M/N\unlhd \tilde N/N$, which completes the proof.
	\end{proof}
	
	\begin{cor}\label{cor:wei-cov-solquo}
	Suppose that $G\unlhd \tG$ are finite groups such that $\tG/G$ is a $p$-group. 
	Let $(\tR,\tilde\varphi)$ be a $p$-weight of $\tG$, and let $(R,\varphi)$ be a weight of $G$ covered by $(\tR,\tilde\varphi)$.
	\begin{enumerate}[\rm(1)]
	 \item $N_{\tG}(R)_\varphi=N_G(R)\tR$ and $\varphi$ extends to $N_{\tG}(R)_\varphi$.
	 \item $|\Alp(G\mid \overline{(\tR,\tilde\varphi)})|=|\tG|/|\tR G|$. Precisely, 
	 \[\Alp(G\mid \overline{(\tR,\tilde\varphi)})=\{\,\overline{(R,\varphi^x)^g}\mid x\in [N_{\tG}(R)/N_G(R)\tR],g\in [\tG/N_{\tG}(R)G]\,\}.\]
	\end{enumerate}
	\end{cor}
	
	\begin{proof}
	This follows by Lemma~\ref{lem:stab-cover} and \cite[Lemma~2.7 and Thm.~2.10]{BS22}.
	\end{proof}


\section{Radical $p$-subgroups of classical groups}\label{S:radical}  

Radical subgroups of most classical groups were classified by Alperin--Fong and An (cf. \cite{AF90,An92,An93a,An93b,An94}) 
in the early 1990s. Using our strategy, we give a uniform proof of this classification in this section. As there are errors 
in the list of radical 2-subgroups and weight 2-subgroups in \cite{An93a}, our proof here also provides some corrections 
to literatures (see Remark~\ref{rmk:c1}, Proposition~\ref{prop:weight-subgp-class-sym} and its proof, and Remark~\ref{rmk:c2}).  

When $p$ is odd and $q$ is a power of~2, the classification of radical $p$-subgroups of classical groups over the finite field 
$\mathbb F_q$ is missing from the literatures; the results of \cite{An94} are all under the assumption that $q$ is odd. 
In \S\ref{SSS:classical-odd}, our result (Theorem~\ref{subsec:odd-radical-subgp-class}) covers this case. In this way, 
we complete the classification of radical subgroups of classical groups.  

In this section, we classify radical $p$-subgroups of classical groups for all primes $p$. We will emphasize more on the case 
of $p=2$ for two reasons: first, the method of classification while $p=2$ can be adapted to treat all primes, and odd primes 
case is even easier than the prime 2 case; second, in later sections of the paper which verify the iBAW condition for groups 
of  $\tp F_{4}$, $\tp B_{n}$, $\tp D_{n}$, $^{2}\tp D_{n}$ and the prime 2, it requires detailed knowledge of 
radical 2-subgroups.      

We define a notion of parity for radical 2-subgroups of an orthogonal group, which can be used to classify radical 2-subgroups 
of spin groups. We study weight 2-subgroups and principal weight 2-subgroups of orthogonal groups. These results will be used 
in later sections while showing the iBAW condition at the prime 2 for orthogonal groups and Chevalley groups of type 
$\tp F_{4}$.

\subsection{Radical $p$-subgroups of general linear and unitary groups}\label{SS:GL2}
First we give some notation for the case $p=2$ and consider the radical 2-subgroups of general linear and unitary groups.

\noindent {\bf Notation.}
Let $D_{8}$ (resp. $Q_{8}$) be the dihedral group (resp. quaternion group) of order 8. Let $\ga\ge 1$ be an integer. 
Denote by $E_{\eta}^{2\ga+1}$ the extraspecial 2-group of order $2^{2\ga+1}$ with type $\eta\in\{+,-\}$. Note that 
$E_+^{2\ga+1}\cong D_8\circ_2 D_8\circ_2\cdots\circ_2 D_8$ is a central product of $D_8$ ($\ga$ times), while 
$E_-^{2\ga+1}\cong D_8\circ_2 D_8\circ_2\cdots\circ_2 D_8\circ_2 Q_8$ is a central product of $D_8$ ($\ga-1$ times) and 
$Q_8$ (once). Define the semidihedral group $S_{2^{\beta}}$ (with $\beta\ge4$), dihedral group $D_{2^{\beta}}$ (with $\beta\ge3$) 
and generalized quaternion group $Q_{2^{\beta}}$ (with $\beta\ge3$) by generators and relations:  
\[S_{2^{\beta}}=\langle x,y: x^{2^{\beta-1}}=1, yxy^{-1}=x^{2^{\beta-2}-1}, y^{2}=1\rangle;\] 
\[D_{2^{\beta}}=\langle x,y: x^{2^{\beta-1}}=1, yxy^{-1}=x^{-1}, y^{2}=1\rangle;\]
\[Q_{2^{\beta}}=\langle x,y: x^{2^{\beta-1}}=1, yxy^{-1}=x^{-1}, y^{2}=x^{2^{\beta-2}}\rangle.\]

Let $p=2$ and fix an odd prime power $q$.
Write \[\varepsilon=(-1)^{\frac{q-1}{2}}.\] Then, $4|(q-\varepsilon)$. Put \[a=v_{2}(q-\varepsilon)=
v_{2}(q^{2}-1)-1\geq 2.\] Many groups considered in this paper are finite groups of Lie type over $\mathbb{F}_{q}$. For this
reason, we often omit the coefficient field $\mathbb{F}_{q}$. That is to say, when no coefficient field is written, 
$\tp F_{4}$, $\Spin_{8,+}$, $\rO_{8,+}$, $\Omega_{8,+}$, etc mean $\tp F_{4}(\mathbb{F}_{q})$, $\Spin_{8,+}(\mathbb{F}_{q})$, 
$\rO_{8,+}(\mathbb{F}_{q})$, $\Omega_{8,+}(\mathbb{F}_{q})$ respectively.  

For a matrix $X\in\GL_{n}(q^{s})$, write $F_{q}(X)\in\GL_{n}(q^s)$ for the matrix whose $(i,j)$-entry is the $q$-th power of the 
$(i,j)$-entry of $X$. We know that \[\underbrace{F_{q}\circ\cdots\circ F_{q}}_{s}(X)=X,\ \textrm{for all} \ X\in\GL_{n}(q^{s}).\] Let 
$X^{tr}$ denote the transpose of a matrix $X$. Let $I_n$ denote the identity matrix of degree $n$. Let \[I_{p,q}=\left(\begin{array}{cc}-I_{p}&\\&I_{q}\\\end{array}\right),\quad  
J_{n}=\left(\begin{array}{cc}&I_{n}\\-I_{n}&\\\end{array}\right),\quad 
J'_{n}=\left(\begin{array}{cc}&I_{n}\\I_{n}&\\\end{array}\right).\]   

Let $G=\GL_n(\eps q)$ with $\eps\in\{\pm 1\}$. Here $\GL_n(-q)$ denotes the general unitary group 
\[\GU_n(q)=\{ M\in\GL_n(q^2)\mid F_q(M)=(M^{-1})^{tr}\}.\] Write $\bar{G}=G/Z(G)$ and let 
\[\pi: G\rightarrow\bar{G}\] be the projection map. 

We first classify radical 2-subgroups. Assume that $q$ is coprime to 2. Let $R$ be a radical 2-subgroup of $G$. 
Define \[A(R)=R\cap \pi^{-1}(\Omega_{1}(Z(\pi(R)))).\] It is clear that \[A(R)/(A(R)\cap Z(\GL_{n}(\epsilon q)))
=\Omega_{1}(Z(\pi(R)))\] and \[A(R)\cap Z(\GL_{n}(\epsilon q))=R\cap Z(\GL_{n}(\epsilon q))=
Z(\GL_{n}(\epsilon q))_{2}=\GL_{1}(\epsilon q)_{2}I_{n}.\] Note that $\GL_{1}(\epsilon q)_{2}$ is a cyclic 
group of order $2^{v_{2}(q-\epsilon)}$.     

\begin{lem}\label{L:m0}
An element $X$ of $R$ is contained in $A(R)$ if and only if \[YXY^{-1}=\pm{X},\ \forall Y\in R\] and 
$X^{2}\in\GL_{1}(\epsilon q)_{2}I$.   
\end{lem}

\begin{proof}
Necessity. Assume that $X\in A(R)$. Since $\pi(X)\in\Omega_{1}(Z(\pi(R)))$, we have $\pi(X)^{2}=1$ and $\pi(X)\in Z(\pi(R))$. 
Thus, $X^{2}=t'I$ for some $t'\in\bar{\mathbb{F}}_{q}^{\times}$; for any $Y\in R$, $YXY^{-1}=tX$ for some 
$t\in\bar{\mathbb{F}}_{q}^{\times}$. Since $X\in\GL_{n}(\epsilon q)$, we must have $t'\in\GL_{1}(\epsilon q)$. Moreover, as 
the order of $X$ is a power of 2, we have $t'\in\GL_{1}(\epsilon q)_{2}$. 
From \[t^{2}t'I=(tX)^{2}=(YXY^{-1})^{2}=YX^{2}Y^{-1}=t'I,\] we get $t=\pm{1}$.  

The sufficiency is clear. 
\end{proof}

By Lemma \ref{L:m0}, define a pairing $m: A(R)\times R\rightarrow\{\pm{1}\}$ by \[m(X,Y)=t,\ \forall(X,Y)\in A(R)\times R,\] 
where $YXY^{-1}X^{-1}=tI$. The following lemma is clear.  

\begin{lem}\label{L:m1}
For any $X,X'\in A(R)$ and $Y,Y'\in R$, we have \[m(XX',Y)=m(X,Y)m(X',Y)\] and 
\[m(X,YY')=m(X,Y)m(X,Y').\]   
\end{lem}

By Lemma \ref{L:m0} again, define a map $\mu: A(R)\rightarrow\{\pm{1}\}$ by 
\[\mu(X)=t'^{\frac{q-\epsilon}{2}},\ \forall X\in A(R),\] where $X^{2}=t'I$.    
 
\begin{lem}\label{L:m2}
For any $X,X'\in A(R)$, we have \[\mu(XX')=\mu(X)\mu(X')m(X,X')^{\frac{q-\epsilon}{2}}.\] 
\end{lem} 

\begin{proof}
Write $X^{2}=tI$, $X'^{2}=t'I$, $XX'X^{-1}X'^{-1}=\lambda I$, where $t,t'\in\GL_{1}(\epsilon q)_{2}$ and $\lambda=\pm{1}$. 
Then, \[(XX')^{2}=X^{2} X^{-1}(X'XX'^{-1}X^{-1})X X'^{2}=tt'\lambda I.\] Thus, \[\mu(XX')=(tt'\lambda)^{\frac{q-\epsilon}{2}}=
t^{\frac{q-\epsilon}{2}}t'^{\frac{q-\epsilon}{2}}\lambda^{\frac{q-\epsilon}{2}}=\mu(X)\mu(X')m(X,X')^{\frac{q-\epsilon}{2}}.\]
\end{proof}

Define \[A_{1}(R)=\{X\in A(R): m(X,Y)=1,\forall Y\in R\}=A(R)\cap Z(R)\] and \[A_{2}(R)=\{X\in A(R): m(X,Y)=1,\forall Y\in A\}
=Z(A(R)).\] Note that $\mu: A_{2}(R)\rightarrow\{\pm{1}\}$ is a group homomorphism. Define 
\[A'_{i}(R)=\{X\in A_{i}(R):\mu(X)=1\},\quad i=1,2.\] We also note that $m$ descends to a non-degenerate pairing 
\[m: A(R)/A_{2}(R)\times A(R)/A_{2}(R)\rightarrow\{\pm{1}\}.\]

\subsubsection{The $A'_{1}$ reduction: direct product.}

\begin{lem}\label{L:A1}
Let $R$ be a radical 2-subgroup of $\GL_{n}(\epsilon q)$. \begin{enumerate}
\item[(1)]The subgroup $A'_{1}(R)$ is diagonalizable and any element $X\in A'_{1}(R)$ is conjugate to some $tI_{p,n-p}$,  
where $t\in\GL_{1}(\epsilon q)_{2}$ and $p\in\{0,1,\dots,n\}$. 
\item[(2)]Assume that $A'_{1}=A'_{1}(R)$ is diagonal. Then, \[Z_{G}(A'_{1})=\prod_{1\leq i\leq s}\GL_{n_{i}}(\epsilon q),\] 
where $(n_{1},\dots,n_{s})$ is a partition of $n$.  
\item[(3)]$R$ is a radical 2-subgroup of $Z_{G}(A'_{1})$. Write $R_{i}$ for the intersection of $R$ with the $i$-th factor 
$\GL_{n_{i}}(\epsilon q)$. Then, \[R=\prod_{1\leq i\leq s}R_{i}\] with each $R_{i}$ ($1\leq i\leq s$) a radical 2-subgroup 
of $\GL_{n_{i}}(\epsilon q)$.  
\item[(4)]For each $i$ ($1\leq i\leq s$), we have $A'_{1}(R_{i})\subset Z(\GL_{n_{i}}(\epsilon q))$. 
\end{enumerate} 
\end{lem} 

\begin{proof}
(1)By the definition of $A'_{1}(R)$, we have $X^2=t'I$ where $t'\in\GL_{1}(\epsilon q)_{2}$. Choose $t\in\GL_{1}(\epsilon q)$ 
such that $t^2=t'$. Then, $(t^{-1}X)^{2}=I$. Thus, $t^{-1}X\sim I_{p,n-p}$ where $p\in\{0,1,\dots,n\}$. Moreover, $A'_{1}(R)$ 
is an abelian subgroup of $\GL_{n}(\epsilon q)$. Thus, it is diagonalizable.    

(2)It follows from (1) trivially. 

(3)By the definition of $A'_{1}$, it follows that $A'_{1}\subset Z(R)$ and it is stable under the adjoint action of $N_{G}(R)$. 
Then, $R$ is a radical 2-subgroup of $Z_{G}(A'_{1})$ by Lemma \ref{L:R5}. By Lemma \ref{L:R6}, each $R_{i}$ is a radical 
2-subgroup of $\GL_{n_{i}}(\epsilon q)$ and $R=\prod_{1\leq i\leq s}R_{i}$. 

(4)Suppose that $A'_{1}(R_{i})\not\subset Z(\GL_{n_{i}}(\epsilon q))$. Then, $A'_{1}(R_{i})$ contains an element conjugate to 
$I_{p,n_{i}-p}$ where $1\leq p\leq n_{i}-1$. By this, $A_{1}(R)$ contains an element conjugate to $I_{p,n-p}$. This contradicts 
 the given form of $Z_{G}(A'_{1})$ in (2).   
\end{proof}

By Lemma \ref{L:A1}, consider $R_{i}\subset\GL_{n_{i}}(\epsilon q)$ ($1\leq i\leq s$) instead. It reduces to classify radical 
2-subgroups of $\GL_{n}(\epsilon q)$ such that $A'_{1}(R)\subset Z(\GL_{n}(\epsilon q))$.

\subsubsection{The $A'_{2}$ reduction: wreath product} 

Let \[Z_{G}(A'_{2}/A'_{1})=Z_{G}(A'_{1})\cap\{Y\in\GL_{n}(\epsilon q): YXY^{-1}=\pm{X},\forall X\in A'_{2}(R)\}.\]   

\begin{lem}\label{L:A2}
Let $R$ be a radical 2-subgroup of $\GL_{n}(\epsilon q)$ such that $A'_{1}(R)\subset Z(\GL_{n}(\epsilon q))$. 
\begin{enumerate}
\item[(1)]The subgroup $A'_{2}(R)$ is diagonalizable and any nontrivial coset in $A'_{2}(R)/A'_{1}(R)$ contains an element conjugate 
to $I_{\frac{n}{2},\frac{n}{2}}$.  
\item[(2)]Assume that $A'_{2}(R)$ is diagonal and write $c=\log_{2}|A'_{2}/A'_{1}|$. Then,  
\[Z_{G}(A'_{2}/A'_{1})=\GL_{n/2^{c}}(\epsilon q)^{2^{c}}\rtimes(Z_{2})^{c},\] where 
$\GL_{n/2^{c}}(\epsilon q)^{2^{c}}=Z_{G}(A'_{2})$ and $(Z_{2})^{c}$ acts on $\GL_{n/2^{c}}(\epsilon q)^{2^{c}}$ by 
permutations.  
\item[(3)]$R$ is a radical 2-subgroup of $Z_{G}(A'_{2}/A'_{1})$ and $R$ is conjugate to some 
\[R'\wr (Z_{2})^{c}=R'^{2^{c}}\rtimes(Z_{2})^{c},\] where $R'$ is a radical 2-subgroup of $\GL_{n/2^{c}}(\epsilon q)$.  
\item[(4)]We have $A'_{1}(R')\subset Z(\GL_{n/2^{c}}(\epsilon q))$.   
\end{enumerate}  
\end{lem}

\begin{proof}
(1)Any nontrivial coset in $A'_{2}/A'_{1}$ contains an element $X$ with $X^{2}=I$ and $YXY^{-1}=-X$ for some $Y\in R$. Then, 
$X\sim I_{\frac{n}{2},\frac{n}{2}}$. By the definition of $m$ and $A'_{2}$, the subgroup $A'_{2}$ is abelian. Thus, 
$A'_{2}$ is diagonalizable. 

(2)It follows from (1) trivially. 

(3)Since $N_{G}(R)$ normalizes $A'_{1}$ and $A'_{2}$, then it normalizes $Z_{G}(A'_{2}/A'_{1})$. By Lemma \ref{L:R8}, $R$ 
is a radical 2-subgroup of $Z_{G}(A'_{2}/A'_{1})$. For any element $X\in A'_{2}-A'_{1}$, there must exist $Y\in R$ such 
that $m(X,Y)=-1$. Then, $R$ intersects with each coset of $Z_{G}(A'_{2})$ in $Z_{G}(A'_{2}/A'_{1})$. By Lemma \ref{L:R7}, 
$R$ is conjugate to $R'\wr (Z_{2})^{c}$, where $R'$ is a radical 2-subgroup of $\GL_{n/2^{c}}(\epsilon q)$.  

(4)It is clear that $\Delta(A'_{1}(R'))\subset A'_{1}(R)$, where $\Delta:\GL_{n/2^{c}}(\epsilon q)\rightarrow
\GL_{n}(\epsilon q)$ is the diagonal map. Since it is assumed that $A'_{1}(R)\subset Z(\GL_{n}(\epsilon q))$, 
we get $A'_{1}(R')\subset Z(\GL_{n/2^{c}}(\epsilon q))$.  
\end{proof}

However, it is not necessarily that $A'_{2}(R')\subset Z(\GL_{n/2^{c}}(\epsilon q))$ in the above conclusion. We need to 
take a successive process of wreath products $R_{0}\wr(Z_{2})^{c_{t}}\wr\cdots\wr(Z_{2})^{c_{1}}$ so as to reach a radical 
2-subgroup $R_{0}$ of $\GL_{n/2^{c_{1}+\cdots+c_{t}}}(\epsilon q)$ with 
$A'_{2}(R_{0})\subset Z(\GL_{n/2^{c_{1}+\cdots+c_{t}}}(\epsilon q))$. By Lemma \ref{L:A2}, it reduces to classify radical 
2-subgroups of $\GL_{n}(\epsilon q)$ such that $A'_{2}(R)\subset Z(\GL_{n}(\epsilon q))$.

\subsubsection{Radical 2-subgroups with trivial $A'_{2}$} 

Suppose that $A'_{1}(R)=A'_{2}(R)=\GL_{1}(\epsilon q)_{2}I$. Then, there are three possibilities for $A_{1}=A_{1}(R)$ and 
$A_{2}=A_{2}(R)$: \begin{enumerate}
\item[(1)]$A_{1}/A'_{1}=A_{2}/A'_{2}=1$;  
\item[(2)]$A_{1}/A'_{1}=1$ and $A_{2}/A'_{2}\cong Z_{2}$; 
\item[(3)]$A_{1}/A'_{1}=A_{2}/A'_{2}\cong Z_{2}$. 
\end{enumerate}    

\begin{lem}\label{L:GL1}
Let $n$ be an even positive integer and $\delta\in\GL_{1}(\epsilon q)$ be a given element. Let $X,Y\in\GL_{n}(\epsilon q)$ 
be two matrices such that $XYX^{-1}Y^{-1}=-I$, $X^{2}=I$, $Y^{2}=\delta I$. 

(1) We have \[(X,Y)\sim(\left(\begin{array}{cc}-I_{n/2}&\\&I_{n/2}\end{array}\right),\left(\begin{array}{cc}&
I_{n/2}\\\delta I_{n/2}&\end{array}\right)).\] 

(2) We have \[Z_{\GL_{n}(\epsilon q)}(\{X,Y\})\cong\GL_{n/2}(\epsilon q).\]
\end{lem} 

\begin{proof}
(1) Since $X^{2}=I$, then $X\sim\left(\begin{array}{cc}-I_{p}&\\&I_{n-p}\end{array}\right)$ for some $p\in\{0,1,\dots,n\}$. Since 
$-X=YXY^{-1}\sim X$, we must have $p=\frac{n}{2}$. Assume that $X=\left(\begin{array}{cc}-I_{p}&\\&I_{n-p}\end{array}\right)$ 
for simplicity. Since $YXY^{-1}=-X$, we have \[Y=\left(\begin{array}{cc}&Y_{1}\\Y_{2}&\end{array}\right),\] where 
$Y_{1},Y_{2}\in\GL_{n/2}(\epsilon q)$. Taking conjugation if necessary, we may assume that $Y_{1}=I$. Then, 
$\delta I=Y^{2}=\diag\{Y_{2},Y_{2}\}$. Thus, $Y_{2}=\delta I_{n/2}$.  

(2) This is clear. 
\end{proof}

\begin{lem}\label{lem:square-sum}
	Let $q'$ be a prime power, $\eps\in\{\pm1\}$ and $\la\in\mathbb F_{q'}$. Then there exist elements $b,b'\in\mathbb{F}_{{q'}^2}$ such that 
	$b^{2}+b'^{2}=\la$, ${b}^{q'}=\eps b$ and ${b'}^{q'}=\eps b'$. 
	\end{lem}
		
	\begin{proof}
	We assume that $\eps=-1$ and $q'$ is odd, as it is well-known when $\eps=1$ or $q'$ is even. Let $Q=\{x^2\mid x\in \mathbb{F}_{{q'}^2}, x^{q'}=-x\}$. Then $Q\subseteq\mathbb F_{q'}$, 
	and for any $y\in\mathbb F_{q'}$ one has that $y\in Q$ if and only if $y=0$ or $y$ is a non-square. Thus $|Q|=\frac{q'+1}{2}$, and 
	so $Q$ and $\{\la-y\mid y\in Q\}$ intersect. This clearly forces that there exist elements $b,b'\in\mathbb{F}_{{q'}^2}$ such that 
	$b^{2}+b'^{2}=\la$, ${b}^{q'}=-b$ and ${b'}^{q'}=-b'$.
	\end{proof}

Choose $b,b'\in\mathbb{F}_{q^2}$ such that $b^{2}+b'^{2}=-1$, ${b}^q=\eps b$, ${b'}^q=\eps b'$ 
and $b\neq 0$. In particular, if $\eps=1$, then $b,b'\in\mathbb{F}_{q}$.	

\begin{lem}\label{L:GL2}
	Suppose that $4|q+\epsilon$. Let $n$ be an even positive integer. Let $X,Y\in\GL_{n}(\epsilon q)$ be two matrices such that 
$XYX^{-1}Y^{-1}=-I$, $X^{2}=Y^{2}=-I$.  

(1)We have \[(X,Y)\sim(\left(\begin{array}{cc}&I_{n/2}\\-I_{n/2}&\end{array}\right),\left(\begin{array}{cc}bI_{n/2}&
b'I_{n/2}\\b'I_{n/2}&-bI_{n/2}\end{array}\right)).\] 

(2)We have \[Z_{\GL_{n}(\epsilon q)}(\{X,Y\})\cong\GL_{n/2}(\epsilon q).\] 
\end{lem} 

\begin{proof}
Write \[X_{1}=\left(\begin{array}{cc}&I_{n/2}\\-I_{n/2}&\end{array}\right)\quad\textrm{ and }\quad Y_{1}=
\left(\begin{array}{cc}bI_{n/2}&b'I_{n/2}\\b'I_{n/2}&-bI_{n/2}\end{array}\right).\] Then, $X_{1},Y_{1}\in\GL_{n}(\epsilon q)$,  
$X_{1}Y_{1}X_{1}^{-1}Y_{1}^{-1}=-I$ and $X_{1}^{2}=Y_{1}^{2}=-I$. Embed $\GL_{n}(\epsilon q)$ into $\GL_{n}(q^{2})$. Choose 
$\lambda\in\mathbb{F}_{q^{2}}$ with $\lambda^{2}=-1$. Then, $(\lambda X)^{2}=(\lambda Y)^{2}=I$. By Lemma \ref{L:GL1}, there 
exists a unique conjugacy class of such pairs $(\lambda X,\lambda Y)$ in $\GL_{n}(q^{2})$, 
and hence the same is true for such pairs $(X,Y)$.
Then, there exists $g\in\GL_{n}(q^{2})$ such that $(X,Y)=(gX_{1}g^{-1},gY_{1}g^{-1})$. Write $\sigma$ for  
the Frobenius automorphism on $\GL_{n}(q^{2})$ with respect to the field extension $\mathbb{F}_{q^{2}}/\mathbb{F}_{q}$. 
Define \[\sigma'(X)=(\sigma(X)^{t})^{-1},\ \forall X\in\GL_{n}(q^{2}).\] Then, $\GL_{n}(q)$ (resp. $\GL_{n}(-q)$) is the 
fixed-points subgroup of $\sigma$ (resp. $\sigma'$). Write $\tau=\sigma$ (or $\sigma'$) for $\GL_{n}(q)$ (or $\GL_{n}(-q)$). 
Then, $(X,Y)=(gX_{1}g^{-1},gY_{1}g^{-1})\subset\GL_{n}(\epsilon q)$ is equivalent to: 
\[g^{-1}\tau(g)\in Z_{\GL_{n}(q^{2})}(\{X_{1},Y_{1}\})=\Delta(\GL_{n/2}(q^{2})).\]
By the Lang--Steinberg theorem 
(\cite{Lang56},\cite{Steinberg68}, see also \cite[Thm.~21.7]{MT11}), there exists $h\in\Delta(\GL_{n/2}(\bar{\mathbb{F}}_{q}))$ such that 
\[g^{-1}\tau(g)=h^{-1}\tau(h).\] Put $g'=gh^{-1}$. Then, $g'\in\GL_{n}(\epsilon q)$ and 
\[(X,Y)=(gX_{1}g^{-1},gY_{1}g^{-1})=(g'X_{1}g'^{-1},g'Y_{1}g'^{-1}).\] This shows the conclusion of this lemma. 
\end{proof} 

\begin{prop}\label{P:GL3}
Let $R$ be a radical 2-subgroup of $\GL_{n}(\epsilon q)$ with \[A_{1}(R)=A_{2}(R)=\GL_{1}(\epsilon q)_{2}I\] and write $\gamma=\frac{1}{2}\log_{2}|A(R)/A_{2}(R)|$. 
Then $R=A(R)$ and $2^\gamma\mid n$. 
If $(\frac{n}{2^\gamma},q)\ne(2,3)$, then
\begin{itemize}
\item[(i)] if $4\mid q-\epsilon$ or $\gamma\leq1$, then there exists a unique
conjugacy class of such radical 2-subgroups for given values of $n$ and $\gamma$;
\item[(ii)] if $4\mid q+\epsilon$ and $\gamma\geq2$, then there are two
conjugacy classes of such radical 2-subgroups for given values of $n$ and $\gamma$.
\end{itemize}
If $(\frac{n}{2^\gamma},q)=(2,3)$, there is no such radical 2-subgroup.
\end{prop}

\begin{proof}
Assume that $\gamma=0$. Then, $A(R)=A_{1}(R)=\GL_{1}(\epsilon q)_{2}I$ and $R/\GL_{1}(\epsilon q)_{2}I$ is a radical 
2-subgroup of $\PGL_{n}(\epsilon q)$ with $\Omega_{1}(R/\GL_{1}(\epsilon q)_{2}I)=A(R)/\GL_{1}(\epsilon q)_{2}I=1$. Thus, 
$R=\GL_{1}(\epsilon q)_{2}I$.     

Assume that $4|q-\epsilon$ and $\gamma\geq 1$. By Lemma \ref{L:m2}, the map $\mu: A(R)\rightarrow\{\pm{1}\}$ is a homomorphism. 
Using Lemmas \ref{L:m1} and \ref{L:m2}, one finds elements $x_{1},x_{2},\dots,x_{2\gamma-1},x_{2\gamma}\in A(R)$ generating 
$A(R)/A_{2}(R)$ and such that: \[m(x_{2i-1},x_{2j-1})=m(x_{2i-1},x_{2j})=m(x_{2i},x_{2j-1})=m(x_{2i},x_{2j})=1,\ i\neq j,\] 
\[m(x_{2i-1},x_{2i})=-1,\ 1\leq i\leq\gamma,\] \[\mu(x_{2\gamma})=\mu(x_{2\gamma-1})=\cdots=\mu(x_{2})=1,\ \delta:=\mu(x_{1})=
\pm{1}.\] 
Take a generator $\delta_0$ of $\GL_1(\epsilon q)_2$.
Moreover, we could assume that
\[
x_i^2=I\quad(2\leq i\leq2\gamma),\quad
x_1^2=\left\{
\begin{array}{ll}
I,&\textrm{if }\delta=1,\\
\delta_0I,&\textrm{if }\delta=-1.
\end{array}\right.
\]
By Lemma \ref{L:GL1}, the conjugacy class of each
$(x_{2i-1},x_{2i})$ ($2\leq i\leq\gamma$) is determined,
and the conjugacy class of $(x_2,x_1)$ is determined by $\delta$.
Moreover, we always have \[Z_{\GL_{n}(\epsilon q)}(\{x_{2i-1},x_{2i}\})\sim
\Delta(\GL_{n/2}(\epsilon q)).\] By induction, one shows that $2^{\gamma}|n$ and there is a unique conjugacy class of $A(R)$ 
for fixed values of $n$, $\gamma$, $\delta$. Moreover, we have \[Z_{G}(A(R)/A_{1}(R))=G'\cdot A(R),\] where 
$G'\cong\GL_{n/2^{\gamma}}(\epsilon q)$. Put $R'=R\cap G'$. By Lemmas \ref{L:R5} and \ref{L:R4}, $R'$ is a radical 2-subgroup 
of $G'$. It is easy to show that $A(R')=\GL_{1}(\epsilon q)_{2}I$. Then, one shows $R'=A(R')=\GL_{1}(\epsilon q)_{2}I$ by the 
same argument as in the first paragraph. Thus, $R=A(R)$. When $\delta=-1$, put $\B=\{x\in A(R):\mu(x)=1\}$. Then, $B$ is an 
index 2 subgroup of $A(R)$ normalized by $N_{G}(R)$. 
Put
\[
S=\{g\in Z(Z_G(Z(B))):g^4=I\}.
\]
Then $RS$ is a normal $2$-subgroup of $N_G(R)$ containing $R$,
and $RS\ne R$.
Thus $O_2(N_G(R))\ne R$, and hence $R$ is not a radical $2$-subgroup.
The $\delta=1$ case gives a radical 2-subgroup if $(n/2^\gamma,q)\ne(2,3)$.

Assume that $4|q+\epsilon$ and $\gamma\geq 1$. Using Lemmas \ref{L:m1} and \ref{L:m2}, one finds elements $x_{1},x_{2},\dots,
x_{2\gamma-1},x_{2\gamma}\in A(R)$ generating $A(R)/A_{2}(R)$ and such that: \[m(x_{2i-1},x_{2j-1})=m(x_{2i-1},x_{2j})=
m(x_{2i},x_{2j-1})=m(x_{2i},x_{2j})=1,\ i\neq j,\] \[m(x_{2i-1},x_{2i})=-1,\ 1\leq i\leq\gamma,\] \[\mu(x_{2\gamma})=
\mu(x_{2\gamma-1})=\cdots=\mu(x_{3})=1,\ \delta:=\mu(x_{2})=\mu(x_{1})=\pm{1}.\] By Lemmas \ref{L:GL1} and \ref{L:GL2}, 
one shows that $2^{\gamma}|n$ and there is a unique conjugacy class of $A(R)$ for fixed values of $n$, $\gamma$, $\delta$.   
The rest of the proof is similar to the $4|q-\epsilon$ case. 
When $\gamma=1$ and $\delta=1$, we have $R\cong D_8$.
Let $K$ be the unique cyclic subgroup of order $4$ in $R$.
Let $S$ be the unique subgroup of order $8$ in $Z(Z_G(K))$.
Then $RS$ is a normal $2$-subgroup of $N_G(R)$ containing $R$
with $RS/R\cong Z_2$.
Thus $O_2(N_G(R))\ne R$, and hence $R$ is not a radical $2$-subgroup.
When $\gamma\geq2$ or $\delta=-1$, the resulting subgroup is radical
if $(n/2^\gamma,q)\ne(2,3)$.
\end{proof}

Let \[Z_{G}(A/A'_{1})=Z_{G}(A'_{1})\cap\{Y\in\GL_{n}(\epsilon q): YXY^{-1}=\pm{X},\forall X\in A(R)\}.\]   

\begin{prop}\label{P:GL4}
Let $R$ be a radical 2-subgroup of $\GL_{n}(\epsilon q)$ with \[A_{1}(R)=A'_{1}(R)=A'_{2}(R)=\GL_{1}(\epsilon q)_{2}I
\quad\textrm{ and }\quad A_{2}(R)/A'_{2}(R)\cong Z_{2}\] and write $\gamma=\frac{1}{2}\log_{2}|A(R)/A_{2}(R)|$. 
Then $4\mid q+\epsilon$ and $2^{\gamma+1}\mid n$.
If $(\frac{n}{2^{\gamma+1}},q)\ne(2,3)$, then there is a unique conjugacy class of such $R$
for fixed values of $n$ and $\gamma$.
If $(\frac{n}{2^{\gamma+1}},q)=(2,3)$, there is no such radical 2-subgroup.  
\end{prop} 

\begin{proof}
By assumption, there exists an element $x_{0}\in A_{2}(R)$ with $\mu(x_{0})=-1$. Using Lemmas \ref{L:m1} and \ref{L:m2}, 
one finds elements $x_{1},x_{2},\dots,x_{2\gamma-1},x_{2\gamma}\in A(R)$ generating $A(R)/A_{2}(R)$ and such that: \[m(x_{2i-1},x_{2j-1})=m(x_{2i-1},x_{2j})=m(x_{2i},x_{2j-1})=m(x_{2i},x_{2j})=1,\ i\neq j,\] 
\[m(x_{2i-1},x_{2i})=-1,\ 1\leq i\leq\gamma,\] \[\mu(x_{2\gamma})=\mu(x_{2\gamma-1})=\cdots=\mu(x_{2})=\mu(x_{1})=1.\] We must 
have \[x_{0}\sim\left(\begin{array}{cc}&I_{n/2}\\\delta_{0}I_{n/2}&\end{array}\right),\] where $\delta_{0}$ is a generator of 
$\GL_{1}(\epsilon q)_{2}$. By Lemma \ref{L:GL1}, one shows that $2^{\gamma+1}|n$ and there is a unique conjugacy classes of 
$A(R)$ for fixed values of $n$ and $\gamma$. Moreover, calculating common centralizers of 
$x_{1},x_{2},\dots,x_{2\gamma-1},x_{2\gamma}$ inductively we get 
\[Z_{\GL_{n}(\epsilon q)}(A(R)/A'_{1}(R))=G'\cdot D,\] where 
\[D=\langle x_{1},x_{2},\dots,x_{2\gamma-1},x_{2\gamma}\rangle\subset A(R)\] and it generates $A(R)/A_{2}(R)$, 
$D\cap A_{2}(R)=\{\pm{I}\}$ and $D/\{\pm{I}\}\cong(Z_{2})^{2\gamma}$. Moreover,
 \[G'\cong\GL_{n/2^{\gamma+1}}(q^{2})\rtimes\langle\tau\rangle\] where $\tau^{2}=1$ and  
\[\tau X\tau^{-1}=\left\{\begin{array}{cc}F_{q}(X)\quad\textrm{ when }\epsilon=1;\\(F_{q}(X)^{t})^{-1}\textrm{ when }
\epsilon=-1.\end{array}\right.\] Put $R''=R\cap G'$. By Lemmas \ref{L:R5} and \ref{L:R4}, $R''$ is a radical 2-subgroup of 
$G'$. Put $R'=R''\cap\GL_{n/2^{\gamma+1}}(q^{2})$. By Lemma \ref{L:R3}, $R'$ is a radical 2-subgroup of 
$\GL_{n/2^{\gamma+1}}(q^{2})$. Considering the action of an element $x\in R''-R'$ on $A(R')$, one shows that $A(R')=A_{2}(R')$ 
and $A_{1}(R')=A'_{1}(R')=A'_{2}(R')=\GL_{1}(q^{2})_2I$. Put $R_{0}=R$, $R_{1}=R'$. Taking induction, we get a sequence of  
radical 2-subgroups $R_{0},R_{1},\dots$. There must exist $k\geq 1$ such that $A_{1}(R_{k-1})/A'_{1}(R_{k-1})=1$, 
$A_{2}(R_{k-1})/A'_{2}(R_{k-1})\cong Z_{2}$, and $A_{1}(R_{k})/A'_{1}(R_{k})=A_{2}(R_{k})/A'_{2}(R_{k})=1$.  

First, suppose that $k=1$ and $\gamma=0$ for simplicity. We must have $R'=\GL_{1}(q^{2})_{2}I$ and $R/R'\cong Z_{2}$. When 
$4|q+\epsilon$, we have: $R=\langle x,y\rangle$ and $R'=\langle x\rangle$, where $x^{2^{a}}=-I$, $yxy^{-1}=x^{\epsilon q}=
x^{2^{a}-1}$, $y^{2}=I$. Embed $\GL_{n/2}(q^{2})\rtimes\langle\tau\rangle$ into $\GL_{n}(\epsilon q)$ and identify $x$ with 
some matrix $\left(\begin{array}{cc}aI&bI\\-bI&aI\end{array}\right)$. Using Lemma \ref{L:GL1} one shows that $y\sim\tau$. 
Such an $R$ is conjugate to some $\Delta(\tilde{R})$, where $\tilde{R}$ is a Sylow 
2-subgroup of $\GL_{2}(\epsilon q)$. 
It is radical if $(n/2,q)\ne(2,3)$.
When $4|q-\epsilon$, we have: $R=\langle x,y\rangle$ 
and $R'=\langle x\rangle$, where $x^{2^{a}}=-I$, $yxy^{-1}=x^{\epsilon q}=x^{2^{a}+1}$, $y^{2}=I$. Again, we have $y\sim\tau$.  
One shows that $N_{G}(R)/Z_{G}(R)R\cong Z_{2}$ and $O_{2}(N_{G}(R)/R)\cong Z_{2}$. Hence, $R$ is not a 2-radical subgroup.  

In general, note that we always have $4|q^{2}-1$. By considering $R_{k-1}$ and $R_{k}$, the above argument shows that 
$4|q+\epsilon$ and $k=1$. 
Moreover, if $(n/2^{\gamma+1},q)\ne(2,3)$, then the above argument leads
to a unique conjugacy class of radical 2-subgroups with a representative
$R=(\Delta(\tilde{R}))\cdot D$, where $\tilde{R}$ is a Sylow 2-subgroup
of $\GL_2(\epsilon q)$ and $D$ is as in the first paragraph.
\end{proof} 

\begin{prop}\label{P:GL5}
Let $R$ be a radical 2-subgroup of $\GL_{n}(\epsilon q)$ such that \[A'_{1}(R)=A'_{2}(R)=
\GL_{1}(\epsilon q)_{2}I\quad\textrm{ and }\quad A_{1}(R)/A'_{1}(R)=A_{2}(R)/A'_{2}(R)\cong Z_{2}.\] Write 
$\gamma=\frac{1}{2}\log_{2}|A(R)/A_{2}(R)|$ and $|R|=2^{2\gamma+\alpha+a}$. 
Then $\alpha\geq1$ and $2^{\gamma+\alpha}\mid n$.
If $n=2^{\alpha+\gamma}$, $\alpha=1$ and $q+\epsilon$ is a power of $2$,
then there is no such radical 2-subgroup.
Otherwise there is a unique conjugacy class of such radical 2-subgroups
for given values of $n$, $\gamma$ and $\alpha$.

\end{prop} 

\begin{proof}
By assumption, there exists an element $x_{0}\in A_{1}(R)$ with $\mu(x_{0})=-1$. Using Lemmas \ref{L:m1} and \ref{L:m2}, one 
finds elements $x_{1},x_{2},\dots,x_{2\gamma-1},x_{2\gamma}\in A(R)$ generating $A(R)/A_{1}(R)$ and such that:  \[m(x_{2i-1},x_{2j-1})=m(x_{2i-1},x_{2j})=m(x_{2i},x_{2j-1})=m(x_{2i},x_{2j})=1,\ i\neq j,\] \[m(x_{2i-1},x_{2i})=-1,
\ 1\leq i\leq\gamma,\] \[\mu(x_{2\gamma})=\mu(x_{2\gamma-1})=\cdots=\mu(x_{2})=\mu(x_{1})=1.\] We must have 
\[x_{0}\sim\left(\begin{array}{cc}&I_{n/2}\\\delta_{0}I_{n/2}&\end{array}\right),\] where $\delta_{0}$ is a generator of 
$\GL_{1}(\epsilon q)_{2}$. By Lemma \ref{L:GL1}, one shows that $2^{\gamma+1}|n$ and there is a unique conjugacy classes of 
$A(R)$ for fixed values of $n$ and $\gamma$. Moreover, we have \[Z_{\GL_{n}(\epsilon q)}(A(R)/A_{1}(R))=G'\cdot A(R),\] 
where $G'\cong\GL_{n/2^{\gamma+1}}(q^{2})$. Put $R'=R\cap G'$. By Lemmas \ref{L:R5} and \ref{L:R4}, $R'$ is a radical 
2-subgroup of $G'$. One can show easily that $A(R')=A_{2}(R')$ and $A'_{1}(R')=A'_{2}(R')=\GL_{1}(q^{2})_{2}I$. Then, \[(A_{1}(R')/A'_{1}(R'),A_{2}(R')/A'_{2}(R'))\cong(1,1), (1,Z_{2})\textrm{ or }(Z_{2},Z_{2}).\] In the $(1,1)$ case, 
it reduces to Proposition \ref{P:GL3}. Due to $4|q^{2}-1$, by Proposition \ref{P:GL4} the $(1,Z_{2})$ case can not happen. 
In the $(Z_{2},Z_{2})$ case, one takes further induction. 
Finally, we obtain a unique conjugacy class of radical 2-subgroups
for given values of $n$, $\gamma$ and $\alpha$, except when
$n=2^{\gamma+1}$, $\alpha=1$ and $q+\epsilon$ is a power of $2$.
\end{proof} 

Now we summarize the above construction of radical 2-subgroups of general linear and unitary groups.
First, we give embeddings of the groups $D_8$ and $Q_8$ into $\GL_2(\eps q)$ as follows. The group $D_8$ can be embedded into 
$\GL_2(\eps q)$, generated by $\begin{pmatrix} -1 & 0 \\ 0 & 1 \end{pmatrix}$ and $\begin{pmatrix} 0 & 1 \\ 1 & 0 \end{pmatrix}$, 
while the group $Q_8$ can be embedded into $\SL_2(\eps q)$, generated by $\begin{pmatrix} 0 & 1 \\ -1 & 0 \end{pmatrix}$ and 
$\begin{pmatrix} b & b' \\ b' & -b \end{pmatrix}$. By taking tensor products, the extraspecial 2-group $E_{\eta}^{2\ga+1}$ can be embedded into 
$\GL_{2^\ga}(\eps q)$ such that its center is embedded into $Z(\GL_{2^\ga}(\eps q))$. For convenience, we also allow $\ga=0$. 
In this case, we interpret $E_\eta^{2\ga+1}$ as the group of order 2.   

Let $\al\in\bbZ_{\ge 0}$. First suppose that $4\mid (q-\eps)$ or $\al\ge 1$.
Let $Z_\al$ be a cyclic group of order $2^{a+\al}$.
Then $Z_\al\circ_2 E_{+}^{2\ga+1}\cong  Z_\al\circ_2 E_{-}^{2\ga+1}$. Thus we assume that $\eta=+$.
The group $Z_\al\circ_2 E_{+}^{2\ga+1}$ can be embedded into $\GL_{2^{\ga}}((\eps q)^{2^\al})$ such that $Z_\al$ is identified
with $Z(\GL_{2^{\ga}}((\eps q)^{2^\al}))_2$. Let $R_{\al,\ga}^1$ be the image of $Z_\al\circ_2 E_{+}^{2\ga+1}$ under the 
embedding
\begin{equation}\label{equ:embedding-cox}
\GL_{2^{\ga}}((\eps q)^{2^\al})\embed\GL_{2^{\al+\ga}}(\eps q).
\addtocounter{thm}{1}\tag{\thethm}
\end{equation}

Now suppose that $4\mid (q+\eps)$ and $\al=0$. \begin{itemize}
\item[(i)] Set $R_{0,0}^1=\{\pm 1\}$, the Sylow 2-subgroup of $\GL_1(\eps q)$. 
\item[(ii)] Write $R_{0,\ga}^1$ for the image of $E_{\eta}^{2\ga+1}$ in $\GL_{2^\ga}(\eps q)$. Here we let $\ga\ge2$ for $\eta=+$ 
while let $\ga\ge1$ for $\eta=-$. Note that here $R_{0,\ga}^1$  stands for two different groups since $\eta=+,-$. Sometimes we also write 
it as $R_{0,\ga}^{1,\eta}$ to indicate $\eta$.
\item[(iii)] We consider the embedding $\GL_2(\eps q)\embed \GL_{2^{\ga}}(\eps q)$, $M\mapsto  I_{2^{\ga-1}}\otimes M$, and denote 
the image of a Sylow 2-subgroup $\GL_2(\eps q)$ under this embedding by $S_{\ga}$ (then $S_{\ga}$ is a semidihedral group of 
order $2^{a+2}$). Here we let $\ga\ge 1$. From this, $S_{\ga}\circ_2 E_{+}^{2\ga-1}\cong S_{\ga}\circ_2 E_{-}^{2\ga-1}$ can be 
embedded into  $\GL_{2^{\ga}}(\eps q)$, and its image is denoted by $R_{0,\ga}^{2}$. 
\end{itemize}

As above, we have defined $R_{\al,\ga}^i$ with $i\in\{1,2\}$. Note that $R_{\al,\ga}^2$ is defined only when $4\mid(q+\eps)$ and 
$\al=0$. Let $m$ be a positive integer. We consider the $m$-fold diagonal embedding 
\begin{equation}\label{equ:m-fold}
\GL_{2^{\al+\ga}}(\eps q)\embed \GL_{m2^{\al+\ga}}(\eps q),\quad 
g\mapsto  \diag(g,g,\ldots,g),
\addtocounter{thm}{1}\tag{\thethm}
\end{equation}
and write $R_{m,\al,\ga}^i$ for the image of  $R_{\al,\ga}^i$ under this embedding. 

With above notation, radical 2-subgroups obtained in Proposition \ref{P:GL3} are 
as follows ($m=\frac{n}{2^{\gamma}}$ below): \begin{itemize} 
\item[(1)]when $\ga=0$, they are $R_{m,0,0}^{1}$;
\item[(2)]when $4|q-\epsilon$ and $\gamma\geq 1$, they are $R_{m,0,\gamma}^{1}$;
\item[(3)]when $4|q+\epsilon$ and $\gamma=1$, they are $R_{m,0,\gamma}^{1,-}$;
\item[(4)]when $4|q+\epsilon$ and $\gamma\geq 2$, they are $R_{m,0,\gamma}^{1,\pm{}}$. 
\end{itemize}
Radical 2-subgroup obtained in Proposition \ref{P:GL4} 
are $R_{m,0,\gamma+1}^{2}$, where $m=\frac{n}{2^{\gamma+1}}$. 
Radical 2-subgroup obtained in Proposition \ref{P:GL5} 
are $R_{m,\alpha,\gamma}^{1}$, where $\alpha\geq 1$ and $m=\frac{n}{2^{\gamma+\alpha}}$.

\begin{rem}
Let $R=R^1_{m,\alpha,\gamma}$ be the subgroup constructed in the proof
of Proposition~\ref{P:GL5}, where $m=n/2^{\alpha+\gamma}$ and $\alpha\ge1$.
Assume that the characteristic subgroups of $R$ are as in that proposition.
Then $R$ is radical except when $m=1$, $\alpha=1$ and $q+\epsilon$
is a power of $2$.
In this case there is no radical subgroup with the given characteristic
subgroups.
For direct products and wreath products, we must also have
$O_2(N_G(R))=R$ in the corresponding group $G$.
\end{rem}

Let $c\ge 0$ be an integer. Denote by $A_c$ the elementary abelian 2-group of order $2^c$ in its left regular representation.
For a sequence of positive integers $\bc=(c_t,\ldots,c_1)$, we define $A_{\bc}=A_{c_1}\wr A_{c_2}\wr\cdots \wr A_{c_t}$. Then 
$A_{\bc}\le\fS_{2^{|\bc|}}$, where  $|\bc|:=c_1+c_2+\cdots+c_t$. For convenience, we also allow $\bc$ to be $\emptyset$.  

Put $R_{m,\al,\ga,\bc}^i=R_{m,\al,\ga}^i\wr A_{\bc}$. We say that these groups $R_{m,\al,\ga,\bc}^i$ are \emph{basic subgroups} 
of $\GL_{m 2^{\al+\ga+|\bc|}}(\eps q)$ except when $4\mid(q+\eps)$, $i=1$, $\al=\ga=0$ and $c_1=1$.

\begin{rmk}
\begin{enumerate}[\rm(1)]
    \item The basic subgroup $R^i_{m,\al,\ga,\bc}$ is radical except when
$(m,\al,q)=(2,0,3)$, or when $m=\al=1$ and $q+\eps$ is a power of $2$.
As in the definition, if $4\mid(q+\eps)$, $i=1$, $\al=\ga=0$,
and $A_{\bc}\ne1$, then $c_1\geq2$.
\item Two basic subgroups $R^i_{m,\al,\ga,\bc}$ and
$R^{i'}_{m',\al',\ga',\bc'}$ are conjugate if and only if
$(i,m,\al,\ga,\bc)=(i',m',\al',\ga',\bc')$
and, when $4\mid(q+\eps)$, $i=1$, $\al=0$ and $\ga\geq1$,
their types $\eta$ are equal.
In particular, if $4\mid(q+\eps)$ and $\ga\geq2$, then
$R^{1,+}_{m,0,\ga,\bc}$ and $R^{1,-}_{m,0,\ga,\bc}$ are not conjugate.
\end{enumerate}
\end{rmk}

\begin{thm}\label{SS:2-rad-GL}
Let $G=\GL_n(\eps q)=\GL(V)$ or $\GU(V)$ with odd $q$. If $R$ is a radical 2-subgroup of $G$, then 
there exists a corresponding decomposition  
\begin{align*}
V&=V_1\oplus \cdots\oplus V_u \ \textrm{(for $\eps=1$) or} \ V=V_1\perp \cdots\perp V_u \ \textrm{(for $\eps=-1$)},\\
R&=R_1\ti\cdots\ti R_u,
\end{align*}
 where $R_i$ ($1\le i\le u$) are basic subgroups of $G_i$, for $G_i=\GL(V_i)$ or $\GU(V_i)$.
\end{thm}

\begin{proof}
This follows by Lemmas~\ref{L:A1}, \ref{L:A2}, and Propositions \ref{P:GL3}, \ref{P:GL4}, \ref{P:GL5}.
\end{proof}

\subsubsection{Radical $p$ ($>2$) subgroups of $\GL_{n}(\epsilon q)$}\label{SSS:GL-odd}

Now we consider radical $p$-subgroups. Assume that $p>2$ and $q$ is coprime to $p$ throughout \S\ref{SSS:GL-odd}. Let $R$ be a radical $p$-subgroup 
of $G$. Define \[A(R)=R\cap\pi^{-1}(\Omega_{1}(Z(\pi(R)))).\]  

When $p\nmid q-\epsilon$, the map $R\rightarrow\pi(R)$ is an isomorphism. Then, $A(R)=\Omega_{1}(Z(R))$. Let $e$ be the 
order of $\epsilon q\pmod{p}$. Then, \[Z_{G}(A(R))\cong\GL_{n_{1}}((\epsilon q)^{e})\times\cdots\GL_{n_{f}}((\epsilon q)^{e})
\times\GL_{n_{0}}(\epsilon q),\] where $n_{0}+\sum_{1\leq i\leq f}en_{i}=n$. By Lemma \ref{L:R5}, 
$R$ is a radical $p$-subgroup of $Z_{G}(A(R))$. By Lemma \ref{L:R6}, $R=R_{1}\times\cdots\times R_{f}\times R_{0}$ 
with $R_{i}$ a radical $p$-subgroup of $\GL_{n_{i}}((\epsilon q)^{e})$ ($1\leq i\leq f$) and $R_{0}$ a radical $p$-subgroup 
of $\GL_{n_{0}}(\epsilon q)$. We must have $R_{0}=1$. For each $R_{i}$ ($1\leq i\leq f$), it reduces to the case 
that $p|q-\epsilon$. 

When $p\mid q-\epsilon$, the treatment is similar to the case of $p=2$. Let 
\[Z_{0}=\{\lambda\in\GL_{1}(\epsilon q): \lambda^{p}=1\}.\] Define a pairing $m: A(R)\times R\rightarrow Z_0$ by 
\[m(X,Y)=t,\ \forall(X,Y)\in A(R)\times R,\] where $YXY^{-1}X^{-1}=tI$. Define a map $\mu: A(R)\rightarrow Z_0$ by 
\[\mu(X)=t'^{\frac{q-\epsilon}{p}},\ \forall X\in A(R),\] where $X^{p}=t'I$. 
Define \[A_{1}(R)=\{X\in A(R): m(X,Y)=1,\forall Y\in R\}=A(R)\cap Z(R)\] and 
\[A_{2}(R)=\{X\in A(R): m(X,Y)=1,\forall Y\in A\}=Z(A(R)).\] Note that $\mu: A_{2}(R)\rightarrow Z_{0}$ is a group 
homomorphism. Define \[A'_{i}(R)=\{X\in A_{i}(R):\mu(X)=1\},\quad i=1,2.\] Then, $m$ descends to a non-degenerate pairing 
\[m: A(R)/A_{2}(R)\times A(R)/A_{2}(R)\rightarrow Z_{0}.\] 

The $A'_{1}$-reduction and $A'_{2}$-reduction are the same as in the case of $p=2$. After that, it reduces to the case that 
$A'_{1}(R)=A'_{2}(R)=\GL_1(\epsilon q)_pI$. Then, we have similar statement as Propositions \ref{P:GL3} and \ref{P:GL5}. However, 
the radical 2-subgroups occurring in Proposition \ref{P:GL4} have no analogue while $p>2$. A bit different with the proof of 
Proposition \ref{P:GL4}, the group $G'$ takes the form \[G'\cong\GL_{n/p^{\gamma+1}}(\epsilon q^{p})\rtimes\langle
\tau\rangle\] 
where $\tau^{p}=1$ and $\tau X\tau^{-1}=F_{q}(X)$ if $\epsilon=1$, $\tau X\tau^{-1}=(F_q(X)^t)^{-1}$ if $\epsilon=-1$. 
In the argument as there, the subgroup $R_{k-1}$ could not 
be a radical subgroup while $p>2$.  

\smallskip 

Radical $p$-subgroups could be constructed as follows. Let $M$ be a non-abelian group of order $p^3$ and exponent $p$, that is 
\[M=\langle x,y,z\mid x^p=y^p=z^p= [x,z]=[y,z]=1, [x,y]=z\rangle.\] Then $M$ can be embedded into $\GL_p(\eta q')$ for a sign 
$\eta\in\{\pm1\}$ and any prime power $q'$ with $p\mid (q'-\eta)$ by letting $x=\diag(1,\zeta,\zeta^2,\ldots,\zeta^{p-1})$ 
and $y=\left(\begin{array}{cc}&1\\I_{p-1}&\\\end{array}\right)$. For $\ga\in\mathbb Z_{\ge 1}$, we denote by $E^{2\ga+1}$ the 
extraspecial $p$-group of order $p^{2\ga+1}$ and exponent $p$, that is, $E^{2\ga+1}=M\circ_p M\circ\cdots\circ_pM$ is a 
central product of $M$ ($\ga$ times). For convenience, we also allow $\ga=0$, and in this case we interpret $E^{2\ga+1}$ as 
the group of order $p$. By taking tensor products, the group $E^{2\ga+1}$ can be embedded into $\GL_{p^\ga}(\eta q')$ such 
that its center is embedded into $Z(\GL_{p^\ga}(\eta q'))$.

Let $e$ be the multiplicative order of $\eps q$ modulo $p$, and let $a=v_p((\eps q)^e-1)$. For $\al\in\mathbb Z_{\ge 0}$, 
let $Z_\alpha$ denote the cyclic group of order $p^{a+\al}$. Then the group $Z_\al\circ_p E^{2\ga+1}$, the central product 
of $Z_\al$ and $E^{2\ga+1}$ over $\Omega_1(Z_\al)=Z(E^{2\ga+1})$, can be embedded into $\GL_{p^\ga}((\eps q)^{ep^{\alpha}})$ 
such that $Z_\al$ is identical to $Z(\GL_{p^\ga}((\eps q)^{ep^{\alpha}}))_p$. Let $R_{\al,\ga}$ be the image of 
$Z_\al\circ_p E^{2\ga+1}$ under the embedding $\GL_{p^{\ga}}((\eps q)^{ep^\al})\embed\GL_{ep^{\al+\ga}}(\eps q)$. 
Entirely analogously as the case $p=2$, we can define the \emph{basic subgroups} $R_{m,\al,\ga,\bc}$ of 
$\GL_{mep^{\al+\ga+|\bc|}}(\eps q)$. It follows that the definition and structure of the above group $R_{m,\al,\ga,\bc}$ is 
just similar as those of the basic 2-subgroup $R_{m,\al,\ga,\bc}^1$ of $\GL_{m 2^{\al+\ga+|\bc|}}(\eps q)$ for the case $p=2$ 
and $4\mid q-\eps$.

\begin{rmk}
\begin{enumerate}[\rm(1)]
    \item The basic subgroup $R_{m,\al,\ga,\bc}$ is radical except when
$(p,q,\eps)=(3,2,-1)$ and $(m,\al)\in\{(2,0),(3,0),(1,1)\}$.
\item Two basic subgroups $R_{m,\al,\ga,\bc}$ and
$R_{m',\al',\ga',\bc'}$ are conjugate if and only if
$(m,\al,\ga,\bc)=(m',\al',\ga',\bc')$.
\end{enumerate}
\end{rmk}

Then we have the following classification of radical $p$-subgroups.

\begin{thm}\label{SS:p-rad-GL}
	Let $G=\GL_n(\eps q)=\GL(V)$ or $\GU(V)$ with $p\nmid q$. If $R$ is a $p$-radical subgroup of $G$, then 
	there exists a corresponding decomposition  
	\begin{align*}
	V&=V_0\oplus V_1\oplus \cdots\oplus V_u \ \textrm{(for $\eps=1$) or} \ V=V_0\perp V_1\perp \cdots\perp V_u \ \textrm{(for $\eps=-1$)},\\
	R&=R_0\ti R_1\ti\cdots\ti R_u,
	\end{align*}
	 where $R_i$ ($1\le i\le u$) are basic subgroups of $G_i$, for $G_i=\GL(V_i)$ or $\GU(V_i)$ for $i\ge 1$, and $R_0$ 
is the trivial subgroup of $\GL(V_0)$ or  $\GU(V_0)$. In particular, $i=0$ occurs only when $p\nmid (q-\eps)$.
\end{thm}

\subsection{Radical $p$-subgroups of orthogonal and symplectic groups}\label{SS:classical2} 

Let $(V,f)$ be a non-degenerate finite dimensional symplectic or orthogonal space over $\bbF_q$ (with odd $q$), 
where $f: V\times V\rightarrow\mathbb{F}_{q}$ is a non-degenerate anti-symmetric (or symmetric) bilinear 2-form. 
Write $I(V)$ for the group of isometries of $V$. 

In the symplectic case, $f$ has a unique isomorphism class. We take $f$ with matrix 
\[\left(\begin{array}{cc}&I_{n/2}\\-I_{n/2}&\end{array}\right)\] under a standard basis of $V$ and write 
$\Sp_{n}(\mathbb{F}_{q})$ or $\Sp_n(q)$ for the corresponding symplectic group, where $n=\dim V$.   

In the orthogonal case, there are two isomorphism classes distinguished by an invariant $\ty(V)\in\{\pm{1}\}$ (type) or 
$\disc(V)\in\{\pm{1}\}$ (discriminant). We have $f\sim\sum_{1\leq i\leq n}a_{i}x_{i}^{2}$ for some 
$a_{1},\dots,a_{n}\in\mathbb{F}_{q}^{\times}$, then $\disc(V):=(a_{1}\cdots a_{n})^{\frac{q-1}{2}}$. When $\dim V=n=2m$ 
is even, we have \[f\sim\sum_{1\leq i\leq m-1}(x_{2i-1}^{2}-x_{2i}^{2})+(x_{2m-1}^{2}-\lambda x_{2m}^{2}),\] where 
$\lambda\in\mathbb{F}_{q}$. Then, $\ty(V):=\lambda^{\frac{q-1}{2}}$ and $\disc(V)=((-1)^{m}\lambda)^{\frac{q-1}{2}}$. 
When $\dim V=n=2m+1$ is odd, we have \[f\sim\sum_{1\leq i\leq m}(x_{2i-1}^{2}-x_{2i}^{2})+\lambda x_{2m+1}^{2},\] where 
$\lambda\in\mathbb{F}_{q}$. Then, $\ty(V):=\lambda^{\frac{q-1}{2}}$ and $\disc(V)=((-1)^{m}\lambda)^{\frac{q-1}{2}}$. 
Hence, it always holds true that \[\disc(V)=\ty(V)(-1)^{\lfloor\frac{n}{2}\rfloor\frac{q-1}{2}}=\ty(V)
\varepsilon^{\lfloor\frac{n}{2}\rfloor}.\] Take $f=\sum_{1\leq i\leq n}x_{i}^{2}$ and write $\rO_{n,+}(q)$ for the 
corresponding orthogonal group. Take $f=\sum_{1\leq i\leq n-1}x_{i}^{2}+\delta_{0}x_{n}^{2}$ with $\delta_{0}$ a generator 
of $\mathbb{F}_{q}^{\times}$ and write $\rO_{n,-}(q)$ for the corresponding orthogonal group. 
Moreover, we also write $\rO_{n}^+(q)$ (resp. $\rO_{n}^-(q)$) for the corresponding orthogonal groups $\rO(V)$ when $\ty(V)=+$ (resp. $-$).

Similar as in previous subsection, we first consider the prime 2.
As before, we write \(\varepsilon=(-1)^{\frac{q-1}{2}}.\) Then, $4|(q-\varepsilon)$. Put \(a=v_{2}(q-\varepsilon)=v_{2}(q^{2}-1)-1\geq 2.\)
Let $G=I(V)$, where $V$ is an $n$-dimensional vector space over $\mathbb{F}_{q}$ with a symmetric or anti-symmetric 
non-degenerate bilinear form $f$. Write $\bar{G}=G/\{\pm I\}$ and let $\pi: G\rightarrow\bar{G}$ be the projection map. Let $R$ 
be a radical 2-subgroup of $G$. Define \[A=A(R)=\pi^{-1}(\Omega_{1}(Z(\pi(R)))).\] It is clear that $A(R)\subset R$ and 
$A(R)/\{\pm{I}\}=\Omega_{1}(Z(\pi(R)))$. The subgroup $A(R)$ has the following characterization.    

\begin{lem}\label{L:m0-2}
Let $X$ be an element of $R$. Then $X\in A(R)$ if and only if \[YXY^{-1}=\pm{X},\ \forall Y\in R\] and 
$X^{2}=\pm{I}$. 
\end{lem}

By Lemma \ref{L:m0-2}, define a pairing $m: A(R)\times R\rightarrow\{\pm{1}\}$ by 
\[m(X,Y)=\lambda,\ \forall (X,Y)\in A(R)\times R,\] where $XYX^{-1}Y^{-1}=\lambda I$.

The following lemma is clear. 

\begin{lem}\label{L:m1-2}
For any $X,X'\in A(R)$ and $Y,Y'\in R$, we have \[m(XX',Y)=m(X,Y)m(X',Y)\] and \[m(X,YY')=m(X,Y)m(X,Y').\] 
\end{lem}

By Lemma \ref{L:m0-2} again, define a map $\mu: A(R)\rightarrow\{\pm{1}\}$ by \[\mu(X)=t,\ \forall X\in A(R),\] where 
$X^{2}=tI$.  
 
\begin{lem}\label{L:m2-2}
For any $X,X'\in A(R)$, we have \[\mu(XX')=\mu(X)\mu(X')m(X,X').\] 
\end{lem}

Let \[A_{1}=A_{1}(R)=\{X\in A: m(X,Y)=1,\forall Y\in R\}=A\cap Z(R)\] and 
\[A_{2}=A_{2}(R)=\{X\in A: m(X,Y)=1,\forall Y\in A\}=Z(A).\] Note that $\mu: A_{2}(R)\rightarrow\{\pm{1}\}$ is a group 
homomorphism. Define \[A'_{i}=A'_{i}(R)=\{X\in A_{i}:\mu(X)=1\},\quad i=1,2.\] Note that $m$ descends to a perfect pairing 
\[m: A/A_{2}\times A/A_{2}\rightarrow\{\pm{1}\}.\]

\subsubsection{The $A'_{1}$ reduction: direct product} \label{SS:A1'redu}

The $A'_{1}$ reduction for orthogonal and symplectic groups is the same as for general linear groups. After this, it reduces 
to classify radical 2-subgroups with $A'_{1}(R)=\{\pm{I}\}$. 

\begin{lem}\label{L:C11}
Assume that $V$ is orthogonal. If there exists $X\in I(V)$ such that $X^{2}=-I$, then $n$ is even and $\disc(V)=1$. 
\end{lem}

\begin{proof} 
Let $X\in I(V)$ have $X^{2}=-I$. For any $v\in V$, we have $(v,Xv)=(Xv,X^{2}v)=(Xv,-v)=-(v,Xv)$. Thus, $(v,Xv)=0$. Choose an 
element $v_{1}\in V$ with $(v_{1},v_{1})\neq 0$. Put $v_{2}=Xv_{1}$. Then, 
$(v_{1},v_{2})=0$ and $(v_{1},v_{1})=(v_{2},v_{2})\neq 0$. Hence, $\spann\{v_{1},v_{2}\}$ is non-degenerate. Considering the 
orthogonal complement of $\spann\{v_{1},v_{2}\}$ and taking induction, we find a basis $\{v_{1},\dots,v_{2m}\}$ of $V$ such 
that \[(v_{i},v_{j})=0,\ i\neq j; b_{i}:=(v_{2i-1},v_{2i-1})=(v_{2i},v_{2i})\neq 0,\ 1\leq i\leq m.\] Then, $n=2m$ is even and 
$\disc(V)=(b_{1}^{2}\cdots b_{m}^{2})^{\frac{q-1}{2}}=(b_{1}\cdots b_{m})^{q-1}=1$.    
\end{proof} 

\begin{lem}\label{L:C12} 
Assume that $V$ is orthogonal. If there exists $X\in I(V)$ such that $X^{2}=I$ and $X\sim -X$, then $n$ is even and 
$\disc(V)=1$.  
\end{lem}

\begin{proof} 
Let $X\in I(V)$ have $X^{2}=I$ and $X\sim -X$. Then, there exists $Y\in I(V)$ such that $YXY^{-1}=-X$. Put 
$V_{\pm}=\{v\in V: Xv=\pm{v}\}$. Then, $V=V_{+}\oplus V_{-}$ and $Y: V_{+}\rightarrow V_{-}$ is an isomorphism. Thus, 
$n$ is even and $\disc(V)=1$. 
\end{proof}  

\begin{lem}\label{L:C-disc}
Assume that $V$ is orthogonal. Let $R$ be a radical 2-subgroup of $I(V)$ with $A'_{1}(R)=\{\pm{I}\}$. If $n=\dim V$ is odd or 
$\disc(V)=-1$, then $R=\{\pm{I}\}$.
\end{lem}

\begin{proof}
Suppose that $n$ is odd or $\disc(V)=-1$. By Lemmas \ref{L:C11} and \ref{L:C12}, one shows that $A(R)=A'_{1}(R)=\{\pm{I}\}$. 
Then $R/\{\pm{I}\}$ is a radical 2-subgroup of $I(V)/\{\pm{I}\}$ with $\Omega_{1}(R/\{\pm{I}\})=A(R)/\{\pm{I}\}=1$. Thus, 
$R=\{\pm{I}\}$. 
\end{proof}

Under the assumption $A'_{1}(R)=\{\pm{I}\}$, Lemma \ref{L:C-disc} indicates that orthogonal groups with interesting radical 
2-subgroups occur only when $\dim V$ is even and $\disc(V)=1$.

\subsubsection{The $A'_{2}$ reduction: wreath product} \label{SS:wreath}

The $A'_{2}$ reduction for orthogonal and symplectic groups is similar to general linear groups. We apply a successive process 
of wreath product and reduce the problem to classify radical 2-subgroups with $A'_{2}(R)=\{\pm{I}\}$.

\subsubsection{Radical 2-subgroups with trivial $A'_{2}$}  

Assume that $A'_{1}(R)=A'_{2}(R)=\{\pm{I}\}$. There are three possibilities for $A_{1}(R)$ and $A_{2}(R)$: 
\begin{enumerate} 
\item[(1)]$A_{1}(R)/A'_{1}(R)=A_{2}(R)/A'_{2}(R)=1$;  
\item[(2)]$A_{1}(R)/A'_{1}(R)=1$ and $A_{2}(R)/A'_{2}(R)\cong Z_{2}$; 
\item[(3)]$A_{1}(R)/A'_{1}(R)=A_{2}(R)/A'_{2}(R)\cong Z_{2}$.  
\end{enumerate}

\begin{lem}\label{L:C1} 
(1)If $X,Y\in I(V)$ satisfy $XYX^{-1}Y^{-1}=-I$ and $X^{2}=Y^{2}=I$, then there exists an orthogonal decomposition 
$V=V_{+}\oplus V_{-}$ with $V_{\pm{}}$ nondegenerate and $\dim V_{\pm{}}=\frac{1}{2}\dim V$ such that: 
\begin{itemize}
\item[(i)]$X|_{V_{\pm{}}}=\pm{1}$, $Y(V_{\pm{}})=V_{\mp{}}$; 
\item[(ii)]$Y: V_{+}\rightarrow V_{-}$ and $Y: V_{-}\rightarrow V_{+}$ are inverse to each other. 
\end{itemize} 

(2)\[Z_{I(V)}(\{X,Y\})=\{\phi\in I(V):\phi(V_{\pm{}})=V_{\pm{}}\textrm{ and }\phi|_{V_{-}}=Y'\phi|_{V_{+}}Y'^{-1}\}\cong I(V_{+}),\] 
where $Y':=Y|_{V_{+}}:V_{+}\rightarrow V_{-}$. 

(3)The conjugacy class of the pair $(X,Y)$ and the orbit $I(V)\cdot V_{+}\subset V$ determine each other. In the symplectic case, 
there exists a unique conjugacy class; in the orthogonal case, there are two conjugacy classes distinguished by $\disc(V_{+})$.   
\end{lem} 

\begin{proof}
(1) Let $V_{\pm}$ be the $\pm{1}$ eigenspace of $X$. The assertions in (1) follow from the conditions $XYX^{-1}Y^{-1}=-I$ and 
$X^{2}=Y^{2}=I$. 

(2)This is clear,  

(3)Apparently, the conjugacy class of the pair $(X,Y)$ determines the orbit $I(V)\cdot V_{+}\subset V$. On the other hand, 
given $V_{+}$, one constructs $(X,Y)$ by taking the orthogonal complement $V_{-}$ and an isometry $\phi: V_{+}\rightarrow V_{-}$. 
Thus, the orbit $I(V)\cdot V_{+}\subset V$ also determines the conjugacy class of $(X,Y)$. In the symplectic case, 
$\frac{1}{2}\dim V$-dimensional non-degenerate subspaces form one $I(V)$ orbit; in the orthogonal case, 
$\frac{1}{2}\dim V$-dimensional non-degenerate subspaces split into two $I(V)$ orbits. Hence, we get the conclusion.  
\end{proof}

\begin{lem}\label{L:C1-2}
In Lemma \ref{L:C1}, let $V'_{+}$ the fixed point subspace of $Y$. Then 
\[\disc(V'_{+})=2^{\frac{q-1}{2}\dim V_{+}}\disc(V_{+}).\]
\end{lem}

\begin{proof}
It is clear that \[V'_{+}=\{v+Yv: v\in V_{+}\}.\] For any $v\in V_{+}$, we have \[(v+Yv,v+Yv)=2(v,Xv).\]  
Then, it follows that $\disc(V'_{+})=2^{\frac{q-1}{2}\dim V_{+}}\disc(V_{+})$. 
\end{proof} 

As in Lemma \ref{lem:square-sum}, choose $b,b'\in\mathbb{F}_{q}$ with $b^{2}+b'^{2}=-1$ and $b\neq 0$.  

\begin{lem}\label{L:C2}
Assume that $V$ is symplectic. 

(1) If $X,Y\in I(V)$ satisfy $XYX^{-1}Y^{-1}=-I$ and $X^{2}=Y^{2}=-I$, then there exists a decomposition $V=V_{+}\oplus V_{-}$ 
with $V_{\pm{}}$ totally isotropic such that: \begin{itemize}
\item[(i)]$X(V_{\pm{}})=V_{\mp{}}$; 
\item[(ii)]$X|_{V_{-}}=-\phi^{-1}$, $Y|_{V_{+}}=bI+b'\phi$, $Y|_{V_{-}}=-bI+b'\phi^{-1}$, where 
$\phi:=X|_{V_{+}}:V_{+}\rightarrow V_{-}$.  
\end{itemize}

(2) $g(u,v):=f(u,Xv)$ defines a non-degenerate symmetric bilinear 2-form on $V_{+}$ and we have  \[Z_{I(V)}(\{X,Y\})=
\big\{\psi\in I(V):\psi(V_{\pm{}})=V_{\pm{}}, \psi|_{V_{-}}=\phi\psi|_{V_{+}}\phi^{-1}\big\}\cong I(V_{+},g).\]   

(3) There are two conjugacy classes of such pairs $(X,Y)$ distinguished by $\disc(V_{+},g)$.  
\end{lem}  

\begin{proof}
(1)Put $Z_{\pm{}}=\frac{1}{b}(Y\mp{}b'X)$,
$V_{+}=(I+Z_{+})(V)$ and $V_{-}=(I-Z_{-})(V)$.
Then, $Z_{\pm{}}^{2}=I$ and $V=V_{+}+V_{-}$.
It is easy to show that $V_{+}\cap V_{-}=0$. Thus, $V=V_{+}\oplus V_{-}$. We have $X^{-1}Z_{\pm{}}X=-Z_{\mp}$, which implies 
that $X(V_{\pm{}})=V_{\mp{}}$. For any $u,v\in V_{+}$, we have \[(Xu,Yv)+(Yu,Xv)=(X^{2}u,XYv)+(Y^{2}u,YXv)=(-u,(XY+YX)v)=0.\] 
Then, \[(bu,bv)=((Y-b'X)u,(Y-b'X)v)=(Yu,Yv)+b'^{2}(Xu,Xv)=(1+b'^{2})(u,v).\] Thus, $(u,v)=0$ and hence $V_{+}$ is totally 
isotropic. Similarly, one shows that $V_{-}$ is totally isotropic. (ii) is clear. 

(2)$u,v\in V_{+}$, we have \[(u,Xv)=(Xu,X^{2}v)=(Xu,-v)=(v,Xu).\] Thus, $g$ is symmetric. Due to $X(V_{+})=V_{-}$, one shows that 
$g$ is non-degenerate on $V_{+}$. Other assertions in (2) are clear. 

(3)It is clear that the conjugacy class of the pair $(X,Y)$ determines the isomorphism class of the orthogonal space $(V_{+},g)$, 
as well as the discriminant of $(V_{+},g)$. On the other hand, given the discriminant of $(V_{+},g)$, one determines the triple 
$(V_{+},V_{-},\phi)$ giving the orthogonal space $(V_{+},g)$ up to conjugacy. From that, the pair $(X,Y)$ is determined.  
\end{proof}  

\begin{lem}\label{L:C2-2}
In Lemma \ref{L:C2}, let $(V'_{+},g')$ the orthogonal space associated to the pair $(Y,X)$. Then 
\[\disc(V'_{+},g')=2^{\frac{q-1}{2}\dim V_{+}}\disc(V_{+},g).\]
\end{lem}

\begin{proof}
As in Lemma \ref{L:C2}, we have a decomposition $V=V_{+}\oplus V_{-}$. From the forms of $X$ and $Y$ there, one shows that the 
fixed point subspace $V'_{+}$ of $\frac{1}{b}(X-b'Y)$ is equal to \[\{bv+(b'+1)Xv: V\in V_{+}\}.\] For any $v\in V_{+}$, we have 
\begin{eqnarray*}&&(bv+(b'+1)Xv,Y(bv+(b'+1)Xv))\\&=&(bv+(b'+1)Xv,b(bv+b'Xv)+(b'+1)(b'v-bXv))\\
&=&(bv+(b'+1)Xv,(b'-1)v-bXv)\\&=&(-b^{2}-(b'+1)(b'-1))(v,Xv)\\&=&2(v,Xv).  
\end{eqnarray*}
Then, it follows that $\disc(V'_{+},g')=2^{\frac{q-1}{2}\dim V_{+}}\disc(V_{+},g)$. 
\end{proof}

The following lemma is an analogue of Lemma \ref{L:C2}. The proof is similar. 

\begin{lem}\label{L:C3}
Assume that $V$ is orthogonal.  

(1)If $X,Y\in I(V)$ satisfy $XYX^{-1}Y^{-1}=-I$ and $X^{2}=Y^{2}=-I$, then there exists a decomposition $V=V_{+}\oplus V_{-}$ 
with $V_{\pm{}}$ totally isotropic such that: \begin{itemize}
\item[(i)]$X(V_{\pm{}})=V_{\mp{}}$; 
\item[(ii)]$X|_{V_{-}}=-\phi^{-1}$, $Y|_{V_{+}}=bI+b'\phi$, $Y|_{V_{-}}=-bI+b'\phi^{-1}$, where 
$\phi:=X|_{V_{+}}:V_{+}\rightarrow V_{-}$.  
\end{itemize}

(2)$g(u,v):=f(u,Xv)$ defines a non-degenerate anti-symmetric bilinear 2-form on $V_{+}$ and we have \[Z_{I(V)}(\{X,Y\})=
\{\psi\in I(V):\psi(V_{\pm{}})=V_{\pm{}}, \psi|_{V_{-}}=\phi\psi|_{V_{+}}\phi^{-1}\}\cong I(V_{+},g).\]    

(3)There is a unique conjugacy class of such pairs $(X,Y)$.  
\end{lem} 

In \cite[\S 2]{Yu13}, we defined three groups $\Sp(s;\epsilon,\delta)$ as the automorphism group of a symplectic 
metric space $(\mathbb{F}_{2}^{\epsilon+2(s+\delta)},m,\mu)$. Here we remark that there are isomorphisms 
\[\Sp(s;1,0)\cong\Sp_{2s}(2)\cong\SO_{2s+1}(2);\] \[\Sp(s;0,0)\cong\rO_{2s}^{+}(2);\] 
\[\Sp(s;0,1)\cong\rO_{2s+2}^{-}(2).\] 
For $s\geq1$, we write $\SO_{2s}^{+}(2)$ (resp. $\SO_{2s}^{-}(2)$) for the kernel of
the Dickson invariant in $\rO_{2s}^{+}(2)$ 
(resp. $\rO_{2s}^{-}(2)$).
It has index $2$ in $\rO_{2s}^{+}(2)$ 
(resp. $\rO_{2s}^{-}(2)$).

\begin{prop}\label{P:C4}
Let $R$ be a radical 2-subgroup of $I(V)$ such that \[A_{1}(R)=A_{2}(R)=\{\pm{I}\}.\] Write 
$\gamma=\frac{1}{2}\log_{2}|A(R)/A_{2}(R)|$. Then, $R=A(R)$ and $2^{\gamma}|n$. Moreover, there are at most three conjugacy 
classes of such radical 2-subgroups for any given values of $n$, $q$ and $\gamma$. 
\end{prop}  

\begin{proof}
When $\gamma=0$, we have $A(R)=A_{2}(R)=\{\pm{I}\}$. Then, $R/\{\pm{I}\}$ is a radical 2-subgroup of $I(V)/\{\pm{I}\}$ with $\Omega_{1}(R/\{\pm{I}\})=A(R)/\{\pm{I}\}=1$. Thus, $R=\{\pm{I}\}$.   

Assume that $\gamma\geq 1$. Using Lemmas \ref{L:m1-2} and \ref{L:m2-2}, one finds elements \[x_{1},x_{2},\dots,x_{2\gamma-1},
x_{2\gamma}\in A(R)\] generating $A(R)/A_{2}(R)$ and such that:  \[m(x_{2i-1},x_{2j-1})=m(x_{2i-1},x_{2j})=m(x_{2i},x_{2j-1})
=m(x_{2i},x_{2j})=1,\ i\neq j,\] \[m(x_{2i-1},x_{2i})=-1,\ 1\leq i\leq\gamma,\] 
\[\mu(x_{2i-1})=\mu(x_{2i})=1,\quad 1\leq i<\gamma,
\qquad \delta:=\mu(x_{2\gamma-1})=\mu(x_{2\gamma})=\pm{1}.\]
By Lemmas \ref{L:C1}, \ref{L:C2} and \ref{L:C3}, one shows that  
$2^{\gamma}|n$ and identifies the conjugacy class of the tuple $(x_{1},\dots,x_{2\gamma})$ according to the value of $\delta$ 
and the discriminant of an orthogonal space $(V_{0},g)$ while $V$ is symplectic and $\delta=-1$ or $V$ is orthogonal and 
$\delta=1$. When $V$ is symplectic and $\delta=-1$, $(V_{0},g)$ is constructed in Lemma \ref{L:C2} from the action of 
the ordered pair $(x_{2\gamma-1},x_{2\gamma})$ on the common fixed point subspace $V_{1}$ of $x_{1},x_{3},\dots,x_{2\gamma-3}$; 
when $V$ is orthogonal and $\delta=1$, $(V_{0})$ is equal to the common fixed point subspace of $x_{1},x_{3},\dots,x_{2\gamma-1}$ 
and $g=f|_{V_{0}}$. Note that \[Z_{G}(A(R)/A_{1}(R))=I(V')\cdot A(R),\] where $V'$ is an $\frac{n}{2^{\gamma}}$-dimensional symplectic 
or orthogonal space. Put $R'=R\cap I(V')$. By Lemmas \ref{L:R5} and \ref{L:R4}, $R'$ is a radical 2-subgroup of $I(V')$. It is 
easy to show that $A(R')=\{\pm{I}\}$. Then, $R'=\{\pm{I}\}$ and hence $R=A(R)$. 

(D1) When $V$ is orthogonal and $\delta=-1$, we have $2^{\gamma+1}|n$ and there exists a unique conjugacy class of ordered 
tuples $(x_{1},\dots,x_{2\gamma})$. Then, there exists a unique conjugacy class of such subgroups $R$ 
(while $2^{\gamma+1}|n$) and $N_{G}(R)/RZ_{G}(R)$ is isomorphic to the group of automorphisms of $R/\{\pm{I}\}$ preserving 
the form $m$ and the function $\mu$. That is to say (cf. \cite[\S 2]{Yu13}) , \[N_{G}(R)/RZ_{G}(R)\cong\Sp(\gamma-1;0,1)
\cong\rO_{2\gamma}^{-}(2).\]

(D2) When $V$ is orthogonal and $\delta=1$, we have $2^{\gamma}|n$ and there exist two conjugacy classes of ordered tuples 
$(x_{1},\dots,x_{2\gamma})$ distinguished by the discriminant of the common fixed point subspace $V_{0}$ of 
$x_{1},x_{3},\dots,x_{2\gamma-1}$. Then, there is a natural inclusion $N_{G}(R)/RZ_{G}(R)\subset\Sp(\gamma;0,0)\cong
\rO_{2\gamma}^{+}(2)$ with image a subgroup of index 1 or 2. If the above inclusion has full image, then there are two 
conjugacy classes of such subgroups; if the above inclusion has image of index 2, then there is a unique conjugacy 
class of such subgroups. Write $s\in\Sp(\gamma;0,0)\cong\rO_{2\gamma}^{+}(2)$ for an automorphism of $R/\{\pm{I}\}$ 
defined by \[s(x_{i})=\left\{\begin{array}{ccc}x_{i}\textrm{ if }i\leq 2\gamma-2;\\x_{2\gamma}\textrm{ if }i=2\gamma-1;\\
x_{2\gamma-1}\textrm{ if }i=2\gamma.\\\end{array}\right.\] 
If $\gamma\ne2$, the conjugates of $s$ generate
$\rO_{2\gamma}^{+}(2)$, and the elements $ss'$, where $s'$
runs over the conjugates of $s$, generate $\SO_{2\gamma}^{+}(2)$.
If $\gamma=2$, let $t$ be the automorphism defined by
$t(x_1)=x_3$, $t(x_2)=x_4$, $t(x_3)=x_1$, $t(x_4)=x_2$.
Then $t\in N_G(R)/RZ_G(R)$ and $t$ has
Dickson invariant $0$.
In this case, $\rO_4^+(2)$ is generated by $t$ and the conjugates
of $s$, while $\SO_4^+(2)$ is generated by $t$ and the elements
$ss'$, where $s'$ runs over the conjugates of $s$.
Therefore, if $s\in N_{G}(R)/RZ_{G}(R)$, then 
$N_{G}(R)/RZ_{G}(R)=\rO_{2\gamma}^{+}(2)$; if $s\not\in N_{G}(R)/RZ_{G}(R)$, then $N_{G}(R)/RZ_{G}(R)=\SO_{2\gamma}^{+}(2)$. 
Write $V'_{0}$ for the common fixed point subspace of $s\cdot x_{1},s\cdot x_{3},\dots,s\cdot x_{2\gamma-1}$. 
By Lemma \ref{L:C1-2}, we have \[\disc(V'_{0})=2^{(\frac{n}{2^{\gamma}})\frac{q-1}{2}}\disc(V_{0})=
(-1)^{(\frac{n}{2^{\gamma}})\frac{q^{2}-1}{8}}\disc(V_{0}).\] Thus, \[s\in N_{G}(R)/RZ_{G}(R)\Leftrightarrow
\disc(V'_{0})=\disc(V_{0})\Leftrightarrow(-1)^{(\frac{n}{2^{\gamma}})\frac{q^{2}-1}{8}}=1\Leftrightarrow
\frac{n}{2^{\gamma}}\frac{q^{2}-1}{8}\textrm{ is even}.\] The last condition of $\frac{n}{2^{\gamma}}\frac{q^{2}-1}{8}$ being 
even happens when $\frac{n}{2^{\gamma}}$ is even or $a=v_{2}(q^{2}-1)-1>2$. Finally, we have: \begin{itemize}
\item if $\frac{n}{2^{\gamma}}$ is even or $a=v_{2}(q^{2}-1)-1>2$, then $N_{G}(R)/RZ_{G}(R)=\rO_{2\gamma}^{+}(2)$ and there 
are two conjugacy classes of such subgroups;  
\item if $\frac{n}{2^{\gamma}}$ is odd and $a=v_{2}(q^{2}-1)-1=2$, then $N_{G}(R)/RZ_{G}(R)=\SO_{2\gamma}^{+}(2)$ and there 
is a unique conjugacy class of such subgroups.   
\end{itemize}

(C1) When $V$ is symplectic and $\delta=1$, we have $2^{\gamma+1}|n$ and there exists a unique conjugacy class of ordered tuples 
$(x_{1},\dots,x_{2\gamma})$. Then, there exists a unique conjugacy class of such subgroups $R$ (while $2^{\gamma+1}|n$) 
and $N_{G}(R)/RZ_{G}(R)$ is isomorphic to the group of automorphisms of $R/\{\pm{I}\}$ preserving the form $m$ and the function 
$\mu$. That is to say (cf. \cite[\S 2]{Yu13}) , \[N_{G}(R)/RZ_{G}(R)\cong\Sp(\gamma;0,0)\cong\rO_{2\gamma}^{+}(2).\]

(C2) When $V$ is symplectic and $\delta=-1$, we have $2^{\gamma}|n$ and there exist two conjugacy classes of ordered tuples 
$(x_{1},\dots,x_{2\gamma})$ distinguished by the discriminant of an orthogonal space $(V_{0},g)$ associated to the ordered 
tuple $x_{1},x_{3},\dots,x_{2\gamma-1}$, which is constructed from the action of the pair $(x_{2\gamma-1},x_{2\gamma})$ on the 
common fixed point subspace $V_{1}$ of $x_{1},x_{3},\dots,x_{2\gamma-3}$ as in Lemma \ref{L:C2}. Then, there is a natural 
inclusion $N_{G}(R)/RZ_{G}(R)\subset\Sp(\gamma-1;0,1)\cong\rO_{2\gamma}^{-}(2)$ with image a subgroup of index 1 or 2. If the 
above inclusion has full image, then there are two conjugacy classes of such subgroups; if the above inclusion has 
image of index 2, then there is a unique conjugacy class of such subgroups. Write 
$s\in\Sp(\gamma-1;0,1)\cong\rO_{2\gamma}^{-}(2)$ for an automorphism of $R/\{\pm{I}\}$ defined by 
\[s(x_{i})=\left\{\begin{array}{ccc}x_{i}\textrm{ if }i\leq 2\gamma-2;\\x_{2\gamma}\textrm{ if }i=2\gamma-1;\\x_{2\gamma-1}
\textrm{ if }i=2\gamma.\\\end{array}\right.\] Note that, the conjugates of $s$ generate $\rO_{2\gamma}^{-}(2)$ and $ss'$ 
($s'$ runs over all conjugates of $s$) generate $\SO_{2\gamma}^{-}(2)$. Therefore, if $s\in N_{G}(R)/RZ_{G}(R)$, then $N_{G}(R)/RZ_{G}(R)=\rO_{2\gamma}^{-}(2)$; if $s\not\in N_{G}(R)/RZ_{G}(R)$, then $N_{G}(R)/RZ_{G}(R)=\SO_{2\gamma}^{-}(2)$. 
Write $(V'_{0},g')$ for the orthogonal space associated to the ordered tuple 
$s\cdot x_{1},s\cdot x_{3},\dots,s\cdot x_{2\gamma-1}$. By Lemma \ref{L:C2-2}, we have \[\disc(V'_{0},g')=
2^{(\frac{n}{2^{\gamma}})\frac{q-1}{2}}\disc(V_{0},g)=(-1)^{(\frac{n}{2^{\gamma}})\frac{q^{2}-1}{8}}\disc(V_{0},g).\]
Thus, \[s\in N_{G}(R)/RZ_{G}(R)\Leftrightarrow\disc(V'_{0},g')=\disc(V_{0},g)\Leftrightarrow
(-1)^{\frac{n}{2^{\gamma}}\frac{q^{2}-1}{8}}=1\Leftrightarrow\frac{n}{2^{\gamma}}\frac{q^{2}-1}{8}\textrm{ is even}.\] The 
last condition of $\frac{n}{2^{\gamma}}\frac{q^{2}-1}{8}$ being even happens when $\frac{n}{2^{\gamma}}$ is even or 
$a=v_{2}(q^{2}-1)-1>2$. Finally, we have: \begin{itemize}
\item if $\frac{n}{2^{\gamma}}$ is even or $a=v_{2}(q^{2}-1)-1>2$, then $N_{G}(R)/RZ_{G}(R)=\rO_{2\gamma}^{-}(2)$ and there 
are two conjugacy classes of such subgroups;  
\item if $\frac{n}{2^{\gamma}}$ is odd and $a=v_{2}(q^{2}-1)-1=2$, then $N_{G}(R)/RZ_{G}(R)=\SO_{2\gamma}^{-}(2)$ and there 
is a unique conjugacy class of such subgroups.   
\end{itemize}

The resulting subgroup is radical if and only if the conditions
in Remark~\ref{rem:3.27} hold.
\end{proof}

\begin{rem}\label{rem:3.27}
Let $R$ be a subgroup constructed in the proof of Proposition~\ref{P:C4}.
Let $V'$ be the space in that proof, so that
$\dim V'=n/2^\gamma$ and $Z_{I(V)}(R)=I(V')$.
When $\gamma\ge1$, write $\eta$ for the type of $R$.
The space $V'$ is orthogonal when $V$ is orthogonal and $\eta=+$,
or when $V$ is symplectic and $\eta=-$.
In the other cases $V'$ is symplectic.
When $\gamma=0$, put $V'=V$ and $R=\{\pm I\}$.

Suppose that $R\not\cong D_8$.
Then $R$ is radical if and only if $O_2(I(V'))=\{\pm I\}$.
This holds if and only if the corresponding condition in the following
table holds:
\[
\begin{array}{c|c|l}
 V'&\dim V'&\text{condition}\\ \hline
 \text{symplectic}&2&q>3\\
 \text{symplectic}&\ge4\text{ even}&\text{always}\\
 \text{orthogonal}&1&\text{always}\\
 \text{orthogonal}&2&\disc(V')=-1,\quad q>3\\
 \text{orthogonal}&3&q>3\\
 \text{orthogonal}&4&q>3\text{ or }\disc(V')=-1\\
 \text{orthogonal}&\ge5&\text{always}
\end{array}
\]
When $R\cong D_8$, the subgroup $R$ is radical if and only if
$a=2$, the condition in the table holds, and $\dim V'\ne2$
when $V'$ is orthogonal.

Only the subgroups satisfying these conditions are radical.
For orthogonal $V'$, both forms must be considered.
The conjugacy classes are determined as in the proof of
Proposition~\ref{P:C4}.
\end{rem}

\begin{lem}\label{L:C4}

Let $X\in I(V)$ have $X^2=-I$.
We have $Z_{I(V)}(X)\cong\GL_{n/2}(\varepsilon q)$.
The group $Z_{I(V)}(X)$ is an index $2$ subgroup of
\[
\{Y\in I(V):YXY^{-1}=\pm X\}
=\langle Z_{I(V)}(X),\tau\rangle,
\]
where $\tau^{2}=\epsilon'I$ with $\epsilon'=\left\{\begin{array}{cc}1\textrm{ if V is orthogonal}\\
-1\textrm{ if V is symplectic,}\\\end{array}\right.$ and  
\[\tau Y\tau^{-1}=(Y^{t})^{-1},\ \forall Y\in\GL_{n/2}(\varepsilon q). \]  
\end{lem}  

\begin{proof}
Assume that $V$ is orthogonal. By Lemma \ref{L:C11}, $n=\dim V$ is even and $\disc(V)=1$. Take $f=\sum_{1\leq i\leq n}x_{i}^{2}$. 
We may assume that $X=\left(\begin{array}{cc}&I_{n/2}\\-I_{n/2}&\end{array}\right)$. Take 
$\tau=\left(\begin{array}{cc}I_{n/2}&\\&-I_{n/2}\end{array}\right)$. Then, $\tau^{2}=I$ and $\tau X\tau^{-1}=-X$. 
One calculates that 
\[\{Y\in I(V):YXY^{-1}=\pm X\}=\langle Z_{I(V)}(X),\tau\rangle,\]
where 
$\tau X\tau^{-1}=(X^{t})^{-1}$ ($\forall X\in\GL_{n/2}(\varepsilon q)$).   

Assume that $V$ is symplectic. Take 
\[f((x_{1},\dots,x_{n}),(y_{1},\dots,y_{n}))=\sum_{1\leq i\leq n/2}(x_{i}y_{i+n/2}-y_{i}x_{i+n/2}).\] We may assume that $X=\left(\begin{array}{cc}&I_{n/2}\\-I_{n/2}&\end{array}\right)$. Take $\tau=\left(\begin{array}{cc}bI_{n/2}&b'I_{n/2}\\b'I_{n/2}&-bI_{n/2}\end{array}\right)$. Then, $\tau^{2}=-I$ and 
$\tau X\tau^{-1}=-X$. One calculates that \[\{Y\in I(V):YXY^{-1}=\pm{X}\}=\langle Z_{I(V)}(X),\tau\rangle,\] 
where $\tau X\tau^{-1}=(X^{t})^{-1}$ ($\forall X\in\GL_{n/2}(\varepsilon q)$).  
\end{proof} 

Let $q$ be a power of an odd prime and $\alpha\geq 1$. Consider the group $\GL_{n}(q^{2^{\alpha}})$. Define 
\[\sigma(X)=(X^{t})^{-1},\ \forall X\in\GL_{n}(q^{2^{\alpha}}).\] For any $Y\in\GL_{n}(p^{2^{\alpha}})$, define 
\[\Ad_{Y}(X)=YXY^{-1},\ \forall X\in\GL_{n}(q^{2^{\alpha}}).\]

\begin{lem}\label{L:C5} 
Consider the group of automorphisms
\[
\langle\Ad_Y,\sigma,F_q:Y\in\GL_n(q^{2^\alpha})\rangle.
\]
(1) Any involutive automorphism in this group is of the form
\[
\Ad_Y\circ\sigma^k\circ F_{q^{2^{\alpha-1}}}^{l},
\quad k,l\in\{0,1\}.
\]
(2) Let $\theta_1$ and $\theta_2$ be two involutive automorphisms in this group.
If
\[
\theta_1|_{\GL_n(\varepsilon q)}=\theta_2|_{\GL_n(\varepsilon q)}
\quad\textrm{and}\quad
\theta_1|_{\GL_1(q^{2^\alpha})_2I}
=\theta_2|_{\GL_1(q^{2^\alpha})_2I},
\]
then $\theta_1=\theta_2$.
\end{lem}
 
\begin{proof}
(1) This is well-known. 

(2) Put $\theta=\theta_{1}\circ\theta_{2}^{-1}$. It suffices to show that $\theta=\id$. Choose a generator $\delta_{0}$ 
of $\GL_{1}(\varepsilon q)_{2}$, which has order equal to $2^{a}$. Then, $\GL_{1}(q^{2^{\alpha}})_{2}$ has a generator $\delta$ 
such that $\delta^{2^{\alpha}}=\delta_{0}$. By (1), there exist $k,l\in\{0,1\}$ such that 
$\theta=\Ad_{Y}\circ\sigma^{k}\circ F_{q^{2^{\alpha-1}}}^{l}$. From \[\delta I=\theta|_{\delta I}=
(\Ad_{Y}\circ\sigma^{k}\circ F_{q^{2^{\alpha-1}}}^{l})(\delta I)=\delta^{(-1)^{k}q^{l2^{\alpha-1}}},\] we get 
$2^{a+\alpha}|q^{l2^{\alpha-1}}-(-1)^{k}$. Note that $a\geq 2$, $\alpha\geq 1$ and $v_{2}(q-\varepsilon)=a$. We must have 
$k=l=0$. Then, $\theta=\Ad_{Y}$. From $\theta|_{\GL_{n}(\varepsilon q)}=\id$, $Y$ commutes with $\GL_{n}(\varepsilon q)$. 
Then, $Y$ is a central element and $\theta=\id$.  
\end{proof} 

The following lemma is well-known.

\begin{lem}\label{L:C6}
Let $C=\GL_m(K)$, where $K$ is a finite field of odd order.
Let $tCt^{-1}=C$ and $t^2\in Z(C)$.
Put $y=Yt$ and $y^2=st^2$, where $Y\in C$ and $s\in\{1,-1\}$.
The $C$-conjugacy classes of $y$ are given by
\[\begin{array}{c|c|c}
\textrm{action of }t&\textrm{condition on }Y&\textrm{number of classes}\\ \hline
\sigma&Y=sY^t&
\begin{array}{ll}
2,&s=1,\\
1,&s=-1,\ m\textrm{ even},\\
0,&s=-1,\ m\textrm{ odd}
\end{array}\\
\rho&Y\rho(Y)=sI&1\\
\sigma\rho&Y=s\rho(Y)^t&1
\end{array}\]
The last two cases occur only when $|K|$ is a square,
and $\rho=F_{|K|^{1/2}}$.
For $C=\GL_m(-q)$, the numbers for $\sigma$ are the same.
\end{lem}

Let \[Z_{I(V)}(A/A_{1})=Z_{I(V)}(A_{1})\cap\{Y\in I(V): YXY^{-1}=\pm{X},\forall X\in A\}.\] 

\begin{prop}\label{P:C6}
Let $R$ be a radical 2-subgroup of $I(V)$ such that \[A'_{1}(R)=A'_{2}(R)=\{\pm{I}\}\quad\textrm{ and }
\quad A_{1}(R)/A'_{1}(R)=A_{2}(R)/A'_{2}(R)\cong Z_{2}.\] Write $\gamma=\frac{1}{2}\log_{2}|A(R)/A_{2}(R)|$ and 
$|R|=2^{a+\alpha+2\gamma}$. Then, $\alpha\geq 0$ and $2^{\alpha+\gamma+1}|n$. There is at most one conjugacy 
classes of such radical 2-subgroups for any given values of $n,\gamma,\alpha$.   
\end{prop} 

\begin{proof}
By assumption, there exists an element $x_{0}\in A_{1}(R)$ with $\mu(x_{0})=-1$. Using Lemmas \ref{L:m1-2} and \ref{L:m2-2}, 
one finds elements $x_{1},x_{2},\dots,x_{2\gamma-1},x_{2\gamma}\in A(R)$ generating $A(R)/A_{2}(R)$ and such that: \[m(x_{2i-1},x_{2j-1})=m(x_{2i-1},x_{2j})=m(x_{2i},x_{2j-1})=m(x_{2i},x_{2j})=1,\ i\neq j,\] 
\[m(x_{2i-1},x_{2i})=-1,\ 1\leq i\leq\gamma,\] \[\mu(x_{2\gamma})=\mu(x_{2\gamma-1})=\cdots=\mu(x_{2})=\mu(x_{1})=1.\] By Lemmas 
\ref{L:C1} and \ref{L:C11}, one shows that $2^{\gamma+1}|n$ and there is a unique conjugacy classes of $A(R)$ for fixed values 
of $n$ and $\gamma$. Moreover, we have \[Z_{I(V)}(A(R)/A_{1}(R))=\GL_{n/2^{\gamma+1}}(\varepsilon q)\cdot A(R).\] Put 
\[R'=R\cap\GL_{n/2^{\gamma+1}}(\varepsilon q).\] By Lemmas \ref{L:R5} and \ref{L:R4}, $R'$ is a radical 2-subgroup of 
$\GL_{n/2^{\gamma+1}}(\varepsilon q)$. One shows that $A(R')=A_{2}(R')$, $A'_{1}(R')=A'_{2}(R')=\GL_{1}(\varepsilon q)_{2}I$.  
Due to $4|q-\varepsilon$, by Proposition \ref{P:GL4} the case \[(A_{1}(R')/A'_{1}(R'),A_{2}(R')/A'_{2}(R'))\cong(1,Z_{2})\] can 
not happen.
Put \(C=\mathrm{GL}_{n/2^{\gamma+1}}(\varepsilon q)\). 
Since \(R=R'A\) and \(R'\cap A=A_2\) where $|A_2|=4$, we have \(|R'|=|R|\cdot|A_2|/|A|=2^{a+\alpha}.\)
Moreover, $R'$ contains \(Z(C)_2\), which has order \(2^a\), since \(R'\) is radical in \(C\). Hence \(\alpha\geq0\).
By Propositions \ref{P:GL3} and \ref{P:GL5}, $R'$ is conjugate to the image of 
$\GL_{1}((\varepsilon q)^{2^{\alpha}})_{2}$ under the composition of homomorphisms 
\[\GL_{1}((\varepsilon q)^{2^{\alpha}})\rightarrow\GL_{2^{\alpha}}(\varepsilon q)\rightarrow\GL_{n/2^{\gamma+1}}(\varepsilon q),\] 
where $2^{\alpha+\gamma+1}|n$ and the second homomorphism above is diagonal map. Then, the conclusion follows.  
\end{proof} 

Let \[Z_{I(V)}(A/A'_{1})=Z_{I(V)}(A'_{1})\cap\{Y\in I(V): YXY^{-1}=\pm{X},\forall X\in A\}.\]   

\begin{prop}\label{P:C5}
Let $R$ be a radical 2-subgroup of $I(V)$ such that \[A_{1}(R)=A'_{1}(R)=A'_{2}(R)=\{\pm{I}\}\quad\textrm{ and }
\quad A_{2}(R)/A'_{2}(R)\cong Z_{2}.\] Write $\gamma=\frac{1}{2}\log_{2}|A(R)/A_{2}(R)|$ and $|R|=2^{a+\alpha+1+2\gamma}$. Then, 
$\alpha\geq 0$ and $2^{\alpha+\gamma+1}|n$. There are at most four conjugacy classes of such radical 2-subgroups for any given values 
of $n$, $\gamma$ and $\alpha$.    
\end{prop} 

\begin{proof}
By assumption, there exists a generator $x_{0}$ of $A_{2}(R)$ with $\mu(x_{0})=-1$. Using Lemmas \ref{L:m1-2} and \ref{L:m2-2}, 
one finds elements $x_{1},x_{2},\dots,x_{2\gamma-1},x_{2\gamma}\in A(R)$ generating $A(R)/A_{2}(R)$ with:  \[m(x_{2i-1},x_{2j-1})=m(x_{2i-1},x_{2j})=m(x_{2i},x_{2j-1})=m(x_{2i},x_{2j})=1,\ i\neq j,\] 
\[m(x_{2i-1},x_{2i})=-1,\ 1\leq i\leq\gamma,\] \[\mu(x_{2\gamma})=\mu(x_{2\gamma-1})=\cdots=\mu(x_{2})=\mu(x_{1})=1.\] By Lemma 
\ref{L:C1} and \ref{L:C11}, one shows that $2^{\gamma+1}|n$ and there is a unique conjugacy classes of $A(R)$ for fixed values 
of $n$ and $\gamma$. By Lemmas \ref{L:C1} and \ref{L:C4}, we have 
\[Z_{I(V)}(A(R)/A'_1(R))=\langle\GL_{n/2^{\gamma+1}}(\varepsilon q),\tau\rangle
\cdot A(R),\]
where $\tau^{2}=\epsilon'I$ and $\tau X\tau^{-1}=(X^{t})^{-1}$ 
($\forall X\in\GL_{n/2^{\gamma+1}}(\varepsilon q)$), $x_{0}=\delta^{2^{a-2}}I\in\GL_{n/2^{\gamma+1}}(\varepsilon q)$ with $\delta$ 
a generator of $\GL_{1}(\varepsilon q)_2$. 
Put
\[
R''=R\cap\langle\GL_{n/2^{\gamma+1}}(\varepsilon q),\tau\rangle.
\]
By Lemmas \ref{L:R5} and \ref{L:R4}, $R''$ is a radical
2-subgroup of
$\langle\GL_{n/2^{\gamma+1}}(\varepsilon q),\tau\rangle$.
Put \[R'=R''\cap\GL_{n/2^{\gamma+1}}(\varepsilon q).\] By Lemma \ref{L:R3}, $R'$ is a radical 2-subgroup of $\GL_{n/2^{\gamma+1}}(\varepsilon q)$. 
Considering the action of an element $x\in R''-R'$ on $A(R')$, one shows that $A(R')=A_{2}(R')$ and 
$A'_{1}(R')=A'_{2}(R')=\GL_{1}(\varepsilon q)_2I$. Due to $4|q-\varepsilon$, by Proposition \ref{P:GL4} the case \[(A_{1}(R')/A'_{1}(R'),A_{2}(R')/A'_{2}(R'))\cong(1,Z_{2})\] can not happen. 
Moreover, $|R'|=2^{a+\al}$ and $\al\ge0$.
By Propositions \ref{P:GL3} and \ref{P:GL5}, 
$R'$ is conjugate to the image of $\GL_{1}((\varepsilon q)^{2^{\alpha}})_{2}$ under the composition of homomorphisms 
\[\GL_{1}((\varepsilon q)^{2^{\alpha}})\rightarrow\GL_{2^{\alpha}}(\varepsilon q)\rightarrow\GL_{n/2^{\gamma+1}}(\varepsilon q),\] 
where $2^{\alpha+\gamma+1}|n$ and the second homomorphism above is diagonal map. 

Replacing $R$ by a conjugate one if necessary, we may assume that $R'$ has a generator \[x=\delta=\left(\begin{array}{ccccc}
0_{\frac{n}{2^{\alpha+\gamma+1}}}&I_{\frac{n}{2^{\alpha+\gamma+1}}}&&&\\&0_{\frac{n}{2^{\alpha+\gamma+1}}}&
I_{\frac{n}{2^{\alpha+\gamma+1}}}&&\\&&0_{\frac{n}{2^{\alpha+\gamma+1}}}&I_{\frac{n}{2^{\alpha+\gamma+1}}}&\\&&&\ddots&
\ddots\\\delta_{0}I_{\frac{n}{2^{\alpha+\gamma+1}}}&&&&0_{\frac{n}{2^{\alpha+\gamma+1}}}\end{array}\right)\in
\GL_{n/2^{\gamma+1}}(\varepsilon q),\] where $\delta_{0}$ is a generator of $\GL_{1}(\varepsilon q)_{2}$. Note that the 
centralizer of $\delta$ in $\GL_{n/2^{\gamma+1}}(\varepsilon q)$ is isomorphic to 
$\GL_{n/2^{\alpha+\gamma+1}}((\varepsilon q)^{2^{\alpha}})$. Since $yxy^{-1}\sim(x^{t})^{-1}$ and $y^{2}\in R'$ for any 
$y\in R''-R'$, we must have $R''=\langle x,y\rangle$, where $yxy^{-1}=x^{2^{a+\alpha-1}-1}=-x^{-1}$ or $yxy^{-1}=x^{-1}$. 
Then, $y^{2}=\pm{I}$. Write $yxy^{-1}=\epsilon'' x^{-1}$, where $\epsilon''=\pm{1}$.  

When $\epsilon''=-1$, we must have $\alpha\geq 1$. Take \[\tau'=\left(\begin{array}{ccccc}&&&&I_{\frac{n}{2^{\alpha+\gamma+1}}}
\\&&&-I_{\frac{n}{2^{\alpha+\gamma+1}}}&\\&&\iddots&&\\&I_{\frac{n}{2^{\alpha+\gamma+1}}}&&&\\-I_{\frac{n}{2^{\alpha+\gamma+1}}}
&&&&\end{array}\right)\tau.\] Then, we have $\tau'\delta\tau'^{-1}=-x^{-1}$ and $\tau'^{2}=-\epsilon'I$. 
Thus, $y$ lies in the coset
$\GL_{n/2^{\alpha+\gamma+1}}((\varepsilon q)^{2^\alpha})\tau'$
of the subgroup
\(\left\langle
\GL_{n/2^{\alpha+\gamma+1}}((\varepsilon q)^{2^\alpha}),
\tau'\right\rangle.\)
By Lemma~\ref{L:C5}, for every $X\in\GL_{n/2^{\alpha+\gamma+1}}((\varepsilon q)^{2^\alpha})$,
\[
\tau'X\tau'^{-1}=\left\{
\begin{array}{ll}
F_q(X),&\textrm{if }\varepsilon=-1\textrm{ and }\alpha=1,\\
F_{q^{2^{\alpha-1}}}((X^t)^{-1}),&\textrm{otherwise}.
\end{array}\right.
\]
By Lemma \ref{L:C6}, we get $y\sim\delta^{k}\tau'$ for 
some $k\in\mathbb{Z}$. Then, \[R''\sim\langle\delta,\tau'\rangle,\quad R=R''A(R).\]   

When $\epsilon''=1$, take \[\tau'=\left(\begin{array}{ccccc}&&&&I_{\frac{n}{2^{\alpha+\gamma+1}}}\\&&&
I_{\frac{n}{2^{\alpha+\gamma+1}}}&\\&&\iddots&&\\&I_{\frac{n}{2^{\alpha+\gamma+1}}}&&&\\I_{\frac{n}{2^{\alpha+\gamma+1}}}&&&&
\end{array}\right)\tau.\] Then, we have $\tau'\delta\tau'^{-1}=\delta^{-1}$ and $\tau'^{2}=\epsilon'I$. Thus, $y$ lies in the coset
$\GL_{n/2^{\alpha+\gamma+1}}((\varepsilon q)^{2^\alpha})\tau'$
of the subgroup
\(\left\langle
\GL_{n/2^{\alpha+\gamma+1}}((\varepsilon q)^{2^\alpha}),
\tau'\right\rangle.\)
Put $m=n/2^{\alpha+\gamma+1}$.
When $\alpha\geq1$, by Lemma~\ref{L:C5} we get
\[
\tau'X\tau'^{-1}=(X^t)^{-1},
\quad X\in\GL_m((\varepsilon q)^{2^\alpha});
\]
when $\alpha=0$, it is clear.
By Lemma~\ref{L:C6}, we can take
\[
R''=\langle\delta,Y\tau'\rangle,\qquad
Y=I_m,\quad \diag(\delta,1,\ldots,1),\quad\text{or}\quad J_{m/2},
\]
where the last case occurs only when $m$ is even. Then $R=R''A(R)$.
The subgroups for $Y=I_m$ and $Y=\diag(\delta,1,\ldots,1)$
are conjugate when $m$ is odd.
\end{proof}

Now we summarize the above construction of radical 2-subgroups of general symplectic and orthogonal groups.
Let $\ga\in\mathbb Z_{\ge 1}$ and $E_\eta^{2\ga+1}$ be the extraspecial group.

Suppose that $V$ is symplectic. Assume that $\dim V=2^\ga$ or $2^{\ga+1}$ according as $\eta=-$ or $+$. Then, $E_\eta^{2\ga+1}$ 
can be embedded as a subgroup of $I(V)$. Moreover, if $\dim V=2^{\ga+1}$, then $I(V)$ has one conjugacy class of 
$E_\eta^{2\ga+1}$. If $\dim V=2^\ga$, then $I(V)$ has one or two conjugacy classes of $E_\eta^{2\ga+1}$ according as $a=2$ 
or $a\ge 3$. 

Suppose that $V$ is orthogonal and $\disc(V)=1$. Assume that $\dim V=2^\ga$ or $2^{\ga+1}$ according as $\eta=+$ or $-$. 
Then, $E_\eta^{2\ga+1}$ can be embedded as a subgroup of $I(V)$. Moreover, if $\dim V=2^{\ga+1}$, then $I(V)$ has one 
conjugacy class of $E_\eta^{2\ga+1}$. If $\dim V=2^\ga$, then $I(V)$ has one or two conjugacy classes of $E_\eta^{2\ga+1}$ 
according as $a=2$ or $a\ge 3$.  

Let $V_{0,\ga}^0$ be the underlying space $V$ in the above two paragraphs. If $I(V_{0,\ga}^0)$ has one class of $E_\eta^{2\ga+1}$, 
then denote $R_{0,\ga}^0$ the image of $E_\eta^{2\ga+1}$ in $I(V_{0,\ga}^0)$. If $I(V_{0,\ga}^0)$ has two classes of 
$E_\eta^{2\ga+1}$, then denote $R_{0,\ga}^0$ and $R_{1,\ga}^0$ the representatives of the two different classes of 
$E_\eta^{2\ga+1}$ in $I(V_{0,\ga}^0)$, and set $V_{1,\ga}^0=V_{0,\ga}^0$. Note that $R_{\al,\ga}^0$ stands for two different groups, 
since $\eta=+, -$. Sometimes we also write it as $R_{\al,\ga}^{0,\eta}$ to indicate $\eta$.

Let $\al\in\bbZ_{\ge 0}$ and $Z_\al$ be a cyclic group of order $2^{a+\al}$ ($\ge 4$) and $ Z_\al \circ_2E_\eta^{2\ga+1}$ a 
central product of $E_\eta^{2\ga+1}$ and $Z_\al$. Here we suppose that $\ga\ge 0$ and in the case $\ga=0$ we set 
$E_\eta^{2\ga+1}$ to be the group of order 2. Note that $Z_\al\circ_2 E_+^{2\ga+1}\cong  Z_\al \circ_2E_-^{2\ga+1}$.
So we assume that $\eta=+$.

Recall that $Z_\al\circ_2 E_+^{2\ga+1}$ can be embedded into $\GL_{2^\ga}((\varepsilon q)^{2^\al})$. Let $V_{\al,\ga}^1$ be a 
symplectic or orthogonal space such that $\dim V_{\al,\ga}^1=2^{\al+\ga+1}$ and $\disc(V_{\al,\ga}^1)=1$ if $V_{\al,\ga}^1$ 
is orthogonal. We denote by $R_{\al,\ga}^1$ the image of $Z_\al\circ_2 E_+^{2\ga+1}$ under the embedding
\begin{equation}\label{equ:embedding-cox-I}
\GL_{2^\ga}((\varepsilon q)^{2^\al})\embed I(V_{\al,\ga}^1).
\addtocounter{thm}{1}\tag{\thethm}
\end{equation} 

Let $M$ be an orthogonal space over $\bbF_q$ such that $\dim M=2^\ga$ and $\disc(M)=1$. Then 
$E_\eta^{2\ga+1}\le I(M)$. Here we suppose that $\ga\ge 0$ and in the case $\ga=0$ we set $E_\eta^{2\ga+1}$ to be the group of 
order 2.

Recall that $S_{2^\beta}$, $D_{2^\beta}$ and $Q_{2^\beta}$ denote respectively semidihedral, dihedral and generalized quaternion 
groups of order $2^\beta$. If $\beta\ge4$ and $\ga\ge 1$, then \[S_{2^\beta}\circ_2 E_+^{2\ga+1}\cong S_{2^\beta}\circ_2 E_-^{2\ga+1},\] 
\[D_{2^\beta}\circ_2 E_+^{2\ga+1}\cong Q_{2^\beta}\circ_2 E_-^{2\ga+1},\] 
\[Q_{2^\beta}\circ_2 E_+^{2\ga+1}\cong D_{2^\beta}\circ_2 E_-^{2\ga+1}.\] 
Let $\al\in\bbZ_{\ge 0}$. 

Let $P=S_{2^{a+\al+1}}$. Assume that $\al\ge1$ for this group. Let $W$ be a symplectic or orthogonal space over $\bbF_q$ such that $\dim W=2^{\al+1}$ and $\disc(W)=1$ 
if $W$ is orthogonal. The group $P$ can be embedded as a subgroup of $I(W)$. Set $V_{\al,\ga}^2=W\otimes M$. Then the group 
$P\circ_2 E_ +^{2\ga+1}$ can be embedded as a subgroup of $I(V_{\al,\ga}^2)$ by taking tensor products. Denote by $R_{\al,\ga}^2$ 
its image. 

Let $P=D_{2^{a+\al+1}}$ or $Q_{2^{a+\al+1}}$. Let $W$ be a symplectic or orthogonal space over $\bbF_q$ according as 
$P=D_{2^{a+\al+1}}$ or $Q_{2^{a+\al+1}}$  such that $\dim W=2^{\al+2}$ and $\disc(W)=1$ if $W$ is orthogonal. The group $P$ 
can be embedded as a subgroup of $I(W)$. Set $V_{\al,\ga}^3=W\otimes M$. Then the group $P\circ_2 E_ +^{2\ga+1}$ can be embedded 
as a subgroup of $I(V_{\al,\ga}^3)$. Denote by $R_{\al,\ga}^3$ its image. If $a=2$, then $R_{0,\ga}^3$ is identical to 
$R_{0,\ga+1}^{0,+}$ or $R_{0,\ga+1}^{0,-}$ according as $W$ is symplectic or orthogonal. 

Let $P=Q_{2^{a+\al+1}}$ or $D_{2^{a+\al+1}}$. Let $W$ be a symplectic or orthogonal space over $\bbF_q$ according as 
$P=Q_{2^{a+\al+1}}$ or $D_{2^{a+\al+1}}$ such that $\dim W=2^{\al+1}$ and $\disc(W)=1$ if $W$ is orthogonal. The group $P$ 
can be embedded as a subgroup of $I(W)$. Set $V_{\al,\ga}^4=W\otimes M$. Then the group $P\circ_2 E_ +^{2\ga+1}$ can be 
embedded as a subgroup of $I(V_{\al,\ga}^4)$. Denote by $R_{\al,\ga}^4$ its image. If $a=2$, then $R_{0,\ga}^4$ is identical 
to $R_{0,\ga+1}^{0,+}$ or $R_{0,\ga+1}^{0,-}$ according as $W$ is orthogonal or symplectic. 

Thus, we have defined $R^i_{\al,\ga}$ for $i=0,1,2,3,4$ as above. Note that if $i=0$, then $\al\in\{0,1\}$ where $\al=1$ occurs 
only when $\dim(V_{\al,\ga}^0)=2^{\ga}$ and $a\ge 3$. For the case $\dim(V_{\al,\ga}^0)=2^{\ga+1}$ or $a= 2$, sometimes 
we also use the symbol $R^{0,\eta}_{1,\ga}$ to denote $R^{0,\eta}_{0,\ga}$ for convenience.

For each $m\in\bbZ_{\ge 1}$ let $V^i_{m,\al,\ga}=V^i_{\al,\ga}\perp\cdots\perp V^i_{\al,\ga}$ ($m$ times) for $i=0,1,2,3,4$. 
Denote by $R^i_{m,\al,\ga}$ the image of $R^i_{\al,\ga}$ under the $m$-fold diagonal mapping 
\begin{equation}\label{equ:embedding-m-I}
I(V^i_{\al,\ga})\embed I(V^i_{m,\al,\ga}),\quad  
g\mapsto \diag(g,g,\ldots,g)
\addtocounter{thm}{1}\tag{\thethm}
\end{equation}
for $i\ne 0$ or $\al=0$ or $m$ is odd. 

In the construction of $R^4_{\alpha,\gamma}$, let $P=\langle x,y\rangle$,
where $o(x)=2^{a+\alpha}$, $yxy^{-1}=x^{-1}$, $y^2=\epsilon'I$.
For even $m$, define
\[
R^{4,-}_{m,\alpha,\gamma}
=\left\langle
\begin{array}{l}
\diag(x,\ldots,x),\\
\diag(xy,y,\ldots,y),\\
\diag(z,\ldots,z)\quad(z\in E_+^{2\gamma+1})
\end{array}
\right\rangle
\leq I(V^4_{m,\alpha,\gamma}).
\]
The two symmetric forms in Lemma~\ref{L:C6}
give $R^4_{m,\alpha,\gamma}$ and $R^{4,-}_{m,\alpha,\gamma}$.

Suppose that $i=0$, $\al=1$ and $m$ is even. Set $V^0_{m,1,\ga}=W\otimes V^0_{0,\ga}$, where $V^0_{0,\ga}$ has dimension 
$2^\ga$ and $W$ is an $m$-dimensional orthogonal space with $\disc(W)=-1$. We denote by $R^0_{m,1,\ga}$ the image of 
$R_{0,\ga}^{0,\eta}$ under the embedding 
\begin{equation}\label{equ:embedding-m-II}
I(V^0_{0,\ga})\to I(V^0_{m,1,\ga}), g\mapsto I_m\otimes g.
\addtocounter{thm}{1}\tag{\thethm}
\end{equation}
When $i=0$ and $\ga\ge1$, sometimes we also write $R^0_{m,\al,\ga}$ as $R^{0,\eta}_{m,\al,\ga}$ to indicate $\eta$.

Let $i=\ga=0$. For symplectic cases, we define $V_{m,0,0}^0$ to be an $m$-dimensional symplectic and 
$R_{m,0,0}^0=\{\pm 1_{V_{m,0,0}^0}\}$. For orthogonal cases, we define $V_{m,0,0}^0$ (resp. $V_{m,1,0}^0$) to be an 
$m$-dimensional orthogonal space with discriminant $1$ (resp. $-1$) and $R_{m,\alpha,0}^0=\{\pm 1_{V_{m,\alpha,0}^0}\}$ 
(with $\alpha\in\{0,1\}$). Sometimes we also write $R_{m,\alpha,0}^0$ as $R_{m,\alpha,0}^{0,+}$ or $R_{m,\alpha,0}^{0,-}$ 
for convenience. 

The radical 2-subgroups obtained in Proposition \ref{P:C4} are  $R_{m,0,\gamma}^{0,\eta}$ and 
$R_{m,1,\gamma}^{0,\eta}$. Moreover, we have 
\[n=\left\{\begin{array}{cccc}m\cdot 2^{\gamma}\ \textrm{when }V\textrm{ is orthogonal and }\eta=+;\\
m\cdot 2^{\gamma+1}\ \textrm{when }V\textrm{ is orthogonal and }\eta=-;\\
m\cdot 2^{\gamma+1}\ \textrm{when }V\textrm{ is symplectic and }\eta=+;\\
m\cdot 2^{\gamma}\ \textrm{when }V\textrm{ is symplectic and }\eta=-.\\\end{array}\right.\] 

\begin{rmk}\label{rmk:c1}
Note that when $m$ is even, the subgroups $R^0_{m,1,\ga}$ and $R^{4,-}_{m,\alpha,\gamma}$ are missing in \cite{An93a}. Moreover, for symplectic groups, the sign in the centralizer of $R^0_{m,\al,\ga}$ in \cite[(1J)(b)]{An93a} is incorrect, and this will be corrected in the proof of Proposition~\ref{prop:weight-subgp-class-sym}.

Radical 2-subgroups obtained in the proof of Proposition \ref{P:C6} are $R_{m,\alpha,\gamma}^{1}$. 

Radical 2-subgroups obtained in the proof of Proposition~\ref{P:C5}
are conjugate to $R^i_{m,\alpha,\gamma}$ ($i=2,3,4$)
or $R^{4,-}_{m,\alpha,\gamma}$ ($m$ even). Moreover,
\begin{itemize}
\item[(1)] when $yxy^{-1}=-x^{-1}$, $o(x)=2^{a+\alpha}$, these are $R_{m,\alpha,\gamma}^{2}$, where $\alpha\geq 1$ 
and  $m=\frac{n}{2^{\alpha+\gamma+1}}$.  
\item[(2)] when $yxy^{-1}=x^{-1}$, $o(x)=2^{a+\alpha}$, $y^{2}=-\epsilon' I$, these are $R_{m,\alpha,\gamma}^{3}$, where 
$\alpha\geq 0$ and $m=\frac{n}{2^{\alpha+\gamma+2}}$.   
\item[(3)] when $yxy^{-1}=x^{-1}$, $o(x)=2^{a+\alpha}$ and
$y^2=\epsilon'I$, these are $R^4_{m,\alpha,\gamma}$
or $R^{4,-}_{m,\alpha,\gamma}$, where $\alpha\geq0$ and
$m=\frac{n}{2^{\alpha+\gamma+1}}$.
The latter occurs only when $m$ is even.
\end{itemize}  
\end{rmk}

\begin{rem}
In the proofs of Propositions~\ref{P:C6} and~\ref{P:C5},
the subgroup $R'$ of $\GL_{n/2^{\gamma+1}}(\varepsilon q)$ must be radical,
where $\varepsilon=(-1)^{(q-1)/2}$.
The following cases are excluded:
\begin{itemize}
\item[(1)] $R'=\GL_1(\varepsilon q)_2I$, $q=3$ and $n/2^{\gamma+1}=2$;
\item[(2)] $R'$ is cyclic and is constructed as in Proposition~\ref{P:GL5},
$q=3$, $\alpha=1$ and $n/2^{\alpha+\gamma+1}=1$.
\end{itemize}
There is no radical subgroup $R$ with such $R'$.
The case of Proposition~\ref{P:GL4} does not occur.

For the subgroups $R$ constructed above, the characteristic subgroups
must be as in the corresponding proposition, and
$O_2(N_{I(V)}(R))=R$ must hold.
For $i=0$, these conditions are given in the remark following
Proposition~\ref{P:C4}.
A subgroup with different characteristic subgroups is considered
in the corresponding case.
When $a=2$, the extraspecial subgroups are considered in
Proposition~\ref{P:C4} and counted once.
When $m$ is even, the same conditions apply.
Only the conjugacy relations proved above are used.
The subgroups obtained by direct products and wreath products must also
satisfy $O_2(N_{I(V)}(R))=R$.
\end{rem}

As before, for a sequence of positive integer $\bc=(c_t,\ldots,c_1)$, we define 
$A_{\bc}=A_{c_1}\wr A_{c_2}\wr \cdots\wr A_{c_t}$. Then $A_{\bc}\le\fS_{2^{|\bc|}}$, where  $|\bc|:=c_1+c_2+\cdots+c_t$. 
For convenience, we also allow $\bc$ to be $(0)$. Let $V^i_{m,\al,\ga,\bc}=V^i_{m,\al,\ga}\perp\cdots\perp V^i_{m,\al,\ga}$ 
($2^{|\bc|}$ times) and put $R_{m,\al,\ga,\bc}^i=R_{m,\al,\ga}^i\wr A_{\bc}$ for $i=0,1,2,3,4$. Note that 
$\disc(V^i_{m,\al,\ga,\bc})=1$ if $i+\ga+|\bc|>0$. When $i=0$, sometimes we also write $R^0_{m,\al,\ga,\bc}$ as 
$R^{0,\pm}_{m,\al,\ga,\bc}$ if $R_{m,\al,\ga,\bc}^{0,\pm}=R_{m,\al,\ga}^{0,\pm}\wr A_{\bc}$. We say that $R_{m,\al,\ga,\bc}^i$ 
is a \emph{basic subgroup} of $I(V^i_{m,\al,\ga,\bc})$ except when $i=\ga=0$ and $c_1=1$. In addition, we suppose that 
$\ga\ne 1$ if $i=0$ and $\eta=+$, and suppose that $\al\ge1$ if $i=2$.    
For even $m$, put $R^{4,-}_{m,\alpha,\gamma,\mathbf c}=R^{4,-}_{m,\alpha,\gamma}\wr A_{\mathbf c}.$
We also call $R^{4,-}_{m,\alpha,\gamma,\mathbf c}$ a basic subgroup of
$I(V^4_{m,\alpha,\gamma,\mathbf c})$.

\begin{rmk}
We present conditions under which $R^{i,\eta}_{m,\al,\ga,\bc}$ (with $\eta\in\{+,-,\emptyset\}$) is radical in $I(V^i_{m,\al,\ga,\bc})$.
\begin{enumerate}[\rm(1)]
    \item When $\ga=0$ and $V^0_{m,\al,\ga,\bc}$ is symplectic, the basic subgroup $R^0_{m,\al,\ga,\bc}$ is radical except when
$(m,q)=(2,3)$. Here $m$ is even and $\al=0$.
When $\ga=0$ and $V^0_{m,\al,\ga,\bc}$ is orthogonal, $R^0_{m,\al,\ga,\bc}$ is radical except in the following cases:
\[
\begin{array}{c|l}
m&\text{condition}\\ \hline
2&\al=0\text{ or }q=3\\
3&q=3\\
4&\al=0,\quad q=3
\end{array}
\]
Now suppose that $\ga\geq1$.
If $V^0_{m,\al,\ga,\bc}$ is symplectic and $\eta=+$, or if $V^0_{m,\al,\ga,\bc}$ is orthogonal and $\eta=-$,
then $R^{0,\eta}_{m,\al,\ga,\bc}$ is radical except when $(m,q)=(1,3)$.
If $V^0_{m,\al,\ga,\bc}$ is symplectic and $\eta=-$, or if $V^0_{m,\al,\ga,\bc}$ is orthogonal and $\eta=+$,
then $R^{0,\eta}_{m,\al,\ga,\bc}$ is radical except in the cases in the above table.
As in the definition, $\ga\ne1$ when $\eta=+$.
\item The basic subgroup $R^1_{m,\al,\ga,\bc}$ is radical
except in the following cases:
\[
\begin{array}{c|l}
\text{parameters}&\text{condition}\\ \hline
(m,\al)\in\{(1,0),(1,1),(2,0)\}
 &q-\varepsilon\text{ is a power of }2\\
(m,\al)\in\{(1,2),(2,1)\}&q=3
\end{array}
\]
\item For $i=2,3$, the basic subgroup $R^i_{m,\al,\ga,\bc}$
is radical except in the corresponding cases below:
\[
\begin{array}{c|l}
i&\text{condition}\\ \hline
2&(q,\al,m)\in\{(3,1,1),(3,1,2)\}\\
3&(q,\al,m)=(3,0,1)
\end{array}
\]
Here $\al\geq1$ for $i=2$.
\item The basic subgroup $R^4_{m,\al,\ga,\bc}$ is radical
if and only if one of the following holds:
\begin{itemize}
    \item $m=1$, $\al=0$, 
    \item $m\geq3$, $(q,\al,m)\notin\{(3,0,3),(3,0,4)\}$.
\end{itemize}
\item 
The basic subgroup $R^{4,-}_{m,\al,\ga,\bc}$ (with $m$ even)
is radical if and only if one of the following holds:
\begin{itemize}
	\item $m\ge 4$,
	\item $a=2$, $m=2$, $\al=0$, $q>3$  $\bigr(V^4_{m,\al,\ga,\bc}\text{ is symplectic or } \ga\geq1\bigr)$.
\end{itemize}
As in the definition, if $i=\ga=0$ and $A_{\bc}\ne1$, then $c_1\geq2$.
\end{enumerate}
\end{rmk}

\begin{rem}
We present the following conjugacy relations for basic subgroups.
\begin{enumerate}[\rm(1)]
\item Suppose that $a=2$. Then we have the following.
\[
\begingroup
\renewcommand{\arraystretch}{1.3}
\begin{array}{c|c|l}
 & \text{conjugacy relation} & \text{condition} \\ \hline
\text{orthogonal}
& R^3_{m,0,\gamma,\mathbf c}\sim R^{0,-}_{m,0,\gamma+1,\mathbf c}
& \gamma\ge0 \\
\text{symplectic}
& R^3_{m,0,\gamma,\mathbf c}\sim R^{0,+}_{m,0,\gamma+1,\mathbf c}
& \gamma\ge1 \\
\text{orthogonal}
& R^4_{m,0,\gamma,\mathbf c}\sim R^{0,+}_{m,0,\gamma+1,\mathbf c}
& \gamma\ge1 \\
\text{symplectic}
& R^4_{m,0,\gamma,\mathbf c}\sim R^{0,-}_{m,0,\gamma+1,\mathbf c}
& \gamma\ge0 \\
\text{orthogonal}
& R^{4,-}_{m,0,\gamma,\mathbf c}\sim R^{0,+}_{m,1,\gamma+1,\mathbf c}
& m\text{ even},\ \gamma\ge1 \\
\text{symplectic}
& R^{4,-}_{m,0,\gamma,\mathbf c}\sim R^{0,-}_{m,1,\gamma+1,\mathbf c}
& m\text{ even},\ \gamma\ge0
\end{array}
\endgroup
\]

\item Suppose that $m$ is odd and $\gamma\ge1$.
If $\eta=+$, assume further that $\gamma\ge2$. Then
\(R^{0,\eta}_{m,1,\gamma,\mathbf c}\sim R^{0,\eta}_{m,0,\gamma,\mathbf c}\)
if one of the following conditions holds:
\begin{itemize}
\item[(a)] $a=2$;
\item[(b)] $V^0_{m,0,\gamma,\mathbf c}$ is orthogonal and $\eta=-$;
\item[(c)] $V^0_{m,0,\gamma,\mathbf c}$ is symplectic and $\eta=+$.
\end{itemize}
\end{enumerate}

\end{rem}

\begin{thm}\label{subsec:radical-subgp-class}
Let $V$ be a non-degenerate finite dimensional symplectic or orthogonal space over $\bbF_q$ (with odd $q$) and let $R$ 
be a radical 2-subgroup of $I(V)$. Then there exists a corresponding decomposition \[V=V_1\perp \cdots\perp V_t\] and  
\[R=R_1\ti \cdots\ti R_t\] such that $R_i$ are basic subgroups of $I(V_i)$.
\end{thm}
	
\begin{proof}
This follows by arguments in \S\ref{SS:A1'redu}, \S\ref{SS:wreath} and Propositions \ref{P:C4}, \ref{P:C6}, \ref{P:C5}.
\end{proof}	 

\begin{lem}\label{L:gamma1}
Let $V$ be an orthogonal space and $\al\in\{0,1\}$. We have the following assertions.
\begin{enumerate}[\rm(1)]
\item If $a>2$, then $R_{m,\alpha,0,(\mathbf{c'},1)}^{0,+}\sim R_{m,\alpha,1,\mathbf{c'}}^{0,+}$ are not 
radical 2-subgroups of $I(V)$.
\item If $m$ is even, then $R_{m,\alpha,0,(\mathbf{c'},1)}^{0,+}
\sim R_{m,\alpha,1,\mathbf{c'}}^{0,+}$
are radical 2-subgroups of $I(V)$ if and only if $a=2$, $m\geq4$
and $(q,m,\alpha)\ne(3,4,0)$.
\item If $m$ is odd, $a=2$ and $(q,m)\ne(3,3)$, \[R_{m,0,0,(\mathbf{c'},1)}^{0,+}\sim R_{m,0,1,\mathbf{c'}}^{0,+}
\sim R_{m,1,0,(\mathbf{c'},1)}^{0,+}\sim R_{m,1,1,\mathbf{c'}}^{0,+}\] are radical 2-subgroups of $I(V)$.  
\end{enumerate}  
\end{lem}
	
\begin{proof}
It reduces to the case that $c'=\emptyset$. From the construction of basic subgroups, we see that $R_{m,\alpha,0,(1)}^{0,+}
\sim R_{m,\alpha,1,\emptyset}^{0,+}$. Thus, it suffices to consider $R_{m,\alpha,1,\emptyset}^{0,+}$ ($\alpha=0$ or 1).
	
(1) Suppose that $a>2$. In the proof of Proposition \ref{P:C4}, we have shown that $R_{m,0,1,\emptyset}^{0,+}$ 
and $R_{m,1,1,\emptyset}^{0,+}$ are non-conjugate; and for each of them (denoted by $R$), we have $N_{G}(R)/RZ_{G}(R)\cong
\rO_{2}^{+}(2)\cong Z_{2}$. Let \[S=\{g\in N_{G}(R): [g,Z_{G}(R)]=1\}.\] Then, $S$ is a normal subgroup of $N_{G}(R)$ 
and $S/R\cong Z_{2}$. Thus, $R$ is not a radical 2-subgroup. 

(2) Suppose that $m$ is even. Put $R=R_{m,\alpha,1,\emptyset}^{0,+}$.
Then $R\cong D_8$ and $Z_{I(V)}(R)\cong\rO_{m,+}(q)$ or
$\rO_{m,-}(q)$ according as $\alpha=0$ or $1$.
By Remark~\ref{rem:3.27}, $R$ is radical
if and only if $a=2$, $m\ne2$ and
$O_2(Z_{I(V)}(R))=\{\pm I\}$.
For even $m\ge4$, the last condition holds if and only if
$(q,m,\alpha)\ne(3,4,0)$.
    
(3) Suppose that $m$ is odd, $a=2$ and $(q,m)\ne(3,3)$. In the proof of Proposition \ref{P:C4}, we have shown that $R_{m,0,1,\emptyset}^{0,+}
\sim R_{m,1,1,\emptyset}^{0,+}$; and for each of them (denoted by $R$), we have $N_{G}(R)=RZ_{G}(R)$. Note that 
$R=\langle I_{m,m},J'_{m}\rangle\subset\rO_{2m,+}(q)$ and $Z_{G}(R)=\Delta(\rO_{m,\pm{}}(q))$. Then, $N_{G}(R)/R 
\cong\PO_{m,\pm{}}(q)\cong\SO_{m,\pm{}}(q)$. As $(q,m)\ne(3,3)$, we have $O_{2}(\SO_{m,\pm{}}(q))=1$, and thus $R$ is a radical 2-subgroup.  
\end{proof}

By Propositions \ref{P:C4}, \ref{P:C6}, \ref{P:C5} and Lemma \ref{L:gamma1}, we get the following Lemma \ref{L:O2-O8}.    

\begin{lem}\label{L:O2-O8}
Let $R$ be a radical 2-subgroup of $\rO_{k,+}$ ($k=2,4,6,8$) with $A'_{1}(R)=\{\pm{I}\}$. \begin{itemize}
\item[(i)] When $k=2$, $R$ is conjugate to one of the following: 
\begin{enumerate} 
\item[(a)]$R_{1,0,0,\emptyset}^{1}$.     
\item[(b)]$R_{1,0,0,\emptyset}^{4}$. 
\end{enumerate}
\item[(ii)] When $k=4$, $R$ is conjugate to one of the following: 
\begin{enumerate}
\item[(a)]$R_{m,0,\gamma,\mathbf{c}}^{0,+}$ or $R_{m,1,\gamma,\mathbf{c}}^{0,+}$, where $(m,\gamma,\mathbf{c})$ is one 
of \[(1,2,\emptyset),\ (1,0,(2)).\]    
\item[(b)]$R_{1,0,1,\emptyset}^{0,-}$.  
\item[(c)]$\{\pm{I}\}$. 
\item[(d)]$R_{m,\alpha,\gamma,\mathbf{c}}^{1}$ where $(m,\alpha,\gamma,\mathbf{c})$ is one of 
\[(2,0,0,\emptyset),\ (1,1,0,\emptyset),\ (1,0,1,\emptyset),\ (1,0,0,(1)).\]   
\item[(e)]$R_{1,1,0,\emptyset}^{2}$.  
\item[(f)]$R_{1,0,0,\emptyset}^{3}$. 
\item[(g)]$R_{m,\alpha,\gamma,\mathbf{c}}^{4}$ where $(m,\alpha,\gamma,\mathbf{c})$ is one of  
\[(2,0,0,\emptyset),\ (1,1,0,\emptyset),\ (1,0,1,\emptyset),\ (1,0,0,(1)).\]   
\item[(h)] $R^{4,-}_{2,0,0,\emptyset}$. 
\end{enumerate}
\item[(iii)] When $k=6$, $R$ is conjugate to one of the following: 
\begin{enumerate}
\item[(a)]$\{\pm{I}\}$. 
\item[(b)]$R_{3,0,0,\emptyset}^{1}$.     
\item[(c)]$R_{3,0,0,\emptyset}^{4}$. 
\end{enumerate}
\item[(iv)] When $k=8$, $R$ is conjugate to one of the following: 
\begin{enumerate}
\item[(a)]$R_{m,0,\gamma,\mathbf{c}}^{0,+}$ or $R_{m,1,\gamma,\mathbf{c}}^{0,+}$, where $(m,\gamma,\mathbf{c})$ is one 
of \[(1,3,\emptyset),\ (1,2,(1)),\ (1,0,(3)),\ (1,0,(1,2)),\ (2,2,\emptyset),\ (2,0,(2)).\]  
\item[(b)]$R_{m,0,\gamma,\mathbf{c}}^{0,-}$, where $(m,\gamma,\mathbf{c})$ is one of 
\[(2,1,\emptyset),\ (1,2,\emptyset),\ (1,1,(1)).\]     
\item[(c)]$\{\pm{I}\}$.  
\item[(d)]$R_{m,\alpha,\gamma,\mathbf{c}}^{1}$, where $(m,\alpha,\gamma,\mathbf{c})$ is one of 
\[(4,0,0,\emptyset),\ (2,1,0,\emptyset),\ (1,2,0,\emptyset),\] 
\[(2,0,1,\emptyset),\ (1,1,1,\emptyset),\ (1,0,2,\emptyset),\] 
\[(2,0,0,(1)),\ (1,1,0,(1)),\ (1,0,1,(1)),\] 
\[(1,0,0,(1,1)),\ (1,0,0,(2)).\]   
\item[(e)]$R_{m,\alpha,\gamma,\mathbf{c}}^{2}$, where $(m,\alpha,\gamma,\mathbf{c})$ is one of 
\[(2,1,0,\emptyset),\ (1,2,0,\emptyset),\ (1,1,1,\emptyset),\ (1,1,0,(1)).\]      
\item[(f)]$R_{m,\alpha,\gamma,\mathbf{c}}^{3}$, where $(m,\alpha,\gamma,\mathbf{c})$ is one of 
\[(2,0,0,\emptyset),\ (1,1,0,\emptyset),\ (1,0,1,\emptyset),\ (1,0,0,(1)).\]   
\item[(g)] $R_{m,\alpha,\gamma,\mathbf{c}}^{4}$ where $(m,\alpha,\gamma,\mathbf{c})$ is one of 
\[(4,0,0,\emptyset),\ (2,1,0,\emptyset),\ (1,2,0,\emptyset),\] 
\[(2,0,1,\emptyset),\ (1,1,1,\emptyset),\ (1,0,2,\emptyset),\] 
\[(2,0,0,(1)),\ (1,1,0,(1)),\ (1,0,1,(1)),\] 
\[(1,0,0,(1,1)),\ (1,0,0,(2)).\]    
\item[(h)] $R^{4,-}_{m,\alpha,\gamma,\mathbf c}$ (with even $m$), where $(m,\alpha,\gamma,\mathbf{c})$ is one of 
\[(4,0,0,\emptyset),\ (2,1,0,\emptyset),\ (2,0,1,\emptyset),\ (2,0,0,(1))\] 
\end{enumerate}
\end{itemize}
\end{lem}

\begin{proof}
Let's explain how we get all possible parameters $(m,\alpha,\gamma,\mathbf{c})$ when $k=8$. In item (a), $m2^{\gamma+|\mathbf{c}|}=8$,  
$\gamma\neq 1$ and $(\gamma,\mathbf{c})\neq (0,(\mathbf{c}',1))$; in item (b), $m2^{\gamma+|\mathbf{c}|}=4$ and $\gamma\geq 1$; 
in items (d), $m2^{\alpha+\gamma+|\mathbf{c}|}=4$; in item (e), $\alpha\geq 1$ and $m2^{\alpha+\gamma+|\mathbf{c}|}=4$; in item (f), $m2^{\alpha+\gamma+|\mathbf{c}|}=2$; in item (g), $m2^{\alpha+\gamma+|\mathbf{c}|}=4$. The cases when $k=2,4,6$ are shown similarly. 
\end{proof}

Let's explain the notations $R_{m,0,\gamma,\mathbf{c}}^{0,+}$ or $R_{m,1,\gamma,\mathbf{c}}^{0,+}$ in Lemma \ref{L:O2-O8}. 
For such a radical 2-subgroup, we first obtain a subspace $V'$ and a radical 2-subgroup $R'$ of $I(V')$ such that 
\[R=(\cdots(R'\wr Z_{2}^{c_{t}})\cdots)\wr Z_{2}^{c_{1}}\] with $\mathbf{c}=(c_{1},\dots,c_{t})$ and $R'$ a radical 2-subgroup of 
$I(V')$. Then, we obtain a subspace $V_{0}$ of $V'$ as the common 1-eigenspace of a tuple $(x_{1},x_{3},\dots,x_{2\gamma-1})$
 in $R'$ acting 
on $V'$ (cf. proof of Proposition \ref{P:C4}). In particular, when $\gamma=0$, we have $V_{0}=V'$. Then, the associated subspace 
$V_{0}$ of $R_{m,0,\gamma,\mathbf{c}}^{0,+}$ (resp. $R_{m,1,\gamma,\mathbf{c}}^{0,+}$) has discriminant $\disc(V_{0})=1$ 
(resp. $\disc(V_{0})=-1$).

\subsubsection{Radical $p$ ($>2$) subgroups of orthogonal and symplectic groups}\label{SSS:classical-odd} 

Let $G=I(V)$ where $V$ is an $n$-dimensional linear space over $\mathbb{F}_{q}$. 
Here $q$ may be odd or even.
Let $p>2$ be a prime coprime to $q$ throughout \S\ref{SSS:classical-odd}.
Suppose that $R$ is a radical $p$-subgroup of $G$. 
Put \[A(R)=\Omega_{1}(Z(R))\] and $V_0=C_V(A(R))$.
Let $e$ be the multiplicative order of $q^2$ modulo $p$, and choose
$\varepsilon\in\{1,-1\}$ such that $p\mid q^e-\varepsilon$.
Consider the nontrivial linear characters of \(A(R)\) occurring in \(V\otimes_{\mathbb F_q}\overline{\mathbb F}_q\). Let \(\chi_1,\ldots,\chi_t\) be representatives of their orbits under \(\chi\mapsto\chi^q\) and \(\chi\mapsto\chi^{-1}\).
Let $m_i$ be the multiplicity of $\chi_i$. Then
\[
Z_G(A(R))\cong I(V_0)\times\prod_{i=1}^t\GL_{m_i}(\varepsilon q^e),
\quad n=\dim V_0+2e\sum_{i=1}^t m_i.
\]
By Lemmas \ref{L:R5} and \ref{L:R6},
$R=R_0\times R_1\times\cdots\times R_t$, where $R_0=1$ and each $R_i$
($1\leq i\leq t$) is a radical $p$-subgroup of $\GL_{m_i}(\varepsilon q^e)$.
It reduces to the case of general linear and unitary groups.

Let $a$ be the positive integer such that $p^a=(q^e-\varepsilon)_p$.
Let $V_{\al,\ga}$ be a symplectic or orthogonal space over $\mathbb F_q$
with $\dim V_{\al,\ga}=2ep^{\al+\ga}$.
If $V_{\al,\ga}$ is orthogonal, assume that
$\ty(V_{\al,\ga})=\varepsilon$, where $\ty(V_{\al,\ga})$ is its usual type.
As for general linear and unitary groups, we let $E^{2\ga+1}$ be the extraspecial $p$-group of order $p^{2\ga+1}$ and exponent $p$ and $Z_\alpha$ denote the cyclic group of order $p^{a+\al}$ for $\al,\ga\in\mathbb Z_{\ge 0}$.
Define $R_{\al,\ga}$ to be the image of $Z_\al\circ_p E^{2\ga+1}$ under the embedding \[Z_\al\circ_p E^{2\ga+1}\hookrightarrow \GL_{p^\ga}(\varepsilon q^{ep^\al}) \hookrightarrow I(V_{\al,\ga}).\]
Entirely analogously to the case $p=2$ and $i=1$, we can define the \emph{basic subgroups} $R_{m,\al,\ga,\bc}$.
The orthogonal space for $R_{m,\al,\ga,\bc}$ has dimension
$2mep^{\al+\ga+|\bc|}$ and type $\varepsilon^m$.
It follows that the definition and structure of the above group $R_{m,\al,\ga,\bc}$ is just similar to those of the basic 2-subgroup $R_{m,\al,\ga,\bc}^1$ of symplectic and orthogonal groups for the case $p=2$ and $i=1$.

\begin{rmk}
\begin{enumerate}[\rm(1)]
    \item The basic subgroup $R_{m,\al,\ga,\bc}$ is radical except when
$(p,q)=(3,2)$ and $(m,\al)\in\{(2,0),(3,0),(1,1)\}$.
\item Two basic subgroups $R_{m,\al,\ga,\bc}$ and
$R_{m',\al',\ga',\bc'}$ are conjugate if and only if
$(m,\al,\ga,\bc)=(m',\al',\ga',\bc')$.
\end{enumerate}
\end{rmk}

Then we have the following classification of radical $p$-subgroups.

\begin{thm}\label{subsec:odd-radical-subgp-class}
	Let $V$ be a non-degenerate finite dimensional symplectic or orthogonal space over $\bbF_q$ (with $p\nmid q$) and let $R$ 
	be a radical $p$-subgroup of $I(V)$. Then there exists a corresponding decomposition \[V=V_0\perp V_1\perp \cdots\perp V_t\] and  
	\[R=R_0\ti R_1\ti \cdots\ti R_t\] such that $R_i$ are basic subgroups of $I(V_i)$ for $i\ge 1$, and $R_0$ is the trivial subgroup of $I(V_0)$.
\end{thm}

\subsection{Weight subgroups}\label{SS:classical-weight}

\begin{prop}\label{prop:wei-glgu}
	Let $G=\GL_n(\eps q)=\GL(V)$ or $\GU(V)$ (with odd $q$) and $R$ be a 2-weight subgroup of $G$. Suppose that 
	\begin{align*}
		V&=V_1\oplus \cdots\oplus V_u \ \textrm{(for $\eps=1$) or} \ V=V_1\perp \cdots\perp V_u \ \textrm{(for $\eps=-1$)},\\
		R&=R_1\ti\cdots\ti R_u,
		\end{align*} are corresponding decompositions in Theorem~\ref{SS:2-rad-GL} such that $R_i$ ($1\le i\le u$) is a basic subgroup of $G_i$ where $G_i=\GL(V_i)$ or $\GU(V_i)$. 
		Then for any $1 \le i \le u$ we have the following.
\begin{enumerate}[\rm(1)]
\item If $4\mid (q-\eps)$, then $R_i=R^1_{m,\al,\ga,\bc}$ with odd $m$.
\item If $4\mid (q+\eps)$, then one of the following holds.
\begin{enumerate}[\rm(a)]
	\item $R_i=R^1_{m,\al,\ga,\bc}$, where $\al\ge 1$ and $m$ is odd.
	\item $R_i=R^1_{m,0,0,\bc}$, where $c_1\ne 1$ and $m$ is odd.
	\item  $R_i=R^{1,-}_{m,0,1,\bc}$ with odd $m$.
	\item $R_i=R^2_{m,0,\ga,\bc}$ with odd $m$.
\end{enumerate}
\end{enumerate}
\end{prop}

\begin{proof}
This follows by \cite[\S3]{An92} and \cite[\S3]{An93b}.
\end{proof}

\begin{rmk}\label{rmk:wei-sub-glgu}
	Keep the hypothesis and setup of Proposition~\ref{prop:wei-glgu}. Now we assume further that $R$ is a principal 2-weight subgroup of $G$. 
	
	(1) We have $m=1$ and $\al=0$ by \cite{An92} and \cite{An93b}. In particular, the situation (2.a) of Proposition~\ref{prop:wei-glgu} does not occur.
	
	(2) Rewrite the decomposition in Proposition~\ref{prop:wei-glgu} as \[R=R_1^{u_1}\ti \cdots\ti R_s^{u_s}\] such that $R_i$ and $R_j$ are different basic subgroups if $i\ne j$.
	Then for any $1\le i\le s$, one has that $u_i$ is a triangular number, \emph{i.e.}, $u_i=\frac{k(k+1)}{2}$ for some non-negative integer $k$. In fact, by the proof of \cite[(3D)]{An92} and \cite[(3D)]{An93b}, $\fS_{u_i}$ possesses a defect zero irreducible character, which implies that there is a 2-core partition of $u_i$. 
	\end{rmk}	

\begin{lem}\label{prop:2-defect-zero-SO}
	The group $G=\SO_{m,\eps}(q)$ (with $m\ge 3$ and $\eps\in\{+,-\}$) has a 2-block with a defect group $Z(G)$ only when $m$ is even and $\eps=-$.
	\end{lem}	
	
	\begin{proof}
	Note that $G=\SO_{m,\eps}(q)$.		
	Let $B$ be a 2-block of $G$ with the defect group $Z(G)$.
	According to \cite[Thm.~21.14]{CE04}, we have that $B=\cE_2(G,s)$, where $s$ is a semisimple $2'$-element of $G^*$.
	Here $G^*=\Sp_{m-1}(q)$ if $m$ is odd, while $G^*=G$ if $m$ is even.
	Let $s=1$ so that $B$ is the principal block of $G$. If $Z(G)$ is a defect group of $B$, then $Z(G)$ is the Sylow 2-subgroup of $G$, which does not occur when $m>2$.
	Now we assume that $s\ne 1$ which implies that $s$ is not quasi-isolated by \cite[Prop.~4.11]{Bo05}.
	Let $V$ be the underlying space of $G^*$, that is, $V$ is symplectic and of dimension $m-1$ if $m$ is odd, while $V$ is orthogonal, of dimension $m$ and of discriminant $\eps$ if $m$ is even.
	
	We denote by $\Irr(\bbF_q[x])$ the set of all monic irreducible polynomials over the field $\bbF_q$.
	For any $\Delta\in \Irr(\bbF_q[x])\setminus\{x\}$, let $\Delta^*$ denote the polynomial in $\Irr(\bbF_q[x])$ whose roots are the inverses of the roots of $\Delta$.
	Following \cite{FS89}, we let
	\begin{align*}
		\cF_{0}&=\left\{ x-1,x+1 ~\right\},\\
		\cF_{1}&=\left\{ \Delta\in\Irr(\bbF_{q}[x])\mid \Delta\notin \cF_0,\ \Delta\neq x,\ \Delta=\Delta^* ~\right\},\\
		\cF_{2}&=\left\{~ \Delta\Delta^* ~|~ \Delta\in\Irr(\bbF_{q}[x])\setminus \cF_0,\ \Delta\neq x,\ \Delta\ne\Delta^* ~\right\},
	\end{align*}
	and let $\cF=\cF_0\cup\cF_1\cup\cF_2$.	
	For $\Gamma\in\cF_1\cup \cF_2$, denote by $d_\Gamma$ its degree and by $\delta_\Gamma:=\frac{d_\Gamma}{2}$.
	Since the polynomials in $\cF_1\cup \cF_2$ have even degree, $\delta_\Gamma$ is an integer.
	For $\Ga\in\cF_1\cup\cF_2$, we define $\varepsilon_\Ga=1$, $-1$ according as $\Ga\in\cF_2$ or $\cF_1$.
	 
	For the semisimple element $s\in G^*$, we have unique orthogonal decompositions
	$V=\sum\limits_{\Gamma\in\cF} V_\Gamma(s)$ and $s=\prod\limits_{\Gamma\in\cF}s_\Ga$, where $V_\Gamma(s)$ are non-degenerate subspaces of $V$, $s_\Ga\in \SO(V_\Gamma(s))$, and $s_\Gamma$ has minimal polynomial $\Gamma$.
	Let $m_\Ga(s)$ be the multiplicity of $\Ga$ in $s_\Ga$.
	Note that $\Ga\ne x+1$ since $s$ is of $2'$-order.
	Write $V_+=V_{x-1}(s)$ and $m_+=\dim V_+$ for short. In addition, $m_+$ is even. Since $s\ne 1$, we have $m_+<\dim V$.
	Then
	 \[Z_{G^*}(s)\cong I_0(V_+)\ti\prod_{\Ga\in\cF_1\cup\cF_2}\GL_{m_\Ga(s)}(\varepsilon_\Ga q^{\delta_\Ga})\]  where $I_0=\Sp$ or $\SO$ according as $V$ is symplectic or orthogonal. 
	
	Let $L$ be the subgroup of $G$ which is in duality with $Z_{G^*}(s)$, then it is isomorphic to $\SO_{m_++1}(q)\ti\prod_{\Ga\in\cF}\GL_{m_\Ga(s)}(\varepsilon_\Ga q^{\delta_\Ga})$ or $\SO_{m_+,\pm}(q)\ti\prod_{\Ga\in\cF}\GL_{m_\Ga(s)}(\varepsilon_\Ga q^{\delta_\Ga})$ according as $m$ is odd or even. 
	According to \cite[Thm.~7.7]{BDR17}, the block $B$ is splendid Rickard equivalent to the block $\cE_2(L,s)$ of $L$.
	In particular, these two blocks share a common defect group.
	Since $s$ is central in $Z_{G^*}(s)$, we have that the block $\cE_2(L,s)$ is isomorphic to the unipotent block (i.e., the principal block) of $L$.
	Thus $Z(G)$ is the Sylow 2-subgroup of $L$, which forces that $m$ is even, $Z_{G^*}(s)$ has only one component and $m_+=0$.
	
	Now we assume that $m$ is even. Then $L\cong Z_{G^*}(s)\cong \GL_{1}(\varepsilon_\Ga q^{\delta_\Ga})$ for some $\Ga\in\cF_1\cup\cF_2$ and $4\nmid(q^{\delta_\Ga}-\varepsilon_\Ga)$. Moreover, by \cite[(1.12)]{FS89}, one has that $\delta_\Ga=\frac{m}{2}$ and $\varepsilon_\Ga=\ty(V)=\varepsilon^{\frac{m}{2}}\eps$. 
	In particular, $4\mid(q^{\frac{m}{2}}+\varepsilon^{\frac{m}{2}}\eps)$.
	From this, one shows that $\eps=-$ (using, for example, \cite[Lemma 5.3 (b)]{FLZ21}), which completes the proof.
	\end{proof}
	
\begin{rmk}\label{rmk:2-defect-zero-Sp}	
	Let $G=\Sp_{2n}(q)$ and let $s$ be a semisimple $2'$-element of $G^*$.
Similarly as the proof of Lemma~\ref{prop:2-defect-zero-SO}, one shows that if the block $B=\cE_2(G,s)$ has the defect group $Z(G)$, then $Z_{G^*}(s)\cong \GL_{1}(\varepsilon_\Ga q^{\delta_\Ga})$ for some $\Ga\in\cF_1\cup\cF_2$ and $4\nmid(q^{\delta_\Ga}-\varepsilon_\Ga)$.
\end{rmk}

	\begin{cor}\label{cor:2-defect-zero-O}
		The orthogonal group $\rO_{m,\eps}(q)$ (with $m\ge 2$ and $\eps\in\{+,-\}$) has a 2-block with a defect group $\{\pm I_m\}$ only when $m$ is even and $\eps=-$.
		\end{cor}	
		
	\begin{proof}	
	If $m=2$, then the group $\rO_{2,\eps}(q)$ is isomorphic to the dihedral group of order $2(q-\varepsilon\eps)$, and thus we can get this assertion directly.
	If $m\ge 3$, then this follows from Lemma~\ref{prop:2-defect-zero-SO} and \cite[Thm. (9.26)]{Na98} immediately.
	\end{proof}
	
	\begin{prop}\label{prop:weight-subgp-class}
		Let $V$ be an orthogonal space with $\dim V\ge 2$ and $R$ a 2-weight subgroup of $I(V)$.
		Suppose that \[V=V_1\perp \cdots\perp V_t\] and \[R=R_1\ti \cdots\ti R_t\] are corresponding decompositions in Theorem~\ref{subsec:radical-subgp-class} such that $R_i$ are basic subgroups of $I(V_i)$. Then for any $1\le i\le t$ we have the following.
		\begin{enumerate}[\rm(1)]
			\item If $R_i=\{\pm 1_{V_i}\}$, then one of the following holds.
			\begin{enumerate}[\rm(a)]
				\item $\dim V_i=1$.
				\item $\dim V_i$ is even and $\disc(V_i)=-$. 
			\end{enumerate}
			\item If $R_i\ne\{\pm 1_{V_i}\}$, then $\disc(V_i)=+$ and one of the following holds.
			\begin{enumerate}[\rm(a)]
				\item $R_i=R^{0}_{1,0,0,\bc}$ or $R^{0}_{1,1,0,\bc}$,  where $c_1\ge2$.
				\item $R_i=R^{0}_{m,0,0,\bc}$ or $R^{0}_{m,1,0,\bc}$, where $m$ is even and $c_1\ge2$.
				\item $R_i=R^{0,-}_{m,0,1,\bc}$.
				\item $R_i=R_{m,\al,\ga,\bc}^1$ or $R_{m,\al,\ga,\bc}^2$, where $m$ is odd.
				\item $R_i=R_{1,0,\ga,\bc}^4$.
		\end{enumerate}
		\end{enumerate}
		\end{prop}
		
	\begin{proof}
		 Since $R$ is a weight subgroup, the group $Z_{I(V)}(R)$ possesses a block with a defect group $Z(R)$. Note that $Z_{I(V)}(R)=\prod_{i=1}^tZ_{I(V_i)}(R_i)$. So $Z_{I(V_i)}(R_i)$ possesses a block with a defect group $Z(R_i)$. If $R_i=\{\pm 1_{V_i}\}$ and $\dim V_i>1$, then by Corollary~\ref{cor:2-defect-zero-O}, we get that $\dim V_i$ is even and $\disc (V_i)=-$. 
		For (2), we assume that $R_i\ne\{\pm 1_{V_i}\}$.
        A direct calculation shows that the basic subgroups $R^{4,-}_{m,\al,\ga,\bc}$ (with $m$ even) do not occur.    
        Then this assertion follows by \cite[\S4 and \S6]{An93a}.   
		\end{proof}

\begin{rmk}\label{rmk:wei-sub-orth}
Keep the hypothesis and setup of Proposition~\ref{prop:weight-subgp-class}. Now we assume further that $R$ is a principal 2-weight subgroup of $I(V)$. 

(1) The situations (1.b), (2.b), (2.c), (2.d) of Proposition~\ref{prop:weight-subgp-class} do not occur.
In fact, if $R_i\ne\{\pm 1_{V_i}\}$, then this follows by \cite{An93a} directly.
If $R_i=\{\pm 1_{V_i}\}$, then by Lemma~\ref{lem:B_0-wei}, $\{\pm 1_{V_i}\}$ is a Sylow 2-subgroup of $I(V_i)$, which forces that $\dim V_i=1$.

(2) We rewrite the decompositions in Proposition~\ref{prop:weight-subgp-class} as \[V=V_1^{u_1}\perp \cdots\perp V_s^{u_s}\] and \[R=R_1^{u_1}\ti \cdots\ti R_s^{u_s}\] such that $R_i$ and $R_j$ are different basic subgroups if $i\ne j$.
Then similar as in Remark~\ref{rmk:wei-sub-glgu}, one has that $u_i$ is a triangular number for any $1\le i\le s$. 
\end{rmk}	

\begin{prop}\label{prop:weight-subgp-class-sym}
		Let $V$ be a symplectic space with $\dim V\ge 2$ and $R$ a 2-weight subgroup of $I(V)$.
		Suppose that \[V=V_1\perp \cdots\perp V_t\] and \[R=R_1\ti \cdots\ti R_t\] are corresponding decompositions in Theorem~\ref{subsec:radical-subgp-class} such that $R_i$ are basic subgroups of $I(V_i)$. Then for any $1\le i\le t$, one of the following holds.
				\begin{enumerate}[\rm(1)]
			\item $R_i=\{\pm 1_{V_i}\}$.
			\item $R_i=R^{0}_{m,0,0,\bc}$ with $c_1\ge 2$.
				\item $R_i=R^{0,-}_{1,0,1,\bc}$ or $R^{0,-}_{1,1,1,\bc}$ if $a\ge 3$.
				\item $R_i=R^{0,-}_{m,1,1,\bc}$, where $m$ is even.
				\item $R_i=R_{m,\al,\ga,\bc}^1$ or $R_{m,\al,\ga,\bc}^2$, where $m$ is odd.
				\item $R_i=R_{1,0,\ga,\bc}^4$.
		\end{enumerate}
\end{prop}

\begin{proof}
This assertion follows by \cite[\S4 and \S6]{An93a} and Corollary~\ref{cor:2-defect-zero-O}. However, as mentioned in Remark~\ref{rmk:c1}, the sign in the centralizer of $R^0_{m,\al,\ga}$ in \cite[(1J)(b)]{An93a} is incorrect, and when $m$ is even, the basic subgroups $R^0_{m,1,\ga,\bc}$ are missing in \cite{An93a}. Using \cite[\S6]{An93a}, the list of weight subgroups can be corrected as soon as we determine the centralizer and normalizer of those basic subgroups.

By the proof of Proposition \ref{P:C4}, we can give the structure of centralizer and normalizer of basic subgroups $R^0_{m,\al,\ga}$ as follows.
Let $\eta\in\{+,-\}$ and $\ga\ge1$. Define $C^{0,\eta}_{\al,\ga}=Z_{I(V^{0}_{\al,\ga})}(R^{0,\eta}_{\al,\ga})$ and $H^{0,\eta}_{\al,\ga}=Z_{N_{I(V^0_{\al,\ga})}(R^{0,\eta}_{\al,\ga})}(C_{\al,\ga}^{0,\eta})$.
	Then $H^{0,\eta}_{\al,\ga}/R^{0,\eta}_{\al,\ga}\cong \rO^\eta_{2\ga}(2)$ if $a\ge 3$, while $H^{0,\eta}_{\al,\ga}/R^{0,\eta}_{\al,\ga}\cong \SO^\eta_{2\ga}(2)$ if $a=2$.
	Denote the image of $H^{0,\eta}_{\al,\ga}$ under the embedding (\ref{equ:embedding-m-I}) or (\ref{equ:embedding-m-II}) by $H^{0,\eta}_{m,\al,\ga}$.
	Let $C_{m,\al,\ga}^{0,\eta}$ and $N_{m,\al,\ga}^{0,\eta}$ be the centralizer and normalizer of $R_{m,\al,\ga}^{0,\eta}$ in $I(V_{m,\al,\ga}^{0})$ respectively.

	Suppose that $V_{m,\al,\ga}^0$ is symplectic or orthogonal according as $\eta=+$ or $-$.
	Then $C_{m,\al,\ga}^{0,\eta}\cong \Sp_{2m}(q)\otimes I_{2^\ga}$.
	Moreover, $H^{0,\eta}_{m,\al,\ga}=Z_{N^{0,\eta}_{m,\al,\ga}}(C^{0,\eta}_{m,\al,\ga})$, $H^{0,\eta}_{m,\al,\ga}\cap C^{0,\eta}_{m,\al,\ga}=Z(G)$ and $H^{0,\eta}_{m,\al,\ga}/R^{0,\eta}_{m,\al,\ga}\cong \rO_{2\ga}^\eta(2)$ or $\SO_{2\ga}^\eta(2)$ according as $a\ge 3$ or $a=2$.
	In the former case $N_{m,\al,\ga}^{0,\eta}=H_{m,\al,\ga}^{0,\eta}C_{m,\al,\ga}^{0,\eta}$ and in the latter case $N_{m,\al,\ga}^{0,\eta}/H_{m,\al,\ga}^{0,\eta}C_{m,\al,\ga}^{0,\eta}$ is of order 2.
	
	Suppose that $V_{m,\al,\ga}^0$ is symplectic or orthogonal according as $\eta=-$ or $+$.
	Then $C_{m,\al,\ga}^{0,\eta}\cong \rO_{m,+}(q)\otimes I_{2^\ga}$ if $m$ is odd or $\al=0$, while $C_{m,\al,\ga}^{0,\eta}\cong \rO_{m,-}(q)\otimes I_{2^\ga}$ if $m$ is even and $\al=1$. 
	Moreover, $H^{0,\eta}_{m,\al,\ga}=Z_{N^{0,\eta}_{m,\al,\ga}}(C^{0,\eta}_{m,\al,\ga})$, $H^{0,\eta}_{m,\al,\ga}\cap C^{0,\eta}_{m,\al,\ga}=Z(G)$ and $H^{0,\eta}_{m,\al,\ga}/R^{0,\eta}_{m,\al,\ga}\cong \rO_{2\ga}^\eta(2)$ or $\SO_{2\ga}^\eta(2)$ according as $a\ge 3$ or $a=2$.
	In addition, $N_{m,\al,\ga}^{0,\eta}/H_{m,\al,\ga}^{0,\eta}C_{m,\al,\ga}^{0,\eta}$ is of order 2 or 1 according as $a=2$ and $m$ is even or otherwise. In the former case, $N_{m,\al,\ga}^{0,\eta}/R_{m,\al,\ga}^{0,\eta}C_{m,\al,\ga}^{0,\eta}\cong\rO^\eta_{2\ga}(2)$.

Now we consider the basic subgroups $R^{4,-}_{m,\al,\ga,\bc}$ with $m$ even. A direct calculation shows that these groups do not occur unless $a=2$ and $\al=\ga=0$.
In this case, they are just the basic subgroups $R^{0,-}_{m,1,1,\bc}$.

	Therefore, using the above arguments, this assertion follows by \cite[\S4 and \S6]{An93a} when $i\ge1$, and follows by Corollary \ref{cor:2-defect-zero-O}  when $i=0$.	
\end{proof}

\begin{rmk}\label{rmk:c2}
We consider symplectic groups. 
The group $R_{\Ga,\ga,\bc}=E_{d_\Ga,1,1}\wr A_{\bc'}$ in \cite[p.~190]{An93a} is the basic subgroup $R^{0,-}_{d_\Ga,0,1,\bc'}$ under our notation.
However, by Proposition~\ref{prop:weight-subgp-class-sym}, this is not a weight subgroup. On the other hand, as mentioned in Remark~\ref{rmk:c1}, the basic subgroups $R^{0,\pm}_{m,1,\ga,\bc}$ are missing in \cite{An93a}.

From this, the groups $R_{\Ga,\ga,\bc}=E_{d_\Ga,1,1}\wr A_{\bc'}$ in \cite[p.~190]{An93a} and \cite[p.~11]{FM22} should be defined as $R_{\Ga,\ga,\bc}=R_{d_\Ga,1,1}^{0,-}\wr A_{\bc'}$, and the arguments in \cite{An93a} and \cite{FM22} still apply. 
\end{rmk}

\begin{rmk}
Keep the hypothesis and setup of Proposition~\ref{prop:weight-subgp-class-sym}. Now we assume further that $R$ is a principal 2-weight subgroup of $I(V)$. 
By a similar argument of Remark~\ref{rmk:wei-sub-orth}, one has that the situations (1), (2), (4), (5) of Proposition~\ref{prop:weight-subgp-class-sym} do not occur.
\end{rmk}

\subsection{Parity}\label{SS:parity}

Let $V$ be an orthogonal space over $\mathbb{F}_{q}$ with $\dim V\geq 2$. Choose two vectors $v,v'\in V$ with 
$(v,v)\in(\mathbb{F}_{q}^{\times})^{2}$ and $(v',v')\not\in(\mathbb{F}_{q}^{\times})^{2}$. Then, $\rO(V)/\Omega(V)\cong
(Z_{2})^{2}$ and $r_{v}^{t}r_{v'}^{t'}$ ($t,t'\in\{0,1\}$) represent all four cosets. We label the coset containing 
$r_{v}^{t}r_{v'}^{t'}$ by a pair $(t,t')$. 

\begin{defn}\label{D:parity}
\begin{enumerate}[\rm(1)] 
\item We call the $\Omega(V)$ coset containing $X\in\rO(V)$ the {\it parity} of $X$. 
\item We call parities of elements in a subgroup $R$ of $I(V)$ the {\it parity of $R$}, which is regarded as a subgroup of 
$(Z_{2})^{2}$.       
\end{enumerate}   
\end{defn}

It is clear that the parity of $X\in\rO(V)$ does not depend on the choice of the vectors $v$ and $v'$. If $V'$ is an orthogonal 
space containing $V$, then the parity of any element $X\in\rO(V)$ is the same whether regarded as an element of $\rO(V)$ or $\rO(V')$.   

The following lemma is clear.  

\begin{lem}\label{L:parity0}
Let $W$ be an orthogonal subspace of $V$. Define $-I_{W}$ to be an element $X\in\rO(V)$ with $X|_{W}=-I$ and 
$X|_{W^{\perp}}=I$. Then, the parity of $-I_{W}$ is as follows: \begin{itemize}
\item if $\dim W$ is even and $\disc(W)=1$, it is equal to $(0,0)$; 
\item if $\dim W$ is even and $\disc(W)=-1$, it is equal to $(1,1)$;
\item if $\dim W$ is odd and $\disc(W)=1$, it is equal to $(1,0)$; 
\item if $\dim W$ is odd and $\disc(W)=-1$, it is equal to $(0,1)$.  
\end{itemize} 
\end{lem}

\begin{lem}\label{L:O2}
In $\rO_{2,+}(q)$, the parity of an element $\left(\begin{array}{cc}a&b\\-b&a\\\end{array}\right)$ is equal to $(t,t)$, 
where $t\in\{0,1\}$ is given by $(a+b\sqrt{-1})^{\frac{q-\varepsilon}{2}}=(-1)^{t}$.    
\end{lem} 

\begin{proof}
Write $e_{1}=(1,0)$ and $e_{2}=(0,1)$ for a standard orthogonal basis of $V=\mathbb{F}_{q}^{2}$ with 
$(e_{1},e_{1})=(e_{2},e_{2})=1$. Then, $r_{e_{2}}=\left(\begin{array}{cc}1&\\&-1\\\end{array}\right)$. There must exist 
a vector $0\neq v=c e_1+d e_{2}\in V$ such that \[r_{v}=\left(\begin{array}{cc}a&b\\b&-a\\\end{array}\right).\] 
Calculating $r_{v}(e_{1})$, we get $a=\frac{d^{2}-c^{2}}{c^{2}+d^{2}}$ and $b=\frac{-2cd}{c^{2}+d^{2}}$. Then, 
the parity of \[\left(\begin{array}{cc}a&b\\-b&a\\\end{array}\right)=r_{e_{2}}r_{v}\] is equal to $(t,t)$, where 
$t\in\{0,1\}$ is given by $(-1)^{t}=(c^{2}+d^{2})^{\frac{q-1}{2}}$. It suffices to verify that: 
\[(\frac{d^{2}-c^{2}}{c^{2}+d^{2}}+\frac{-2cd}{c^{2}+d^{2}}\sqrt{-1})^{\frac{q-\varepsilon}{2}}=(c^{2}+d^{2})^{\frac{q-1}{2}}.\] 
This is equivalent to: \[(d-c\sqrt{-1})^{q-\varepsilon}=(c^{2}+d^{2})^{q-\frac{1+\varepsilon}{2}}.\] When $\varepsilon=1$, 
both sides are equal to 1; When $\varepsilon=-1$, both sides are equal to $c^{2}+d^{2}$. 
\end{proof}  

Assume that $\dim W=2m$ is even and $\disc(W)=1$. Let $X_{0}\in I(W)$ be an element with $X_{0}^{2}=-I$. By Lemma 
\ref{L:C4}, we know that \[Z_{I(W)}(X_{0})\cong\GL_{m}(\varepsilon q).\] With this identification, one obtains a 
homomorphism \[\det: Z_{I(W)}(X_{0})\rightarrow\GL_{1}(\varepsilon q).\] 

\begin{lem}\label{L:parity1}
For any element $Y\in Z_{I(W)}(X_{0})$, the parity of $Y$ is equal to $(t,t)$, where $t\in\{0,1\}$ is determined by 
\[(-1)^{t}=\det(Y)^{\frac{q-\varepsilon}{2}}.\] 
\end{lem}

\begin{proof}
Choose a 2-dimensional $X_{0}$-stable orthogonal subspace $W_{0}$ of $W$ with $\disc(W_{0})=1$. Note that the derived subgroup $[Z_{I(W)}(X_{0}),Z_{I(W)}(X_{0})]\cong\SL_{m}(\varepsilon q)$) and $[Z_{I(W)}(X_{0}),Z_{I(W)}(X_{0})]$ together with 
$Z_{I(W_{0})}(X_{0})$ generate $Z_{I(W)}(X_{0})$. The parity of $[Z_{I(W)}(X_{0}),Z_{I(W)}(X_{0})]$ must be equal to $(0,0)$. 
Then, the statement reduces to the case that $\dim W=2$. When $\dim W=2$, the conclusion follows from Lemma \ref{L:O2}. 
\end{proof} 

\begin{lem}\label{L:parity2}
(1)The parity of $R_{m,\alpha,\gamma,\mathbf{c}}^{i}$ ($i=1,2,3,4$) is the same as that of $R_{m,\alpha,\gamma,\emptyset}^{i}$.  

(2)  The parity of $R^{0,\pm}_{m,\alpha,\gamma,\mathbf c}$
($\alpha=0,1$) is the same as that of
$R^{0,\pm}_{m,\alpha,\gamma,\emptyset}$, except when $\gamma=0$,
$m$ is odd, $a=2$ and $\mathbf c\ne\emptyset$ with $c_1=1$.
In this case, it is equal to $\{(0,0),(1,0),(0,1),(1,1)\}$.
\end{lem} 

\begin{proof}
(1)Write $\mathbf{c}=(d,\mathbf{c'})$. Then, there is an even-dimensional orthogonal space $V'$ with $\disc(V')=1$ and a basic 
subgroup $R'\sim R_{m,\alpha,\gamma,\mathbf{c'}}^{i}$ of $I(V')$ such that 
\[R_{m,\alpha,\gamma,\mathbf{c}}\sim R_{m,\alpha,\gamma,\mathbf{c'}}\wr (Z_{2})^{d}.\] Since $\dim V'$ is even and 
$\disc(V')=1$, the above $(Z_{2})^{d}$ subgroup has parity equal to $\{(0,0)\}$. Then the parity of 
$R_{m,\alpha,\gamma,\mathbf{c}}^{i}$ is equal to that of $R_{m,\alpha,\gamma,\mathbf{c'}}^{i}$. By induction one obtains 
the conclusion. 

(2)The proof is similar as above.  
\end{proof}

\begin{lem}\label{L:parity3}
We have the following assertions: 
\begin{enumerate}
\item when $m$ is even and $|\mathbf{c}|\neq 1$, the parity of $R_{m,0,0,\mathbf{c}}^{0,+}$ (resp. $R_{m,1,0,\mathbf{c}}^{0,+}$) 
is equal to $\{(0,0)\}$ (resp. $\{(0,0),(1,1)\}$);
\item when $m$ is odd, the parity of $R^{0,+}_{m,0,0,\mathbf c}$
(resp. $R^{0,+}_{m,1,0,\mathbf c}$) is equal to $\{(0,0),(1,0)\}$
(resp. $\{(0,0),(0,1)\}$), except when $a=2$ and
$\mathbf c\ne\emptyset$ with $c_1=1$.
In this case, both parities are equal to
$\{(0,0),(1,0),(0,1),(1,1)\}$;
\item The parity of $R_{m,0,1,\emptyset}^{0,-}$ is equal to $\{(0,0)\}$). 
\end{enumerate}
\end{lem}

\begin{proof}
(1) and (2) follows from Lemma~\ref{L:parity0} and \ref{L:parity2}; (3) follows from Lemma \ref{L:parity1}. 
\end{proof}

\begin{lem}\label{L:parity4}
(1)The parity of $R_{1,\alpha,0,\emptyset}^{1}$ is equal to $\{(0,0),(1,1)\}$.  

(2)The parity of $R_{1,\alpha,0,\emptyset}^{2}$ is equal to $\{(0,0),(1,1)\}$.  

(3)The parity of $R_{1,\alpha,0,\emptyset}^{3}$ is equal to $\{(0,0)\}$.  

(4)When $\alpha\geq 1$, the parity of $R_{1,\alpha,0,\emptyset}^{4}$ is equal to $\{(0,0),(1,1)\}$; the parity of 
$R_{1,0,0,\emptyset}^{4}$ is equal to $\{(0,0),(1,0),(0,1),(1,1)\}$.  
\end{lem}

\begin{proof}
(1)$R_{1,\alpha,0,\emptyset}^{1}$ contains an element $X_{0}$ with $X_{0}^{2}=-I$ and it is a cyclic group 
generated by an element $X\in Z_{I(W)}(X_{0})\cong\GL_{2^{\alpha}}(\varepsilon q)$ with $\det X$ a generator of 
$\GL_{1}(\varepsilon q)_{2}$. By Lemma \ref{L:parity1}, the parity of $X$ is equal to $(1,1)$. So is the parity of 
$R_{1,\alpha,0,\emptyset}^{1}$. 

(2)When $i=2$, we have $\alpha\geq 1$, $\dim W=2^{\alpha+1}$, and $R_{1,\alpha,0,\emptyset}^{2}$ contains 
$R_{1,\alpha,0,\emptyset}^{1}$ as an index 2 normal subgroup. We could choose an element 
$Y\in R_{1,\alpha,0,\emptyset}^{2}-R_{1,\alpha,0,\emptyset}^{1}$ of the form $Y=(-I_{W})Y'$, where $W$ is a 
$2^{\alpha}$-dimensional orthogonal subspace with $\disc(W)=1$ and $Y'\in Z_{I(W)}(X_{0})$ with $Y'^2=-I$. Then, the 
parities of $-I_{W}$ and $Y'$ are both equal to $(0,0)$. Thus, the parity of $R_{1,\alpha,0,\emptyset}^{2}$ 
is equal to $\{(0,0),(1,1)\}$.  

(3)When $i=3$, we have $\alpha\geq 0$, $\dim W=2^{\alpha+2}$, and $R_{1,\alpha,0,\emptyset}^{3}$ contains 
$R_{1,\alpha,0,\emptyset}^{1}\otimes I_{2}$ as an index 2 normal subgroup. Note that the parity of 
$R_{1,\alpha,0,\emptyset}^{1}\otimes I_{2}$ is equal to $(0,0)$. We could choose an element 
$Y\in R_{1,\alpha,0,\emptyset}^{3}-R_{1,\alpha,0,\emptyset}^{1}\otimes I_{2}$ of the form $Y=(-I_{W})Y'$, where $W$ is a 
$2^{\alpha+1}$-dimensional orthogonal subspace with $\disc(W)=1$ and $Y'\in Z_{I(W)}(X_{0})$ with $Y'^2=-I$. Then, the 
parities of $-I_{W}$ and $Y'$ are both equal to $(0,0)$. Thus, the parity of $R_{1,\alpha,0,\emptyset}^{3}$ is equal to 
$\{(0,0)\}$.   

(4)When $i=4$, we have $\alpha\geq 0$, $\dim W=2^{\alpha+1}$, and $R_{1,\alpha,0,\emptyset}^{4}$ contains 
$R_{1,\alpha,0,\emptyset}^{1}$ as an index 2 normal subgroup. We could choose an element 
$Y\in R_{1,\alpha,0,\emptyset}^{4}-R_{1,\alpha,0,\emptyset}^{1}$ of the form $Y=(-I_{W})Y'$, where $W$ is a 
$2^{\alpha}$-dimensional orthogonal subspace with $\disc(W)=1$ and $Y'\in Z_{I(W)}(X_{0})$. If $\alpha\geq 1$, 
then the parity of $-I_{W}$ is equal to $(0,0)$; if $\alpha=0$, then the parity of $-I_{W}$ is equal to $(1,0)$. 
We could specify a choice of $Y'$ with parity $(0,0)$. Thus, the conclusion follows. 
\end{proof}

\begin{lem}\label{L:parity5}
When $m$ is even or $\gamma\geq 1$, the parity of $R_{m,\alpha,\gamma,\emptyset}^{i}$ ($i=1,2,3,4$) is equal to $\{(0,0)\}$. 

When $m$ is odd, the parity of $R_{m,\alpha,0,\emptyset}^{i}$ ($i=1,2,3,4$) is the same as that of $R_{1,\alpha,0,\emptyset}^{i}$.  
\end{lem} 

\begin{proof}
There is an orthogonal space $V'$ and a basic subgroup $R'$ ($\sim R_{1,\alpha,\gamma,\emptyset}^{i}$) of $I(V')$ such that 
$V=V^{'\oplus m}$ and $R_{m,\alpha,\gamma,\emptyset}^{i}\sim R'\otimes I_{m}$. If $m$ is even, this implies that the parity of $R_{m,\alpha,\gamma,\emptyset}^{i}$ ($i=1,2,3,4$) is equal to $\{(0,0)\}$; if $m$ is odd, this implies that the parity of 
$R_{m,\alpha,0,\emptyset}^{i}$ ($i=1,2,3,4$) is equal to that of $R_{1,\alpha,0,\emptyset}^{i}$. 

Assume that $\gamma\geq 1$. Then, there are orthogonal spaces $V_{1}$ and $V_{2}$, and a basic subgroup 
$R_{1}\sim R_{m,\alpha,0,\emptyset}^{i}$ of $I(V_{1})$ (resp. a basic subgroup $R_{2}\sim R_{1,0,\gamma,\emptyset}^{0,+}$ of 
$I(V_{2})$) such that $V=V_{1}\otimes V_{2}$ and $R_{m,\alpha,\gamma,\emptyset}^{i}\sim R_{1}\otimes R_{2}$. Since $\dim V_{1},  
\dim V_{2}$ are both even and $\disc(V_{1})=\disc(V_{2})=1$, it follows that the parity of $R_{m,\alpha,\gamma,\emptyset}^{i}$ 
($i=1,2,3,4$) is equal to $\{(0,0)\}$. 
\end{proof}

\begin{rem}
Suppose that $m$ is even. The parity of
$R^{4,-}_{m,\alpha,\gamma,\mathbf c}$ is equal to
\[
\begin{cases}
\{(0,0),(1,1)\},&\gamma=0,\\
\{(0,0)\},&\gamma\geq1.
\end{cases}
\]
By Lemma~\ref{L:parity4}(1), the parity of $x$ is $(1,1)$
when $\gamma=0$. It is $(0,0)$ when $\gamma\geq1$,
by the proof of Lemma~\ref{L:parity5}. Since $m$ is even and
\(
\diag(xy,y,\ldots,y)
=\diag(x,I,\ldots,I)\diag(y,y,\ldots,y),
\)
the other generators have parity $(0,0)$, and the conclusion follows
for $\mathbf c=\emptyset$.
The proof for general $\mathbf c$ is the same as in
Lemma~\ref{L:parity2}(1).
\end{rem}

\begin{lem}\label{L:parity6}
When $\gamma\geq 2$, the parity of $R_{m,\alpha,\gamma,\emptyset}^{0,\pm{}}$ ($\alpha=0,1$) is equal to $\{(0,0)\}$.  
\end{lem}

\begin{proof}
There are orthogonal spaces $V_{1}$ and $V_{2}$ with $\dim V_{2}=2$ and 
$V=V_{1}\otimes V_{2}$, and a basic subgroup $R_{1}\sim R_{m,\alpha,\gamma-1,\emptyset}^{0,\pm{}}$ of $I(V_{1})$ 
(resp. a basic subgroup $R_{2}\sim R_{1,0,1,\emptyset}^{0,+}$ of $I(V_{2})$) such that 
$R_{m,\alpha,\gamma,\emptyset}^{0,\pm{}}\sim R_{1}\otimes R_{2}$. Since $\dim V_{1},\dim V_{2}$ are both even and 
$\disc(V_{1})=\disc(V_{2})=1$, it follows that the parity of $R_{m,\alpha,\gamma,\emptyset}^{0,\pm{}}$ ($\alpha=0,1$) 
is equal to $\{(0,0)\}$.  
\end{proof}

Define $I_{0}(W)=\{X\in I(W):\det X=1\}$. 

\begin{cor}\label{C:parity6}
(1)Assume that $\gamma\geq 0$. Then, $R_{m,\alpha,\gamma,\mathbf{c}}^{0,+}\not\subset I_{0}(W)$ happens only when $m$ is odd 
and $\gamma\leq 1$.  

(2)Assume that $\gamma\geq 1$. We always have $R_{m,0,\gamma,\mathbf{c}}^{0,-}\subset\Omega(W)$. 

(3)We always have $R_{m,\alpha,\gamma,\mathbf{c}}^{i}\subset I_{0}(W)$ ($i=1,2$). 

(4)Assume that $\gamma\geq 0$. We always have $R_{m,\alpha,\gamma,\mathbf{c}}^{3}\subset\Omega(W)$.  

(5)$R_{m,\alpha,\gamma,\mathbf{c}}^{4}\not\subset I_{0}(W)$  happens only when $m$ is odd and $\alpha=\gamma=0$.  
\end{cor}

\begin{proof}
These follow from Lemmas \ref{L:parity2}-\ref{L:parity6}. 
\end{proof}

The above arguments give an algorithm to calculate the parity
of a general element in
any basic subgroup $R$ of an orthogonal group $I(W)$. With that, one can calculate $R\cap I_{0}(W)$ and $R\cap\Omega(W)$ and 
specify which elements lie outside $I_{0}(W)$ (or $\Omega(W)$). Doing so, we can classify radical 2-subgroups of $\Spin(W)$.

\subsection{Principal weights}\label{SS:prin-wei}

\begin{prop}\label{P:Ta1}  
	Basic subgroups giving principal weights are as in the following Table~\ref{Ta1}.    
	\end{prop}
	
	\begin{proof}
	By Proposition \ref{prop:weight-subgp-class} and Remark \ref{rmk:wei-sub-orth}, we get the list of basic subgroups. 
	By Lemmas \ref{L:parity2}--\ref{L:parity6}, we calculate their parities. 
	\end{proof}
	
	In Table~\ref{Ta1}, we list those basic subgroups $R$ from Proposition~\ref{P:Ta1} and their parities in the 4th and 5th columns respectively, list the information of the underlying spaces $V$ in the 2nd and 3rd columns respectively, and list the number of such basic subgroups in the last column.
	In addition, we denote by 0, $(1,0)$, $(0,1)$, $(Z_2)^2$ the groups $\{(0,0)\}$, $\{(0,0),(1,0)\}$, $\{(0,0),(0,1)\}$, $\{(0,0),(1,0),(0,1),(1,1)\}$, respectively.  

	\begin{table}[!htb]
		\centering
		\renewcommand\arraystretch{1.3}
		\begin{tabular}{|c|c|c|c|c|c|}  
			\hline  
			& $\dim(V)$ & $\disc(V)$ & $R$ &  parity &  number  \\
		\hline		 
		1 &  $2^n$ (with $n\ge 2$)  & $+$ & $R^4_{1,0,\ga,\bc}$ with $\ga+|\bc|=n-1$ and $\ga\ge 1$  & 0  &   $2^{n-2}$ \\
		\hline	 	
		2 & $2^n$ (with $n\ge 2$) & $+$ & $R^0_{1,0,0,\bc}$ with $|\bc|=n$ and $c_1\ge 2$  &  $(1,0)$  &  $2^{n-2}$  \\
		\hline	
		3 & $2^n$ (with $n\ge 2$) & $+$ & $R^0_{1,1,0,\bc}$ with $|\bc|=n$ and $c_1\ge 2$ &  $(0,1)$ &  $2^{n-2}$  \\
		\hline	 
		4 & $2^n$ (with $n\ge 2$) & $+$ & $R^4_{1,0,0,\bc}$ with $|\bc|=n-1$ & $(Z_2)^2$  &   $2^{n-2}$ \\
		\hline
		5 & 2 & $+$ & $R_{1,0,0,\emptyset}^{4}$ &  $(Z_2)^2$ &  1  \\
		\hline
		6 & 1 & $+$ & $\{\pm 1_V\}$ & $(1,0)$  &   1 \\
		\hline
		7 & 1 & $-$ & $\{\pm 1_V\}$ & $(0,1)$  &   1 \\
		\hline
		\end{tabular} 
	\caption{Basic subgroups for principal weight subgroups of orthogonal groups}\label{Ta1}
	\end{table}
	
Suppose that $q$ is some power of an odd prime $r$. Let
\[
F_0(X)=YF_r(X)Y^{-1},\qquad
Y=\begin{cases}
	I_n,&\text{if }\disc(V)=+,\\
	\diag(1,\dots,1,\delta_0^{(r-1)/2}),&\text{if }\disc(V)=-,
\end{cases}
\]
where $n=\dim V$ and $\delta_0$ is as in \S\ref{SS:classical2}.

\begin{prop}\label{prop:act-fie-auto}
Let $G=\rO(V)$, and let $R$ be a basic subgroup in Table~\ref{Ta1}.
Then after a change of $R$ by a suitable $G$-conjugate, there exists
$g\in\Omega(V)$ such that the following statements hold.
\begin{enumerate}[\rm(1)]
	\item $R$ is $\langle\Ad_g\circ F_0\rangle$-stable.
	\item The group $\langle\Ad_g\circ F_0\rangle$ acts trivially on
	$N_G(R)/RZ_G(R)$.
\end{enumerate}
\end{prop}

\begin{proof}
(1) Note that permutation matrices are $\langle F_0\rangle$-invariant. This statement will be proved once we prove it for the situation $\bc=\emptyset$.
If $i=0$, then $R$ is the center of $G$ and thus it is $\langle F_0\rangle$-stable. Now let $i=4$. If $\ga=0$, then $R$ is a Sylow 2-subgroup of $G$. We may take $R$ to be the group generated by the image of the Sylow 2-subgroup of (\ref{equ:embedding-cox-I}) (where we let $\al=\ga=0$) and the permutation matrix of degree 2 and order 2. For $\ga>0$, one shows that $R$ is $\langle F_0\rangle$-stable by construction.

(2) Similar as above, it is sufficient to show this statement for the situation $\bc=\emptyset$.  The case $i=0$ is special by construction and we let $i=4$. If $\ga=0$, then $N_G(R)=R$, and so (2) holds. Now suppose that $\ga\ge 1$. Then the parity of $R$ is equal to $\{(0,0)\}$. First we assume that $a\ge 3$ and let $\hat R=Z_R([R,R])$. Then $N_G(R)=N'R$ with $N'=Z_{N_G(R)}(Z(\hat R))$. Note that $Z(\hat R)$ contains an element whose square is $-I$ and thus the parity of $N_G(R)$ can be computed by Lemma~\ref{L:parity1}. Now $Z_R(Z(\hat R))$ is a central product of a cyclic group of $2^a$ and an extraspecial 2-group of order $2^{2\ga+1}$. In addition, the group $Z_R(Z(\hat R))$, as a subgroup of $\GL_{2^\ga}(\varepsilon q)$, is $R_{1,0,\ga}$. The determinant of the normalizer of $\tR_{1,0,\ga}$ in $\GL_{2^\ga}(\varepsilon q)$ is determined in \cite[Prop.~3.12]{FLZ21}. From this, we see that the parity of $N_G(R)$ is equal to $\{(0,0)\}$. Finally we let $a=2$ in which situation $R^{4}_{1,0,\ga}=R^{0}_{1,0,\ga+1}$. If $\ga=1$, then a set of generators of $N_G(R)$ can be found in the proof of \cite[(1H)]{An93a}, from which we can check that $\langle F_0\rangle$ acts trivially on $N_G(R)/R$. If $\ga\ge 2$, then $N_G(R)=N'R$ with $N'/R\cong \Omega_{2\ga+2,+}(2)$ and thus by \cite[Prop.~4.6.8]{KL90}, the parity of $N_G(R)$ is equal to $\{(0,0)\}$. This completes the proof.
\end{proof}

\begin{cor}\label{cor:act-field-cla}
Let $B$ be the principal 2-block of one of the groups $\rO_{2n+1}(q)$, $\rO_{2n,\pm}(q)$, $\SO_{2n+1}(q)$, $\SO_{2n,\pm}(q)$, $\Omega_{2n+1}(q)$, $\Omega_{2n,\pm}(q)$ (with odd $q$). Then every element of $\Alp(B)$ is $\langle F_0\rangle$-invariant.
\end{cor}	

\begin{proof}
	Let $V$ be a non-degenerate finite-dimensional orthogonal space over $\mathbb F_q$ (with odd $q$). 

First we suppose that $B$ is the principal 2-block of the group $G=\rO(V)$. The principal weight subgroups of $G$ are listed in Theorem~\ref{subsec:radical-subgp-class}, Proposition \ref{prop:weight-subgp-class}, Remark \ref{rmk:wei-sub-orth}, and Table~\ref{Ta1}. 
By the proof of  Proposition~\ref{prop:act-fie-auto}, in every conjugacy class of principal weight subgroups of $G$, there is a representative $R$ and an element $g\in\Omega(V)$ such that $R$ is $\langle\Ad_g\circ F_0\rangle$-stable.
 According to \cite[\S6]{An93a}, a principal weight subgroup of $G$ provides only one principal weight of $G$. Thus every element of $\Alp(B)$ is $\langle F_0\rangle$-invariant.

Suppose that $B$ is the principal 2-block of the group $X=\Omega(V)$. Let $G=\rO(V)$, and denote by $\tilde B$ the principal 2-block of $G$. Let $(S,\phi)$ be a $B$-weight of $X$. 
Then by the proof of  Proposition~\ref{prop:act-fie-auto}, up to conjugacy, we may assume that $(S,\phi)$ is covered by a $\tilde B$-weight $(R,\varphi)$ of $G$ such that, for some $g\in X$, $R$ is $\langle\Ad_g\circ F_0\rangle$-stable and $\langle\Ad_g\circ F_0\rangle$ acts trivially on $N_G(R)/RZ_G(R)$.
Thus $\langle\Ad_g\circ F_0\rangle$ acts trivially on $\Irr(N_X(R)\mid\varphi)$.
As $\langle\Ad_g\circ F_0\rangle$ acts trivially on $G/X$, every element of the $G$-orbit in $\Alp(B)$ containing $\overline{(S,\phi)}$ is $\langle\Ad_g\circ F_0\rangle$-invariant by Lemma~\ref{lem:act-wei}.
Since $g\in X$, every element of $\Alp(B)$ is $\langle F_0\rangle$-invariant.

If $B$ is the principal 2-block of one of the groups $\SO(V)$, then this assertion follows by the above arguments.
\end{proof}

\section{Radical 2-subgroups of $\tp F_{4}(q)$}\label{S:F4}   

Recall that The group $G:=\tp F_{4}(q)$ has two conjugacy classes of involutions (\cite{Gr91},\cite{HY13}). Write $\sigma_{1}$ 
and $\sigma_{2}$ (\cite{HY13}) for involutions with centralizers of types $\tp A_{1}\tp C_{3}$ and $\tp B_{4}$, respectively. 
Let $R$ be a radical 2-subgroup of $\tp F_{4}(\mathbb{F}_{q})$. Define \[A(R)=\Omega_{1}(Z(R)),\] which is an elementary abelian 
2-subgroup. Define \[B(R)=\{1\}\cup\{x\in A(R):x\sim\sigma_{2}\}.\] From \cite[\S 5]{Yu13}, we know that $B(R)$ is a 
subgroup of $A(R)$. It is clear that $A(R)$ and $B(R)$ are both normalized by $N_{G}(R)$. By Lemma \ref{L:R5}, $R$ 
is a radical 2-subgroup of both $Z_{G}(A(R))$ and $Z_{G}(B(R))$. Let $r$ and $s$ be the ranks of $B(R)$ and $A(R)/B(R)$ 
respectively. By results in \cite[\S 5]{Yu13} and the ``process of reduction modulo $\ell$" \cite[Thm. A.12]{GR98}, we 
know $r\leq 2$ and $s\leq 3$, and the pair $(r,s)$ determines the conjugacy class of $A(R)$ in $\tp F_{4}(\bar{\mathbb{F}}_{q})$. 
As in \cite[\S 5]{Yu13}, write $F_{r,s}$ for an elementary abelian 2-subgroup of $\tp F_{4}(\bar{\mathbb{F}}_{q})$ with 
$\rank B(R)=r$ and $\rank A(R)/B(R)=s$.

\subsection{The case of $r=0$}\label{SS:F4-1}

Assume that $r=0$, i.e., all non-identity elements of $A(R)$ are in the class of $\sigma_{1}$. When $(r,s)=(0,0)$, it is clear 
that $R=1$.  

\subsubsection{The case of $(r,s)=(0,1)$}\label{SS:F4-2}

Assume that $(r,s)=(0,1)$. Then, \[\tp F_{4}(\bar{\mathbb{F}}_{q})^{A(R)}\cong(\Sp(6,\bar{\mathbb{F}}_{q})\times
\Sp(2,\bar{\mathbb{F}}_{q}))/\langle(-I_{6},-I_{2})\rangle\] is connected. By the Lang-Steinberg theorem (\cite{Lang56}), there 
is a unique $\mathbb{F}_{q}$ form of $A(R)$, i.e., there is a unique conjugacy class of such $A(R)$ in $G=\tp F_{4}(q)$. Write 
$H=Z_{G}(A(R))$ and let $H^{0}$ be the derived subgroup of $H$. Then, $H$ is the fixed point subgroup of $F_{q}$ on 
$(\Sp(6,\bar{\mathbb{F}}_{q})\times\Sp(2,\bar{\mathbb{F}}_{q}))/\langle(-I_{6},-I_{2})\rangle$. Thus, 
\[H^{0}\cong(\Sp(6,\mathbb{F}_{q})\times\Sp(2,\mathbb{F}_{q}))/\langle(-I_{6},-I_{2})\rangle,\] and $H$ is generated by $H^{0}$ 
and an element \[\tau=(\diag\{\eta I_{3},\eta^{-1}I_{3}\},\diag\{\eta,\eta^{-1}\}),\] where $\eta\in\mathbb{F}_{q^{2}}$ 
and $\eta^{2}\in\mathbb{F}_{q}^{\times}-(\mathbb{F}_{q}^{\times})^{2}$. 

Let $R'=R\cap H^{0}$. By Lemma \ref{L:R3}, $R'$ is a radical 2-subgroup of $H^{0}$. By Lemmas \ref{L:R4} and \ref{L:R6}, 
$R'=(R_{1}\times R_{2})/\langle(-I_{6},-I_{2})\rangle$, where $R_{1}$ (resp. $R_{2}$) is a radical 2-subgroup of 
$\Sp(6,\mathbb{F}_{q})$ (resp. $\Sp(2,\mathbb{F}_{q})$). As in \S \ref{SS:classical2}, we define the subgroups 
$A'_{i}(R_{j})$ of $R_{j}$ ($i=1,2$, $j=1,2$). It is clear that $A'_{1}(R_{2})=\{\pm{I}_{2}\}$. In any case  
$A'_{1}(R_{1})$ is conjugate to one of the following: $\{\pm{I}_{6}\}$, $\langle I_{4,2},-I_{6}\rangle$, 
$\langle I_{4,2},I_{2,4},-I_{6}\rangle$. Note that $R/R'\cong 1$ or $Z_{2}$. When $A'_{1}(R_{1})\neq\{\pm{I}_{6}\}$, 
it always leads to $A(R)=\Omega_{1}(Z(R))$ having order $>2$. Thus, $A'_{1}(R_{1})=\{\pm{I}_{6}\}$. Since $\frac{6}{2}$ 
and $\frac{2}{2}$ are odd, one further shows that $A'_{2}(R_{1})=\{\pm{I}_{6}\}$ and $A'_{2}(R_{2})=\{\pm{I}_{2}\}$. 
Then, $R_{j}$ ($j=1,2$) is a basic subgroup of $\Sp(n_{j},\mathbb{F}_{q})$, where $n_{1}=6$ and $n_{2}=2$. More precisely, 
\begin{itemize}  
\item When $a=2$ (i.e., when $q^{2}-1$ is not a multiple of 16), the conjugacy class of $R_{j}$ has three possibilities: 
\begin{enumerate}
\item[(1)] $\{\pm{I_{n_{j}}}\}$; 
\item[(2)] ($Z_{4}\cong$) $\GL_{1}(\varepsilon q)_{2}\subset\Sp_{n_{j}}(q)$; 
\item[(3)] ($Q_{8}\cong$)  Sylow 2-subgroup of $\Sp_{2}(q)\subset\Sp_{n_{j}}(q)$. 
\end{enumerate} 
\item When $a>2$ (i.e., when $q^{2}-1$ is a multiple of 16), the conjugacy class of $R_{j}$ has five possibilities: 
\begin{enumerate}  
\item[(1)] $\{\pm{I_{n_j}}\}$; 
\item[(2)] a $Q_{8}$ subgroup; 
\item[(3)] another $Q_{8}$ subgroup; 
\item[(4)] $\GL_{1}(\varepsilon q)_{2}\subset\Sp_{n_{j}}(q)$; 
\item[(5)] Sylow 2-subgroup of $\Sp_{2}(q)\subset\Sp_{n_{j}}(q)$.
\end{enumerate}
\end{itemize}

Note that $R=O_{2}(N_{H}(R'))$ is always a radical 2-subgroup of $H$. However, there may exist other possibilities. To have 
a precise classification, one just calculates $N_{H}(R')$ and classify radical 2-subgroups of $N_{H}(R')/R'$ isomorphic 
to $1$ or $Z_{2}$. 

\begin{rem}
In \S\ref{SS:F4-2}, suppose that $(r,s)=(0,1)$ and
\[
 R'=R\cap H^0=
 (\{\pm I_6\}\times\{\pm I_2\})/
 \langle(-I_6,-I_2)\rangle.
\]
Here $H^0$ is the image of $\Sp_6(q)\times\Sp_2(q)$ in $H$.
For $q=3$, there is no radical $2$-subgroup with these properties.
For $q>3$, there is one conjugacy class, with $R=R'$.
\end{rem}

\subsubsection{The case of $(r,s)=(0,2)$}\label{SS:F4-3}

Assume that $(r,s)=(0,2)$. Then, there are two $\mathbb{F}_{q}$ forms of $A(R)$, i.e., there are two conjugacy classes 
of such $A(R)$ in $G=\tp F_{4}(q)$. 
Write $H=Z_G(A(R))$. In each of the two cases below, let $H^0$ be the image
of the classical groups.

(Inner class) Take \[A(R)=\langle(I_{6},-I_{2}),(\left(\begin{array}{cc}&I_{3}\\-I_{3}&\\\end{array}\right),
\left(\begin{array}{cc}&1\\-1&\\\end{array}\right))\rangle\] in $((\Sp(6,\bar{\mathbb{F}}_{q})\times
\Sp(2,\bar{\mathbb{F}}_{q}))/\langle(-I_{6},-I_{2})$. Then, \[H^{0}\cong(\GL_{3}(\varepsilon q)\times
\GL_{1}(\varepsilon q))/\langle(-I_{3},-1)\rangle,\] and $H$ is generated by $H^{0}$, 
\[\tau_{1}=(\eta\left(\begin{array}{cc}c I_{3}&d I_{3}\\-d I_{3}&c I_{3}\\\end{array}\right),
\eta\left(\begin{array}{cc}c&d\\-d&c\\\end{array}\right))\] and 
\[\theta_{1}=(\left(\begin{array}{cc}a I_{3}&b I_{3}\\b I_{3}&-a I_{3}\\\end{array}\right),
\left(\begin{array}{cc}a&b\\b&-a\\\end{array}\right)),\] where $a,b,c,d\in\mathbb{F}_{q}$ with $a^{2}+b^{2}=-1$ 
and $c^{2}+d^{2}=\eta^{-2}$, and $\tau_1$ is a 2-element. Note that $\tau_{1}$ commutes with $H^{0}$ and 
\[\theta_{1}(X,\lambda)\theta_{1}^{-1}=((X^{t})^{-1},\lambda^{-1}),\quad\forall (X,\lambda)\in\GL_{3}\times\GL_{1}.\] 

(Outer class) Take \[A(R)=\langle(I_{6},-I_{2}),(\left(\begin{array}{cc}&\eta I_{3}\\-\eta^{-1}I_{3}&\\\end{array}\right),
\left(\begin{array}{cc}&\eta\\-\eta^{-1}&\\\end{array}\right))\rangle\] in $((\Sp(6,\bar{\mathbb{F}}_{q})\times
\Sp(2,\bar{\mathbb{F}}_{q}))/\langle(-I_{6},-I_{2})$. Then, \[H^{0}\cong(\GL_{3}(-\varepsilon q)\times
\GL_{1}(-\varepsilon q))/\langle(-I_{3},-1)\rangle,\] and $H$ is generated by $H^{0}$, 
\[\tau_{2}=(\left(\begin{array}{cc}&\eta I_{3}\\-\eta^{-1}I_{3}&\\\end{array}\right),
\left(\begin{array}{cc}&\eta\\-\eta^{-1}&\\\end{array}\right))\] and 
\[\theta_{2}=(\left(\begin{array}{cc}a I_{3}&b\eta^{2}I_{3}\\b I_{3}&-a I_{3}\\\end{array}\right),
\left(\begin{array}{cc}a&b\eta^{2}\\b&-a\\\end{array}\right)),\] where $a,b\in\mathbb{F}_{q}$ with 
$a^{2}+b^{2}\eta^{2}=-1$. Note that $\tau_{2}$ commutes with $H^{0}$ and 
\[\theta_{2}(X,\lambda)\theta_{2}^{-1}=((X^{t})^{-1},\lambda^{-1}),\quad\forall (X,\lambda)\in\GL_{3}\times\GL_{1}.\]  Thus $H^0\unlhd H$ in both cases.

The following statement is clear.  

\begin{lem}\label{L:sigma2-1} 
For any involution $x\in G=\tp F_{4}(q)$ in the conjugacy class of $\sigma_1$, $[G^{x},G^{x}]$ is an index 2 subgroup of 
$G^{x}$. If $A=\langle x,y\rangle$ is a rank 2 elementary abelian 2-subgroup of $G$ with non-identity elements all in 
the class of $\sigma_{1}$, then it is in the inner class if and only if $y\in [G^{x},G^{x}]$.  
\end{lem}  

For the inner class, let $R'=R\cap H^{0}$. By Lemma \ref{L:R3}, $R'$ is a radical 2-subgroup of $H^{0}$. By Lemmas \ref{L:R4} 
and \ref{L:R6}, $R'=(R_{1}\times R_{2})/\langle(-I_{3},-1)\rangle$, where $R_{1}$ (resp. $R_{2}$) is a radical 2-subgroup of 
$\GL_{3}(\varepsilon q)$ (resp. $\GL_{1}(\varepsilon q)$). Like in the case of $(r,s)=(0,1)$, again we have $A'_{1}(R_{1})=
A'_{2}(R_{1})=\{\pm{I}_{3}\}$ and $A'_{1}(R_{2})=A'_{2}(R_{2})=\{\pm{1}\}$. Then, $R_{1}=\GL_{1}(\varepsilon q)_{2}I_{3}$ 
and $R_{2}=\GL_{1}(\varepsilon q)_{2}$. Hence, $R$ is conjugate to 
\[\langle(\GL_{1}(\varepsilon q)_{2}I_{3}\times\GL_{1}(\varepsilon q)_{2})/\langle(-I_{3},-1)\rangle,\tau_{1}\rangle\] or 
\[\langle(\GL_{1}(\varepsilon q)_{2}I_{3}\times\GL_{1}(\varepsilon q)_{2})/\langle(-I_{3},-1)\rangle,\tau_{1},\theta_{1}\rangle.\]  
  
For the outer class, similarly one shows that: $R$ is conjugate to 
\[\langle(\GL_{1}(-\varepsilon q)_{2}I_{3}\times\GL_{1}(-\varepsilon q)_{2})/\langle(-I_{3},-1)\rangle,\tau_{2}\rangle.\]   
  
\subsubsection{The case of $(r,s)=(0,3)$}\label{SS:F4-4}

\begin{lem}\label{L:sigma2-2}
There are two conjugacy classes of rank 3 elementary abelian 2-subgroups $A$ of $G$ with all non-identity elements in the 
conjugacy class of $\sigma_{1}$.
\end{lem} 

\begin{proof}
Define a function $m: A\times A\rightarrow\{\pm{1}\}$. For two elements $x,y\in A$, if $\langle x,y\rangle$ has rank $<2$, define 
$m(x,y)=1$; if $\langle x,y\rangle$ has rank $2$ and it is the inner case as in Subsection \ref{SS:F4-3}, define $m(x,y)=1$; 
otherwise define $m(x,y)=-1$. It is clear that $m$ is symmetric. By Lemma \ref{L:sigma2-1}, $m$ is multiplicative for both 
variables. Then, $\dim_{\mathbb F_2}\ker m=1$ or $3$. With direct calculation in $G^{x}$ ($1\neq x\in A$), one shows that there exists a unique 
conjugacy class of $A$ for either value of $\ker m$. 
\end{proof}

Now assume that $(r,s)=(0,3)$. By Lemma \ref{L:sigma2-2}, there are two $\mathbb{F}_{q}$ forms of $A(R)$. i.e., there are two 
conjugacy classes of such $A(R)$ in $G=\tp F_{4}(q)$. In any case, we have $Z_{G}(A(R))=G'\times A(R)$ with $G'\cong
\SO(3,\mathbb{F}_{q})$. Then, $R=A(R)$.   
  
\begin{lem}\label{P:r=0}
If $R$ is a radical subgroup of $G$ with $B(R)=1$, then $R$ is not a principal weight subgroup of $G$.
\end{lem} 
	
\begin{proof}
Suppose that $R$ is a principal weight subgroup of $G$ with $B(R)=1$. Then $r=0$ and $s=\rank \Omega_1(Z(R))$. By Remark 
\ref{rmk:voe2}, $R$ is also a principal weight subgroup of $Z_{G}(A(R))$. According to Lemma~\ref{lem:center-weisub}, there 
is a Sylow 2-subgroup $S$ of $Z_{G}(A(R))$ such that $Z(S)\subset Z(R)\subset R\subset S$. From the above calculation for 
$Z_{G}(A(R))$, one shows that: in any case, $\Omega_{1}(Z(S))$ contains an element in the conjugacy class of $\sigma_{2}$. 
Hence, $Z(S)\subset Z(R)$ could not hold. This gives a contradiction and thus $R$ is not a principal weight subgroup of $G$. 
\end{proof}

\subsection{The case of $r=1$}\label{SS:Spin9} 

When $r=1$, we have $B(R)\cong Z_{2}$ and $R$ is a radical 2-subgroup of \[Z_{G}(B(R))\cong\Spin_{9,+}(q).\] 
Moreover, as $B(R)$ is normalized by $N_{G}(R)$. We have $N_{G}(R)\subset N_{G}(B(R))=Z_{G}(B(R))$. Therefore, any radical 
2-subgroup $R$ of $Z_{G}(B(R))$ (with the prescribed subgroup $B(R)$) is also a radical 2-subgroup of $G$. Note that the condition of 
$\rank B(R)=1$ poses a constraint on $R$ as a radical 2-subgroup of $Z_{G}(B(R))\cong\Spin_{9,+}(q)$. 

Consider the projection $p:\Spin_{9,+}(q)\rightarrow\Omega_{9,+}(q)$ and inclusion $\Omega_{9,+}(q)\rightarrow\rO_{9,+}(q)$. 
Any radical 2-subgroup $R$ of $\Spin_{9,+}(q)$ is of the form $R=p^{-1}(R'\cap\Omega_{9,+})$ for some radical 2-subgroup 
$R'$ of $\rO_{9,+}(q)$. Any $R'$ is the form $R'=R'_{1}\times\cdots\times R'_{s}$, where $\mathbb{F}_{q}^{9}=
\bigoplus_{1\leq i\leq s} V_{i}$ is an orthogonal decomposition and $R'_{i}$ is a basic subgroup of $I(V_{i})$ 
($1\leq i\leq s$). Except when $\dim V_{i}$ is even and $\disc(V_{i})=1$, we have $R'_{i}=\{\pm{I}\}$. When $\dim V_{i}$ is 
even and $\disc(V_{i})=1$, $R'_{i}$ is classified as in Lemma \ref{L:O2-O8}. Using parities of $R'_{i}$ ($1\leq i\leq s$), 
we can calculate $R'\cap\Omega_{9,+}$ and then $R=p^{-1}(R'\cap\Omega_{9,+})$. We caution that some distinct 
(resp. non-conjugate) $R'$ may lead to equal (resp. conjugate) $R$. After excluding those with $\rank B(R)>1$ and 
clarifying conjugate ones, we reach a classification of radical 2-subgroups of $G$ with $\rank B(R)=1$.  

In Table \ref{Ta2}, we list some radical subgroup of $G$, say $R_1$--$R_{21}$. 

\begin{lem}\label{P:Ta2}
If $R$ is a principal weight subgroup of $G$ with $|B(R)|=2$, then $R$ is conjugate to one of $R_k$ ($1\le k\le 21$) as in Table 
\ref{Ta2}. 
\end{lem}

\begin{proof}
Let $R$ be a principal weight subgroup of $G$ with $|B(R)|=2$. Let $H=Z_G(B(R))$ so that $H\cong\Spin_{9,+}(q)$. By Remark \ref{rmk:voe2}, $R$ is also a principal weight subgroup of $H$.  	
We identify the principal weights of $H$ with the principal weights of $\Omega_{9,+}(q)$, which implies that $p(R)$ is a principal weight subgroup of $\Omega_{9,+}(q)$. According to Remark~\ref{rmk:extension}, there 
is a principal weight subgroup $R'$ of $\rO_{9,+}$ such that $p(R)=R'\cap \Omega_{9,+}$. By Proposition \ref{P:Ta1} and Remark \ref{rmk:wei-sub-orth}, 
we deduce that $R$ is conjugate to one of  $R_1$--$R_{21}$.    
\end{proof}	

\begin{table}[!htb]                   
\centering
\renewcommand\arraystretch{1.3}
\begin{tabular}{|c|c|c|c|c|}    
\hline  
 &  $R'$  &  parity & $\log_{2}|R/\{\pm{1}\}|$ & $\log_{2}|Z(R'\cap\Omega_{9,+})|$  \\  
\hline	
$R_1$ & $R^0_{1,0,0,(3)}\ti R^0_{1,0,0,\emptyset}$ & (1,0) & 11 & 1   \\
$R_2$ & $R^0_{1,1,0,(3)}\ti R^0_{1,0,0,\emptyset}$ & $(Z_2)^2$ &  &   \\
$R_3$ & $R^0_{1,0,0,(1,2)}\ti R^0_{1,0,0,\emptyset}$ & (1,0) &  13 & 1  \\
$R_4$ & $R^0_{1,1,0,(1,2)}\ti R^0_{1,0,0,\emptyset}$ & $(Z_2)^2$  &  &   \\
$R_5$ & $R^4_{1,0,2,\emptyset}\ti R^0_{1,0,0,\emptyset}$ & (1,0) & &  \\
$R_6$ & $R^4_{1,0,1,(1)}\ti R^0_{1,0,0,\emptyset}$ &  (1,0) & &  \\  
$R_7$ & $R^4_{1,0,0,(1,1)}\ti R^0_{1,0,0,\emptyset}$ & $(Z_2)^2$ & $4a+6$ & 1  \\
$R_8$ & $R^4_{1,0,0,(2)}\ti R^0_{1,0,0,\emptyset}$ & $(Z_2)^2$ & $4a+5$ & 1  \\
\hline
$R_9$ & $R^0_{1,0,0,(2)}\ti R^0_{1,1,0,(2)}\ti R^0_{1,0,0,\emptyset}$ & $(Z_2)^2$ & 11 & 2  \\
$R_{10}$ & $R^0_{1,0,0,(2)}\ti R^4_{1,0,1,\emptyset}\ti R^0_{1,0,0,\emptyset}$ & (1,0) & $a+9$ & 2  \\
$R_{11}$ &  $R^0_{1,0,0,(2)}\ti R^4_{1,0,0,(1)}\ti R^0_{1,0,0,\emptyset}$    & $(Z_2)^2$   & $2a+8$ & 2     \\
$R_{12}$ & $R^0_{1,1,0,(2)}\ti R^4_{1,0,1,\emptyset}\ti R^0_{1,0,0,\emptyset}$ & $(Z_2)^2$ & & \\
$R_{13}$ & $R^0_{1,1,0,(2)}\ti R^4_{1,0,0,(1)}\ti R^0_{1,0,0,\emptyset}$ & $(Z_2)^2$ & $2a+8$ & 2  \\
$R_{14}$ & $R^4_{1,0,1,\emptyset}\ti R^4_{1,0,0,(1)}\ti R^0_{1,0,0,\emptyset}$ & $(Z_2)^2$ & $3a+5$ & 2  \\
	 \hline
$R_{15}$ & $R^0_{1,0,0,(2)}\ti R^4_{1,0,0,\emptyset}\ti (R^0_{1,0,0,\emptyset})^3$ & $(Z_2)^2$ & $a+8$ & 4 or 5  \\
$R_{16}$ & $R^0_{1,1,0,(2)}\ti R^4_{1,0,0,\emptyset}\ti (R^0_{1,0,0,\emptyset})^3$ & $(Z_2)^2$ & $a+8$ & 4  \\
$R_{17}$ & $R^4_{1,0,1,\emptyset}\ti R^4_{1,0,0,\emptyset}\ti (R^0_{1,0,0,\emptyset})^3$ & $(Z_2)^2$ & $2a+5$ & 4 or 5  \\
$R_{18}$ & $R^4_{1,0,0,(1)}\ti R^4_{1,0,0,\emptyset}\ti (R^0_{1,0,0,\emptyset})^3$ & $(Z_2)^2$ & $3a+5$ & 4  \\
\hline
$R_{19}$ & $(R^4_{1,0,0,\emptyset})^3\ti (R^0_{1,0,0,\emptyset})^3$ & $(Z_2)^2$ & $3a+4$ & 5  \\
\hline
$R_{20}$ & $R^4_{1,0,0,\emptyset}\ti R^0_{1,0,0,\emptyset} \ti (R^{0}_{1,1,0,\emptyset})^6$ & $(Z_2)^2$ & $a+6$ & 6  \\
\hline
$R_{21}$ & $(R^0_{1,0,0,\emptyset})^3\ti (R^{0}_{1,1,0,\emptyset})^6$  & $(Z_2)^2$ & 7 & 7  \\
\hline
\end{tabular} 
  \vspace{0.8em}
\begin{minipage}{0.8\linewidth}
 \footnotesize
When  $R\in\{ R_{15},R_{17}\}$, 
$\log_{2}|Z(R'\cap\Omega_{9,+})|=5$ or $4$ according as $a=2$ or $a>2$. 
\end{minipage}
\caption{Principal weight subgroups of $\Spin_{9,+}$}\label{Ta2}
\end{table} 

Among radical 2-subgroups in Table \ref{Ta2}, $R_2$, $R_4$, $R_5$, $R_6$, $R_{12}$ lead to subgroups $R$ with $B(R)\cong Z_{2}^{2}$, 
so should be excluded. All others lead to subgroups $R$ with $B(R)\cong Z_{2}$. Let $R'$ be one of $R_2$, $R_4$, $R_5$, $R_6$, 
$R_{12}$. Write $R=p^{-1}(R'\cap\Omega_{9,+})$. 
\begin{enumerate}
\item For $R'=R_{2}$, $R^0_{1,0,0,\emptyset}=\{\pm{1}\}$ has parity $(1,0)$ and $R^0_{1,1,0,(3)}$ has parity $(0,1)$. Then, 
$R'\cap\Omega_{9,+}(q)=R^0_{1,1,0,(3)}\cap\Omega_{8,+}(q)$. Therefore, $B(R)\cong(Z_{2})^{2}$. 
\item For $R'=R_{4}$, $R^0_{1,0,0,\emptyset}=\{\pm{1}\}$ has parity $(1,0)$ and $R^0_{1,1,0,(1,2)}$ has parity $(0,1)$. Then, 
$R'\cap\Omega_{9,+}(q)=R^0_{1,1,0,(1,2)}\cap\Omega_{8,+}(q)$. Therefore, $B(R)\cong(Z_{2})^{2}$. 
\item For $R'=R_{5}$, $R_{1,0,0,\emptyset}^{0}=\{\pm{1}\}$ has parity $(1,0)$ and $R_{1,0,2,\emptyset}^{4}$ has parity $(0,0)$. 
Then, $R'\cap\Omega_{9,+}(q)=R_{1,0,2,\emptyset}^{4}\cap\Omega_{8,+}(q)$. Therefore, $B(R)\cong(Z_{2})^{2}$. 
\item For $R'=R_{6}$, $R^0_{1,0,0,\emptyset}=\{\pm{1}\}$ has parity $(1,0)$ and $R_{1,0,1,(1)}^{4}$ has parity $(0,0)$. Then, 
$R'\cap\Omega_{9,+}(q)=R_{1,0,1,(1)}^{4}\cap\Omega_{8,+}(q)$. Therefore, $B(R)\cong(Z_{2})^{2}$.  
\item For $R'=R_{12}$, $R^0_{1,0,0,\emptyset}=\{\pm{1}\}$ has parity $(1,0)$, $R_{1,1,0,(2)}^{0}$ has parity $(0,1)$ and 
$R_{1,0,1,\emptyset}^{4}$ has parity $(0,0)$. Then, $R'\cap\Omega_{9,+}(q)=(R_{1,1,0,(2)}^{0}\times R_{1,0,1,\emptyset}^{4})
\cap\Omega_{8,+}(q)$. Therefore, $B(R)\cong(Z_{2})^{2}$.    
\end{enumerate} 

\begin{lem}\label{L:Spin9-conjugacy}
Let $R'$ be one of the weight subgroups of $\rO_{9,+}$ in Table \ref{Ta2}, except
$R_2,R_4,R_5,R_6,R_9,R_{12}$. When $a=2$, assume further that
$R'\notin\{R_{11},R_{14},R_{15},R_{17}\}$.
Write $R''=R'\cap\Omega_{9,+}$. Then $R'=O_2(N_{\rO_{9,+}}(R''))$.
In particular, the conjugacy class of $R''$ determines that of $R'$.
\end{lem}

\begin{proof}
For each weight subgroup $R'$ of $\rO_{9,+}$ as in Lemma \ref{L:Spin9-conjugacy}, we can calculate  
$R''=R'\cap\Omega_{9,+}$ from parities of elements of $R'$. Then, we can calculate $Z(R'')$ and $N_{\rO_{9,+}}(Z(R''))$. 
It is clear that $N_{\rO_{9,+}}(R'')\subset N_{\rO_{9,+}}(Z(R''))$. Finally, we can calculate $N_{\rO_{9,+}}(R'')$ and 
see that $R'=O_{2}(N_{\rO_{9,+}}(R''))$.  
\end{proof}

\subsection{The case of $r=2$}\label{SS:Spin8}

When $r=2$, $R$ is a radical 2-subgroup of \[Z_{G}(B(R))\cong\Spin(V)\cong\Spin_{8,+}(q),\] where $V$ is an 8-dimensional 
orthogonal space over $\mathbb{F}_{q}$ with discriminant 1. Consider the projection $p:\Spin(V)\rightarrow\Omega(V)$. Any 
radical 2-subgroup $R$ of $\Spin(V)$ is of the form $R=p^{-1}(R'\cap\Omega(V))$ for some radical 2-subgroup $R'$ of $\rO(V)$. 
Note that \[N_{G}(B(R))/Z_{G}(B(R))\cong\fS_{3}.\]  

Choose an orthogonal basis $\{e_{1},\dots,e_{8}\}$ of $V$ with $(e_{i},e_{i})=1$ ($1\leq i\leq 8$). Define 
\[z:=e_{1}\cdots e_{8}\in\Spin(V).\] Then, $Z(\Spin(V))=\{1,-1,z,-z\}\cong Z_{2}\times Z_{2}$. We have $p(z)=-I\in\Omega(V)$. 
Note that field automorphisms and diagonal automorphisms fix $-1,z,-z$, and triality (graph automorphism) permutes them. There 
are four involution classes in $P\Omega(V)$ with representatives $y_{1},y_{2},y_{3},y_{4}$ as follows: 
\[y_{1}:=[I_{2,6}],\quad y_{2}:=[I_{4,4}],\quad y_{3}:=[J_{4}],\quad y_{4}:=\diag\{J_{1},J_{3}\}.\] Let  
\[x_{1}=e_{1}e_{2},\] \[x_{2}=e_{1}e_{2}e_{3}e_{4},\]  \[x_{3}=\frac{1+e_{1}e_{5}}{\sqrt{2}}\frac{1+e_{2}e_{6}}{\sqrt{2}}\frac{1+e_{3}e_{7}}{\sqrt{2}}\frac{1+e_{4}e_{8}}{\sqrt{2}},\]
\[x_{4}=\frac{1+e_{1}e_{2}}{\sqrt{2}}\frac{1+e_{3}e_{6}}{\sqrt{2}}\frac{1+e_{4}e_{7}}{\sqrt{2}}\frac{1+e_{5}e_{8}}{\sqrt{2}}.\] 
Then, $x_{i}\in\Spin(V)$ is a pre-image of $y_{i}$ ($i=1,2,3,4$). We have $x_{1}^{2}=-1$, $x_{3}^{2}=z$, $x_{4}^{2}=-z$ 
and $x_{2}^{2}=1$. Note that, triality stabilizes the conjugacy class containing $y_{2}$ and permutes three conjugacy classes 
containing $y_{1}$, $y_{3}$, $y_{4}$, respectively. Let $\pi:\rO(V)\rightarrow\PO(V)$ be the projection map. Define 
$\tilde{A}(R)$ to be the pre-image of $\Omega_{1}(Z(\pi(p(R))))$.     

\begin{lem}\label{P:Ta3}
	If $R$ is a principal weight subgroup of $G$ with $\rank B(R)=2$, then $R$ is conjugate to one of $R_k$ ($22\le k\le 37$) as in Table 
	\ref{Ta3}. 
	\end{lem}

	\begin{proof}
Similar to the proof of Lemma \ref{P:Ta2}.
	\end{proof}	

\begin{table}[!htb]
	\centering
	\renewcommand\arraystretch{1.3}
	\begin{tabular}{|c|c|c|c|c|}  
		\hline  
        & $R'$ &  parity & $\log_{2}|R/Z(\Spin(V))|$ & $\log_{2}|\tilde{A}(R)/Z(\Spin(V))|$   \\
	\hline	
	    $R_{22}$	 & $R^0_{1,0,0,(3)}$   & (1,0)  & 9 & 3   \\
	    $R_{23}$		 & $R^0_{1,1,0,(3)}$  & (0,1)  & 9 & 3      \\
	    $R_{24}$		 & $R^0_{1,0,0,(1,2)}$  & (1,0)   & 11 & 1     \\
     $R_{25}$      	 & $R^0_{1,1,0,(1,2)}$ &  (0,1)   & 11  & 1  \\
	   $R_{26}$	 & $R^4_{1,0,2,\emptyset}$  & 0  & $a+4$ & 6 or 5       \\
	   $R_{27}$  	 & $R^4_{1,0,1,(1)}$    & 0 & $2a+6$ & 1   \\
	   $R_{28}$     & $R^4_{1,0,0,(1,1)}$ &   $(Z_2)^2$   & $4a+4$ & 1       \\
	   $R_{29}$  	 &  $R^4_{1,0,0,(2)}$  &  $(Z_2)^2$ & $4a+3$ & 2    \\ 
			 \hline
    $R_{30}$   &           $R^0_{1,0,0,(2)}\ti R^0_{1,1,0,(2)}$    & $(Z_2)^2$  & 9  & 1    \\
   $R_{31}$    &          $R^0_{1,0,0,(2)}\ti R^4_{1,0,1,\emptyset}$    &  (1,0) & $a+7$ & 1    \\
    $R_{32}$       &  $R^0_{1,0,0,(2)}\ti R^4_{1,0,0,(1)}$    &  $(Z_2)^2$ & $2a+6$ & 1    \\
     $R_{33}$   &         $R^0_{1,1,0,(2)}\ti R^4_{1,0,1,\emptyset}$ & (0,1)   & $a+7$  & 1    \\
    $R_{34}$    &         $R^0_{1,1,0,(2)}\ti R^4_{1,0,0,(1)}$ &   $(Z_2)^2$    & $2a+6$ & 1     \\
    $R_{35}$    &          $R^4_{1,0,1,\emptyset}\ti R^4_{1,0,0,(1)}$  &  $(Z_2)^2$  & $3a+3$ & 1      \\
	 \hline
   $R_{36}$ 	 & $R^4_{1,0,0,\emptyset}\ti (R^{0}_{1,0,0,\emptyset})^6$  & $(Z_2)^2$   & $a+4$ &  6 or 5    \\
		 \hline
    $R_{37}$ 	 & $R^4_{1,0,0,\emptyset}\ti (R^{0}_{1,1,0,\emptyset})^6$   &  $(Z_2)^2$    & $a+4$ &  6 or 5      \\
		\hline	 
	\end{tabular} 
      \vspace{0.8em}
\begin{minipage}{0.8\linewidth}
 \footnotesize
When  $R\in\{ R_{26},R_{36},R_{37}\}$, 
$\log_{2}|\tilde{A}(R)/Z(\Spin(V))|=6$ or $5$ according as $a=2$ or $a>2$. 
\end{minipage}
\caption{Principal weight subgroups of $\Spin_{8,+}$}\label{Ta3}
\end{table}

The group $\tilde{A}(R)$ has the following characterization.  

\begin{lem}\label{L:Spin8-A}
An element $x$ of $R$ is contained in $\tilde{A}(R)$ if and only if \[yxy^{-1}x^{-1}\in\{1,-1,z,-z\} (\forall y\in R)\textrm{ and }  
x^{2}\in\{1,-1,z,-z\}.\]    
\end{lem}

Let $\Gamma\cong \fS_{4}$ be the subgroup of $\Aut(Z_{G}(B(R)))/\Int(Z_{G}(B(R)))$ generated by diagonal automorphisms and graph 
automorphism. We identify $\Gamma$ with $\fS_{4}$, identify \[\Gamma\supset N_{G}(B(R))/Z_{G}(B(R))\] with the subgroup 
\[S=\langle(12),(13)\rangle\cong \fS_{3},\] identify $\Gamma\supset\rO(V)/\Omega(V)$ with the subgroup  
\[E=\langle(12),(34)\rangle\cong(Z_{2})^{2}.\] Let \[\Gamma\supset N:=\langle(12)(34),(13)(24)\rangle\cong(Z_{2})^{2}\] and 
\[\Gamma\supset D:=\langle(12),(34),(13)(24)\rangle\cong D_{8}.\] Then, $N$ is a normal subgroup of $\Gamma$ and $D$ is a 
Sylow 2-subgroup of $\Gamma$. 

\begin{lem}\label{L:D4}
Let $\sigma\in D-E$. If $R$ is a radical 2-subgroup of $\Spin(V)$ (resp. $R$ appears in Table \ref{Ta3}), then so is 
$\sigma(R)$ (resp. $\sigma(R)$ is conjugate to one in Table \ref{Ta3}). 
For any $x\in\rO(V)$, the parity of $\sigma(x)$ is
$(t_{2},t_{1})$ if the parity of $x$ is $(t_{1},t_{2})$, where $t_{1},t_{2}\in\{0,1\}$.
\end{lem} 

\begin{proof}
The first statement is clear. Choose a generator $\delta_{0}$ of $\mathbb{F}_{q}^{\times}$ and choose an element 
$t\in\mathbb{F}_{q^{2}}$ with $t^{2}=\delta_{0}$. We may realize $\sigma$ as $\sigma=\Ad(tY)$, where $Y\in\GL(V)$ with 
\[(Yu,Yv)=\delta_{0}^{-1}(u,v),\quad\forall u,v\in V.\] For any $v\in V$ with $(v,v)\neq 0$, we have 
$(Yv,Yv)^{\frac{q-1}{2}}=\delta_{0}^{-\frac{q-1}{2}}(v,v)^{\frac{q-1}{2}}=-(v,v)^{\frac{q-1}{2}}$. This implies the 
second statement of this lemma. 
\end{proof}

We aim to determine $S$ orbits on conjugacy classes of weight subgroups in Table \ref{Ta3} and calculate stabilizers. To achieve 
this goal, we consider the action of the larger group $\Gamma$, find representatives of orbits and calculate stabilizers. Then, 
the analysis of double cosets $S\backslash\Gamma/\Gamma_{[R]}$ leads to our goal, where $[R]$ means the $Z_{G}(B(R))$ conjugacy 
class containing $R$; write $\Gamma_{[R]}$ (resp. $S_{[R]}$) for the stabilizer of $\Gamma$ (resp. $S$) at $[R]$. 

\begin{prop}\label{P:weight-r2}
The following are representatives of $S$ orbits of weight subgroups in Table \ref{Ta3} and their stabilizers in $\Gamma$ and 
$S$: \begin{enumerate} 
\item[(a1-1)] $R_{36}$: $\Gamma_{[R]}=E$, $\Gamma\cdot[R]$ consists of all conjugacy classes from $R_{26}$, $R_{36}$, $R_{37}$; 
$S_{[R]}=\langle(12)\rangle$, $S\cdot[R]$ consists of a conjugacy class from $R_{36}$ and two conjugacy classes from $R_{26}$.  
\item[(a1-2)] $R_{37}$: $\Gamma_{[R]}=E$, $\Gamma\cdot[R]$ consists of all conjugacy classes from $R_{26}$, $R_{36}$, $R_{37}$; 
$S_{[R]}=\langle(12)\rangle$, $S\cdot[R]$ consists of a conjugacy class from $R_{37}$ and two conjugacy classes from $R_{26}$.  
\item[(a2-1)] $R_{32}$: $\Gamma_{[R]}=E$, $\Gamma\cdot[R]$ consists of all conjugacy classes from $R_{27}$, $R_{32}$, $R_{34}$; 
$S_{[R]}=\langle(12)\rangle$, $S\cdot[R]$ consists of a conjugacy class from $R_{32}$ and two conjugacy classes from $R_{27}$.   
\item[(a2-2)] $R_{34}$: $\Gamma_{[R]}=E$, $\Gamma\cdot[R]$ consists of all conjugacy classes from $R_{27}$, $R_{32}$, $R_{34}$; 
$S_{[R]}=\langle(12)\rangle$, $S\cdot[R]$ consists of a conjugacy class from $R_{34}$ and two conjugacy classes from $R_{27}$. 
\item[(b1-1)] $R_{22}$: $\Gamma_{[R]}=S_{[R]}=S$, $\Gamma\cdot[R]$ consists of all conjugacy classes from $R_{22}$, $R_{23}$; 
$S\cdot[R]$ consists of a conjugacy class from $R_{22}$.   
\item[(b1-2)] $R_{23}$: $\Gamma_{[R]}=\langle(13),(34)\rangle$, $\Gamma\cdot[R]$ consists of all conjugacy classes from $R_{22}$, 
$R_{23}$; $S_{[R]}=\langle(13)\rangle$, $S\cdot[R]$ consists of a conjugacy class from $R_{22}$ and two conjugacy classes from 
$R_{23}$. 
\item[(b2-1)] $R_{24}$: $\Gamma_{[R]}=S_{[R]}=S$, $\Gamma\cdot[R]$ consists of all conjugacy classes from $R_{24}$, $R_{25}$; 
$S\cdot[R]$ consists of a conjugacy class from $R_{24}$.   
\item[(b2-2)] $R_{25}$: $\Gamma_{[R]}=\langle(13),(34)\rangle$, $\Gamma\cdot[R]$ consists of all conjugacy classes from $R_{24}$, 
$R_{25}$; $S_{[R]}=\langle(13)\rangle$, $S\cdot[R]$ consists of a conjugacy class from $R_{24}$ and two conjugacy classes from 
$R_{25}$.   
\item[(b3-1)] $R_{31}$ ($a>2$): $\Gamma_{[R]}=S_{[R]}=S$, $\Gamma\cdot[R]$ consists of all conjugacy classes from $R_{31}$, $R_{33}$; 
$S\cdot[R]$ consists of a conjugacy class from $R_{31}$.   
\item[(b3-2)] $R_{33}$ ($a>2$): $\Gamma_{[R]}=\langle(13),(34)\rangle$, $\Gamma\cdot[R]$ consists of all conjugacy classes from $R_{31}$, 
$R_{33}$; $S_{[R]}=\langle(13)\rangle$, $S\cdot[R]$ consists of a conjugacy class from $R_{31}$ and two conjugacy classes from 
$R_{33}$. 
\item[(c1)] $R_{28}$: $\Gamma_{[R]}=\Gamma$ and $S_{[R]}=S$, $\Gamma\cdot[R]=S\cdot[R]$ consists of a conjugacy class from $R_{28}$. 
\item[(c2)] $R_{29}$: $\Gamma_{[R]}=\Gamma$ and $S_{[R]}=S$, $\Gamma\cdot[R]=S\cdot[R]$ consists of a conjugacy class from $R_{29}$. 
\item[(c3)] $R_{30}$: $\Gamma_{[R]}=\Gamma$ and $S_{[R]}=S$, $\Gamma\cdot[R]=S\cdot[R]$ consists of a conjugacy class from $R_{30}$. 
When $a=2$, $R_{30}$, $R_{31}$, $R_{33}$ and $R_{35}$ are conjugate.
\item[(c4)] $R_{35}$ ($a>2$): $\Gamma_{[R]}=\Gamma$ and $S_{[R]}=S$, $\Gamma\cdot[R]=S\cdot[R]$ consists of a conjugacy class from $R_{35}$. 
\end{enumerate}
\end{prop}

\begin{proof}
Note that $a\geq 2$. By comparing orders of $R/Z(\Spin(V))$ and $\tilde A(R)/Z(\Spin(V))$, we divide possible $\Gamma$ conjugate radical  
subgroups in Table \ref{Ta3} into groups with few exceptions: \begin{enumerate}  
\item[(a1)]: $R_{26}$, $R_{36}$, $R_{37}$.  
\item[(a2)]: $R_{27}$, $R_{32}$, $R_{34}$. 
\item[(b1)]: $R_{22}$, $R_{23}$. 
\item[(b2)]: $R_{24}$, $R_{25}$. 
\item[(b3)]: $R_{31}$, $R_{33}$. 
\item[(c1)]: $R_{28}$. 
\item[(c2)]: $R_{29}$. 
\item[(c3)]: $R_{30}$. 
\item[(c4)]: $R_{35}$.  
\end{enumerate} 
When $a=2$, the subgroups in groups (b3), (c3) and (c4) are conjugate.
In this case, we consider these subgroups in (c3).
The few exceptional cases are: (1) when $a=3$, those in groups (a2) and (c4) have the same orders;
(2) when $a=4$, those in groups (b2) and (b3) have the same orders.
These exceptions are distinguished in the following argument.

(a) Take $R=R_{36}$ (resp. $R_{32}$) in group (a1) (resp. group (a2)). By Lemma \ref{L:D4}, we get $D_{[R]}=E$. Then, 
$\Gamma_{[R]}=E$ since no other subgroup of $\Gamma$ having intersection with $D$ being equal to $E$. Then, $\Gamma\cdot[R]$ 
consists of 6 conjugacy classes. Analyzing double cosets in $S\backslash\Gamma/E$, we get the conclusion. 

(b) Take $R$ in one of the groups (b1)-(b3), with $a>2$ in (b3).
By Lemma \ref{L:D4}, we get $D_{[R]}=E_{[R]}=\langle(12)\rangle$, $\langle(34)\rangle$ 
or $E$. Using the argument in (a), one shows that $D_{[R]}=E_{[R]}=E$ could not happen. We further show that 
$\Gamma_{[R]}\sim S$. Otherwise, $\Gamma_{[R]}$ has order 2 and there should exist at least $|\Gamma/\Gamma_{[R]}|=12$ 
conjugacy classes of radical 2-subgroups in Table \ref{Ta3} having the same order as $R$, which is not the case. Then, 
$\Gamma\cdot[R]$ consists of 4 conjugacy classes. Analyzing double cosets in $S\backslash\Gamma/S$, we get the conclusion.  

(c) Take $R$ in one of the groups (c1)-(c4) , with $a>2$ in (c4).
 By Lemma \ref{L:D4}, we get $D_{[R]}=D$. Moreover, we have $\Gamma_{[R]}=\Gamma$ since 
$\sigma([R])$ ($\sigma\in\Gamma-D$) could not be a conjugacy class in any other group. Then, $\Gamma\cdot[R]=[R]$.  
\end{proof}

\begin{lem}\label{L:Spin8-conjugacy}
For any weight subgroup $R'$ of $\rO_{8,+}$ in Table \ref{Ta3}, write
$R''=R'\cap\Omega_{8,+}$. Then, the parity of $N_{\rO_{8,+}}(R'')$ is the same
as that of $R'$, except when $a=2$ and
$R'=R^0_{1,\alpha,0,(2)}\times R^4_{1,0,1,\emptyset}$ with $\alpha=0,1$.
In this case, it is equal to $(Z_2)^2$.
\end{lem}

\begin{proof}
This is implied by the formula of $\Gamma_{[R]}$ shown above. Alternatively, one could perform direct computation to reach this 
statement.  
\end{proof}

With Proposition \ref{P:weight-r2}, we classify principal weight subgroups $R$ of $G$ with $\rank B(R)=2$. Similar method applies 
to classify radical 2-subgroups of $G$ with $\rank B(R)=2$ as well. Again, the crucial step is to classify $S$ orbits of radical 
2-subgroups and calculate their stabilizers in $\Gamma$ and $S$. Since there are many more candidates of radical 2-subgroups, the required computing load is heavier.


\section{The inductive BAW condition for groups of types $\tp B_n$, $\tp D_n$, $^2\tp D_n$ and $\tp F_4$}\label{S:BAW} 

In this section, we investigate the inductive BAW condition for 2-blocks of quasi-simple groups of types $\tp B_n$, $\tp D_n$, $^2\tp D_n$ and $\tp F_4$ and prove Theorem~\ref{mainthm:BAW-BD} and Theorem~\ref{mainthm:BAW-F4}.

Let $\bG$ be a simple algebraic group over $\overline\bbF_q$ where $q$ is prime power, and $F:\bG\to\bG$ be a 
Frobenius endomorphism endowing an $\bbF_q$-structure on $\bG$. Recall that an element $t\in\bG$ is called 
\emph{quasi-isolated} in $\bG$ if the centralizer $\C_{\bG}(t)$ is not contained in a proper Levi subgroup of $\bG$. Quasi-isolated semisimple
elements have been classified in \cite{Bo05}. 

Let $\bG^*$ be the Langlands dual of $\bG$ with corresponding Frobenius endomorphism also denoted by $F$. For a semisimple 
element $s$ of ${\bG^*}^F$, we denote the Lusztig series corresponding to $s$ by $\cE(\bG^F,s)$. Every element 
$s\in Z({\bG^*}^F)$ defines a linear character of $\bG^F$ (see \cite[(8.19)]{CE04}), which is denoted by $\hat s$. For a prime 
$p$ (with $p\nmid q$) and a semisimple $p'$-element $s$ of ${\bG^*}^F$, we denote by $\cE_p(\bG^F,s)$ the union of the Lusztig 
series $\cE(\bG^F, st)$, where $t$ runs through semisimple $p$-elements of ${\bG^*}^F$ commuting with $s$. By \cite{BM89}, the 
set $\cE_p(\bG^F,s)$ is a union of $p$-blocks of $\bG^F$. If a semisimple $p'$-element $s$ is quasi-isolated, then the blocks 
in $\cE_p(\bG^F,s)$ are called \emph{quasi-isolated blocks} of $\bG^F$. The quasi-isolated blocks of exceptional groups have 
been classified in \cite{KM13}. 

The following  comes from the reduction of the inductive BAW condition to quasi-isolated blocks; see Hypothesis~5.5 and Theorem~5.7 of \cite{FLZ22}.
This has been established in \cite{FW26}, using the $A(\infty)$ property established by Cabanes and Sp\"ath (see  \cite[Thm.~B]{CS19} and \cite[Thm.~A]{Sp23}).

\begin{thm}[{\cite[Thm.~1.3]{FW26}}]\label{A-infinity}
Suppose that $\bG$ is simple and of simply-connected type,  $\bG\hookrightarrow \tilde\bG$ is a regular embedding, and $D$ is the group consisting of field and graph automorphisms of $\bG^F$. 
In every $\tilde\bG$-orbit of $\IBr(\bG^F)$ there exists a Brauer character $\psi\in\IBr(\bG^F)$ such that $(\tilde\bG\rtimes D)_\psi=\tilde\bG_\psi\rtimes D_\psi$, and $\psi$ extends to $\bG^F\rtimes D_\psi$.
\end{thm}

\subsection{Groups of type $\tp B_n$, $\tp D_n$ and $^2\tp D_n$}

Let $q$ be an odd prime power and let $V$ be a non-degenerate finite-dimensional orthogonal space over $\mathbb F_q$. 
Write $G=\Spin(V)$, $I=\rO(V)$, $H=\SO(V)$ and $w=\dim(V)$.
If $w$ is odd, we may assume that $\disc(V)=+$.
Let $F$ be a Frobenius endomorphism such that $G=\bG^F$ and $H=\bH^F$, where $\bG$ and $\bH$ are the spin group and special orthogonal group on $\mathbf V$ respectively.
Here $\mathbf V$ is a non-degenerate $w$-dimensional orthogonal space over $\overline{\mathbb F}_q$. 

We recall the parametrisation of unipotent conjugacy classes of $\bH$ (cf. \cite{Wa63}).
To every unipotent class of $\bH$ we have a corresponding partition $\la\vdash w$ such that any even part occurs an even number of times in $\la$.
Denote the set of such partitions $\la$ by $\mathscr P_w$.
Let $c_i$ denote the number of parts of $\la$ equal to $i$.
Let $a(\la)$ be the number of odd $i$ with $c_i>0$, and $\kappa(\la)$ be the largest number in the set $\{c_i\mid \textrm{$i$ is odd and $c_i>0$}\}$. 
In particular, $a(\la)>0$ and $\ka(\la)>0$ if $w$ is odd.
Write $\mathbf C_\la$ for the unipotent class of $\mathbf H$ parametrised by $\la$, where $\la$ is counted twice when $a(\la)=0$. Sometimes we also write $\mathbf C_\la$ as $\mathbf C_\la^\pm$ for the situation $a(\la)=0$ to distinguish those two conjugacy classes of $\bH$.
Note that $\mathbf C_\la$ is $I$-stable if $a(\la)>0$, while $\mathbf C_\la^+$ and $\mathbf C_\la^-$ are fused by $I$ when $a(\la)=0$.

For the unipotent conjugacy classes of $H$, we have that ${\mathbf C_{\la}}^F=\emptyset$ if and only if $\ty(V)=-$ and $a(\la)=0$ (see for example \cite[\S3 and \S7]{LS12}). 
Define $b(\la)=a(\la)-1$ or $0$ according as $a(\la)>0$ or $a(\la)=0$.
Moreover, ${\mathbf C_{\la}}^F$ consists of one conjugacy classes in $H$ if $\ty(V)=+$ and $a(\la)=0$, while ${\mathbf C_{\la}}^F$ consists of $2^{b(\la)}$ conjugacy classes in $H$ if $a(\la)>0$.

Denote by $u(X)$ the number of unipotent conjugacy classes of a finite reductive group $X$.
Now we determine $u(I)$, which also follows by \cite[(2.6.17)]{Wa63}.
First, $u(\rO_w^+(q))+u(\rO_w^-(q))$ is the coefficient of $t^{w}$ in the generating function
\begin{equation}\label{equ-5.1}
	(1+2\sum_{k=1}^\infty t^k)(\sum_{k=0}^\infty t^{4k})(1+2\sum_{k=1}^\infty t^{3k})(\sum_{k=0}^\infty t^{8k})\cdots
	=\prod_{k=1}^{\infty}\frac{(1+t^{2k-1})^2}{(1-t^{2k})}.
	\addtocounter{thm}{1}\tag{\thethm}
\end{equation}
Next, $u(\rO_w^+(q))-u(\rO_w^-(q))$ is the coefficient of $t^{w}$ in 
\begin{equation}\label{equ-5.1-1}
(\sum_{k=0}^\infty t^{4k})(\sum_{k=0}^\infty t^{8k})\cdots
	=\prod_{k=1}^{\infty}\frac{1}{1-t^{4k}}.
	\addtocounter{thm}{1}\tag{\thethm}
\end{equation}
In particular, $u(\rO_w^+(q))\ne u(\rO_w^-(q))$ occurs only when $4\mid w$.

Denote by $\omega_0(X)$ the number of conjugacy classes of principal weights of a finite group~$X$.
By the description of principal weights in \S\ref{S:radical}, we determine $\omega_0(I)$.
This is also calculated in \cite{AC95}, but for later use we still give a sketch.
By the construction of weights in \cite[\S6]{An93a}, a radical 2-subgroup $R$ of $I$ is a principal weight 2-subgroup if and only if $R$ is described as in Proposition~\ref{prop:weight-subgp-class} and Remark~\ref{rmk:wei-sub-orth}, and every radical 2-subgroup of $I$ affords at most one weight for a given block.
The associated basic subgroups are listed in Table~\ref{Ta1}.
In addition, the multiplicity of each basic subgroup is a triangular number.
Therefore, using the $2$-core tower structure of partitions (cf. \cite[(1A)]{AF90}), the principal 2-weights of $I$ are in bijection with the tuples $(\la_1,\la_2,\la_3,\la_4,\ka_{+},\ka_{-},\ka)$ such that $\la_1$, $\la_2$, $\la_3$, $\la_4$ are partitions, $\ka_+$, $\ka_-$, $\ka$ are $2$-cores such that $\disc(V)=(-1)^{|\kappa_-|}$ and 
\begin{equation}\label{equ-5.2}
4(|\la_1|+|\la_2|+|\la_3|+|\la_4|)+|\ka_+|+|\ka_-|+2|\ka|=w.	
\addtocounter{thm}{1}\tag{\thethm}
\end{equation}
So $\omega_0(\rO_{w,+}(q))+\omega_0(\rO_{w,-}(q))$ is the coefficient of $t^{w}$ in 
\begin{equation}\label{equ-5.3}
	(\sum_{k=0}^\infty t^{4k})^4(\sum_{k=0}^\infty t^{8k})^4\cdots\vartheta(t)^2\vartheta(t^2)=\prod_{k=1}^\infty \frac{1+t^k}{(1-t^k)(1+t^{2k})^2}.
	\addtocounter{thm}{1}\tag{\thethm}
\end{equation}
Here $\vartheta(t)=\sum\limits_{k=1}^\infty t^{\frac{k(k-1)}{2}}$, and we use Jacobi’s identity $\vartheta(t)=\prod\limits_{k=1}^\infty (1+t^k)(1-t^{2k})$.
Now we determine $\omega_0(\rO_{w,+}(q))-\omega_0(\rO_{w,-}(q))$. Since $\omega_0(\rO_{w,+}(q))=\omega_0(\rO_{w,-}(q))$ when $w$ is odd, we may assume that $w$ is even.
Note that $\sum\limits_\mu t^{|\mu|}=\frac{\vartheta(t)+\vartheta(-t)}{2}$ and  $\sum\limits_{\mu'} t^{|\mu'|}=\frac{\vartheta(t)-\vartheta(-t)}{2}$
where $\mu$, $\mu'$ run through all $2$-cores such that $|\mu|$ and $|\mu'|$ are even and odd respectively.
Thus $\omega_0(\rO_{w,+}(q))-\omega_0(\rO_{w,-}(q))$ is the coefficient of $t^{w}$ in the generating function
\begin{align*}
(\sum_{k=0}^\infty t^{4k})^4(\sum_{k=0}^\infty t^{8k})^4\cdots \vartheta(t^2)((\frac{\vartheta(t)+\vartheta(-t)}{2})^2-(\frac{\vartheta(t)-\vartheta(-t)}{2})^2)
\end{align*}
and thus using identities $\vartheta(t)=\prod\limits_{k=1}^\infty\frac{(1-t^{2k})^2}{1-t^k}$, $\prod\limits_{k=1}^\infty \frac{1}{1-t^{4k-2}}=\prod\limits_{k=1}^\infty (1+t^{2k})$ and $\prod\limits_{k=1}^\infty \frac{1}{(1-t^k)(1-(-t)^k)}=\prod\limits_{k=1}^\infty \frac{1+t^{2k}}{(1-t^{2k})^2}$, we see that the above generating function is equal to 
\begin{equation}\label{equ-5.3-2}
\prod_{k=1}^\infty \frac{(1-t^{2k})^3}{(1-t^{4k})^2}\prod\limits_{k=1}^\infty \frac{1+t^{2k}}{(1-t^{2k})^2}=\prod_{k=1}^{\infty}\frac{1}{1-t^{4k}}.
	\addtocounter{thm}{1}\tag{\thethm}
\end{equation}
In particular, $\omega_0(\rO_{w,+}(q))\ne \omega_0(\rO_{w,-}(q))$ occurs only when $4\mid w$, and this forces that $\omega_0(\rO_{w,\eps}(q))= \omega_0(\rO_{w}^\eps(q))$ for any $w$ and any $\eps\in\{\pm\}$.
Notice that $\prod\limits_{k=1}^\infty (1+t^k)^2=\prod\limits_{k=1}^\infty (1+t^{2k-1})^2(1+t^{2k})^2$, therefore, by (\ref{equ-5.1}), (\ref{equ-5.1-1}), (\ref{equ-5.3}) and (\ref{equ-5.3-2}), one gets that $\omega_0(I)=u(I)$.

\begin{lem}\label{BAWC-SO}
The blockwise Alperin weight conjecture holds for the principal 2-block of $H=\SO(V)$.
\end{lem}	

\begin{proof}
By \cite[Prop. 2.7]{Ch20}, $|\IBr(B_0(H))|=u(H)$, and thus it suffices to show that $\omega_0(H)=u(H)$. First we calculate $u(H)$.
If $w$ is odd or $\ty(V)=-$, then $u(H)=u(I)$.
If $w$ is even and $\ty(V)=+$, then $u(H)-u(I)$ is equal to the number of the partitions $\la\in\mathscr P_w$ with $a(\la)=0$, that is, the coefficient of $t^w$ in  (\ref{equ-5.1-1}).
In particular,  $u(H)\ne u(I)$ occurs only when $4\mid w$ and $\ty(V)=+$.

Next, we determine $\omega_0(H)$.
Note that a principal weight of $I$ covers one or two principal weights of $H$, and every principal weight of $H$ is covered by a unique principal weight of $I$.
The parity of the group $H$ is $(1,1)$, thus by Corollary~\ref{cor:wei-cov-solquo} and Proposition~\ref{P:Ta1}, a principal weight $(R,\varphi)$ of $I$ covers two principal weights of $H$ if and only if the parity of $R$ is~0, which occurs only when $4\mid w$ and $\ty(V)=+$.
By Table~\ref{Ta1}, the principal weights of $I$ which covers two principal weights of $H$ are in bijection with the partitions $\la\in\mathscr P_w$ with $a(\la)=0$. 
Thus $\omega_0(H)-\omega_0(I)$ is also equal to the coefficient of $t^w$ in  (\ref{equ-5.1-1}). In particular,  $\omega_0(H)\ne\omega_0(I)$ occurs only when $4\mid w$ and $\ty(V)=+$.
From the arguments above, we get that $\omega_0(H)=u(H)$.
\end{proof}

Under the natural isogeny $\bG\twoheadrightarrow\bH$, the unipotent classes in these two algebraic groups are in bijection. From this we use the same parametrisation for unipotent conjugacy classes of $\bG$ and $\bH$.
Moreover, this isogeny induces a bijection between the unipotent classes of $G$ and of $\Omega(V)$. 
Define $\delta(\la)=1$ if $a(\la)=0$,  $\delta(\la)=a(\la)$ if $a(\la)>0$ and $\ka(\la)=1$, and $\delta(\la)=a(\la)-1$ if $a(\la)>0$ and $\ka(\la)>1$.

We write $A_{\bH,\la}$ (resp. $A_{\bG,\la}$) for the component group $A_{\bH}(u)=Z_{\bH}(u)/Z^\circ_{\bH}(u)$ (resp. $A_{\bG}(u)=Z_{\bG}(u)/Z^\circ_{\bG}(u)$) for $u\in \mathbf C_\la$.
The component group $A_{\bH,\la}$ is an elementary abelian 2-group of order $2^{b(\la)}$; see \cite[IV, 2.26 and 2.27]{SS70}.
In \cite[\S14]{Lu84}, Lusztig gave a description of the component group $A_{\bG,\la}$; the group $A_{\bH,\la}$ is a quotient of $A_{\bG,\la}$ with central kernel of order at most 2.
In particular, $A_{\bG,\la}=A_{\bH,\la}$ if $a(\la)>0$ and $\ka(\la)>1$, 
$|A_{\bG,\la}|/|A_{\bH,\la}|=2$ if $a(\la)=0$ or $\ka(\la)=1$,
and $A_{\bG,\la}$ is non-abelian if $a(\la)>2$ and $\ka(\la)=1$.
By \cite[Lemma~2.39]{Tay12}, if $a(\la)>0$ and $\ka(\la)=1$, then the group $A_{\bG,\la}$ has $2^{b(\la)}+1$ or $2^{b(\la)}+2$ conjugacy classes according as $w$ is odd or even.

Let $X$ be arbitrary finite group and $\sigma\in\Aut(X)$, $x_1,x_2\in X$. We say that $x_1$ and $x_2$ are \emph{$\sigma$-conjugate} if there exists a $g\in X$ with $x_2=\sigma(g)x_1g^{-1}$.
The equivalence classes for this relation are called \emph{$\sigma$-conjugacy classes} of $X$.

\begin{lem}\label{lem:tw-conju}
	Let $X$ be a finite group and $\sigma\in\Aut(X)$.
\begin{enumerate}[\rm(1)]
		\item The number of $\sigma$-conjugacy classes of $X$ is equal to the number of the conjugacy classes of $X$ fixed by $\sigma$.
\item If $\sigma$ induces an inner automorphism on $X$, then the number of $\sigma$-conjugacy classes of $X$ is equal to the number of conjugacy classes of $X$.
\end{enumerate}
\end{lem}	

\begin{proof}
Let $x_1,x_2\in X$. Then $x_1$ and $x_2$ are $\sigma$-conjugate in $X$ if and only if $x_1\sigma$ and $x_2\sigma$ are $X$-conjugate in the semidirect product $X\rtimes\langle\sigma\rangle$.
So the number of $\sigma$-conjugacy classes of $X$ is equal to the number of $X$-conjugacy classes in the coset $X\sigma$.

(1) Let $g\in X$. Denote by $\textup{Fix}_{X\sigma}(g)$ the fixed point under the conjugacy action of $g$ on $X\sigma$. If $x_1\sigma,x_2\sigma\in\textup{Fix}_{X\sigma}(g)$. Then $x_1x_2^{-1}=(x_1\sigma)(x_2\sigma)^{-1}=(gx_1\sigma g^{-1})(gx_2\sigma g^{-1})^{-1}=g(x_1x_2^{-1})g^{-1}$, which implies that $x_1x_2^{-1}\in Z_X(g)$. 
So $|\textup{Fix}_{X\sigma}(g)|\le |Z_X(g)|$.
On the other hand, if $x\sigma\in\textup{Fix}_{X\sigma}(g)$ and $c\in Z_X(g)$, then $cx\sigma\in\textup{Fix}_{X\sigma}(g)$, and thus $|Z_X(g)|\le |\textup{Fix}_{X\sigma}(g)|$.
Therefore, $|\textup{Fix}_{X\sigma}(g)|=|Z_X(g)|$ if $\textup{Fix}_{X\sigma}(g)\ne\emptyset$.

For $x\in X$, we have that $x\sigma\in\textup{Fix}_{X\sigma}(g)$ if and only if $x\sigma=g(x\sigma)g^{-1}=(gxg^{-1})(g\sigma g^{-1}\sigma^{-1})\sigma$, and thus if and only if $x^{-1}gx=\sigma g\sigma^{-1}$.
This implies that $\textup{Fix}_{X\sigma}(g)\ne\emptyset$ if and only if the conjugacy class of $X$ containing $g$ is fixed by $\sigma$.
Hence the assertion follows by Burnside's Lemma.

(2) This follows by (1) immediately.
\end{proof}

\begin{lem}\label{lem:dh}
Let $X$ be a dihedral group of order~8. 
\begin{enumerate}[\rm(1)]
	\item  Let $W$ be a set acted by $X$ such that $Z(X)$ acts transitively on $W$. Then $W$ is a singleton.
	\item  Write $X=\langle a,b\mid a^4=b^2=(ab)^2=1\rangle$. Let $W_1$ and $W_2$ be two sets acted by $X$ with $|W_1|=|W_2|<\infty$. Suppose that there is a bijection $f$ between the $Z(X)$-orbits on $W_1$ and $W_2$ which preserves $Z(X)$-orbit lengths. Then there exists an $X$-equivariant bijection between $W_1$ and $W_2$ if the following conditions hold.
	\begin{enumerate}[\rm(a)]
		\item If $O$ is a $Z(X)$-orbit on $W_1$, then $f(\{ x.O\mid x\in X\})=\{x.f(O)\mid x\in X\}$, that is, $f$ preserves the $X$-orbits on the $Z(X)$-orbits.
		\item If $O$ is a $Z(X)$-orbit on $W_1$ with length one, then $f(x.O)=x.f(O)$ for any $x\in X$.
		\item If $O$ is a $Z(X)$-orbit on $W_1$ with length two and satisfying that $|X_w|=2$ for $w\in O$, then $O$ is $\langle a^2,b \rangle$-stable if and only if $f(O)$ is $\langle a^2,b \rangle$-stable. 
\end{enumerate}
\end{enumerate}
\end{lem}	

\begin{proof}
Let $w\in W$. If $W$ is not a singleton, then $|W|=2$ and the stabilizer $X_w$ of $w$ in $X$ has four elements. Moreover, $X=Z(X)X_w$. However, this is impossible as $Z(X)$ is contained in any normal subgroup of $X$ of index 2. Thus (1) holds.

For (2), we mention that $Z(X)=\langle a^2\rangle$.
By (a), it suffices to prove this lemma for the situation that the actions of $X$ on $W_1$ and $W_2$ are transitive.

Now we assume that $X$ acts transitively on $W_1$ and $W_2$. By (b), we may assume that $\langle a^2\rangle$ acts non-trivially on $W_1$ (and then non-trivially on $W_2$).
By (1), one has that $|W_1|=|W_2|=4$ or~$8$.
If $|W_1|=|W_2|=8$, then it is clear that there exists an $X$-equivariant bijection between $W_1$ and $W_2$, as the actions are regular for this situation.
Now we suppose that $|W_1|=|W_2|=4$. If  $O$ is not $\langle a^2,b \rangle$-stable, then $W_1=O\coprod b.O$, $W_2=f(O)\coprod b.f(O)$ by (c). 
Recall that $X$ has three conjugacy classes of involutions, i.e., $\{a^2\}$, $\{b,a^2b\}$, $\{ab,a^3b\}$.
So there exist $w_1\in W_1$, $w_2\in W_2$ such that $X_{w_1}=X_{w_2}=\langle ab\rangle$.
If $O$ and $f(O)$ are $\langle a^2,b \rangle$-stable, then $X_{w_1}=\langle b\rangle$ or $\langle a^2b\rangle$ for $w_1\in O$, and $X_{w_2}=\langle b\rangle$ or $\langle a^2b\rangle$ for $w_2\in f(O)$. 
Thus we conclude this assertion, since $\langle b\rangle$ and $\langle a^2b\rangle$ are conjugate in $X$.
\end{proof}	

\begin{lem}\label{lem:F-conju}
	Let $\mathbf M\in\{\bG,\bH\}$ and $u$ be a unipotent element of $\mathbf M$ such that the conjugacy class of $\mathbf M$ containing $u$ is $F$-fixed. Assume further that $F$ acts trivially on $Z(\mathbf M)$.
Then the number of $F$-conjugacy classes of $A_{\mathbf M}(u)$ is equal to the number of conjugacy classes of $A_{\mathbf M}(u)$.
\end{lem}	

\begin{proof}
For special orthogonal group, this is trivial. Thus we only consider spin groups here and we also assume that $\disc(V)=+$.
By Lemma~\ref{lem:tw-conju}, it suffices to prove that $F$ fixes every conjugacy class of $A_{\mathbf M}(u)$.
The conjugacy classes of $A_{\mathbf M}(u)$ are described in the proof of \cite[Lemma~2.39]{Tay12}, and this assertion can be deduced by the results there directly.
\end{proof}

\begin{prop}\label{BAWC-Spin}
The blockwise Alperin weight conjecture holds for the principal 2-block of $G=\Spin(V)$.
\end{prop}

\begin{proof}
If $w$ is even and $\disc(V)=-$, then $\SO(V)=\{\pm I_V\}\ti \Omega(V)$ and thus this proposition follows from Lemma~\ref{BAWC-SO} immediately. So we assume that $\disc(V)=+$.

By \cite[Prop.~2.7]{Ch20}, $|\IBr(B_0(G))|=u(G)$.
Let $\mathbf C_\la$ be a unipotent class of $\bG$. 
By \cite[Thm.~21.11]{MT11}, the number of $G$-conjugacy classes in $\mathbf C_\la^F$ is equal to the number of $F$-conjugacy classes in $A_{\bG,\la}$.
So by Lemma~\ref{lem:F-conju} and \cite[Lemma~2.39]{Tay12}, the number of $G$-conjugacy classes in $\mathbf C_\la^F$ is $2^{\delta(\la)}$ if $a(\la)=0$ or $\ka(\la)>1$, and if $a(\la)>0$ and $\ka(\la)=1$, then $\mathbf C_\la^F$ splits into $2^{b(\la)}+1$ or $2^{b(\la)}+2$ conjugacy classes in $G$ according as $w$ is odd or even.

We first consider the situation that $w$ is odd. Then $|\IBr(B_0(G))|-|\IBr(B_0(H))|=u(G)-u(H)$ is equal to the number of partitions $\la$ of $w$ with $\ka(\la)=1$, which is thus equal to the coefficient of $t^{w}$ in the generating function
\begin{equation}\label{equ-5.4}
	(1+t)(\sum_{k=0}^\infty t^{4k})(1+t^{3})(\sum_{k=0}^\infty t^{8k})\cdots=\prod_{k=1}^\infty \frac{1+t^{2k-1}}{1-t^{4k}}.
	\addtocounter{thm}{1}\tag{\thethm}
\end{equation}
Now we count the number of weights $\omega_0(G)=\omega_0(\Omega(V))$.
Let $(T,\eta)$ be a principal weight of $\Omega(V)$. Then $(T,\eta)$ is covered by a unique weight $(T_1,\eta_1)$ of $H$, which is also a principal weight.
Also, $(T_1,\eta_1)$ is covered by a unique weight $(R,\varphi)$ of $I$, which is a principal weight.
Note that the parity of $R$ always contains $(1,0)$ as a subgroup and a weight of $H$ covers one or two weights of $\Omega(V)$.
By Corollary~\ref{cor:wei-cov-solquo}, $(T_1,\eta_1)$ covers two weights  (up to conjugacy) of $\Omega(V)$ if and only if $|H/\Omega(V)T_1|=2$, and thus if and only if the parity of $R$ is just $ (1,0)$, in which case the basic subgroup appearing in $R$ only comes from lines 1, 2, 6 of Table~\ref{Ta1}.
Using \cite[(1A)]{AF90} again, we conclude that the principal weights of $H$ covering  two weights (up to conjugacy) of $\Omega(V)$ are in bijection with the triples $(\la_1,\la_2,\ka)$ such that $\la_1$, $\la_2$ are partitions, $\ka$ is a 2-core satisfying that 
\[4(|\la_1|+|\la_2|)+|\ka|=w.\]
Therefore, $\omega_0(G)-\omega_0(H)$ is equal to the coefficient of $t^{w}$ in the generating function
\begin{equation}\label{equ-5.5}
	(\sum_{k=0}^\infty t^{4k})^2(\sum_{k=0}^\infty t^{8k})^2\cdots \vartheta(t)=\prod_{k=1}^\infty \frac{1+t^{2k-1}}{1-t^{4k}}.
	\addtocounter{thm}{1}\tag{\thethm}
\end{equation}
Therefore, by Lemma~\ref{BAWC-SO} and (\ref{equ-5.4}), (\ref{equ-5.5}), we get $|\IBr(B_0(G))|=\omega_0(G)$.

Now we suppose that $w$ is even and $\disc(V)=+$.
If $\ty(V)=-$, then $4\nmid w$ and thus $a(\la)>0$ for all $\la\in\mathscr P_w$ with $\mathbf C_\la^F\ne\emptyset$.
Thus $u(G)-u(H)$ is equal to the double of the number of $\la\in\mathscr P_w$ with $\ka(\la)=1$, which is thus equal to the double of the coefficient of $t^w$ in (\ref{equ-5.4}).
If $\ty(V)=+$, then $u(G)-u(H)$ is equal to the double of the number of $\la\in\mathscr P_w$ with $a(\la)=0$ or $\ka(\la)=1$, which is thus equal to the double of the coefficient of $t^w$ in (\ref{equ-5.4}).
Now we count the number of weights $\omega_0(G)=\omega_0(\Omega(V))$.
Similar as for the case that $w$ is odd, we know that $\omega_0(G)-\omega_0(H)$ is equal to the number of the conjugacy classes of principal weight subgroups of $I$ with parities $(1,0)$, $(0,1)$ and $0$ and such  principal weight subgroups are counted once, once, twice respectively.
From this, $\omega_0(G)-\omega_0(H)$ is equal to the double of the coefficient of $t^{w}$ in (\ref{equ-5.5}).
Therefore, by Lemma~\ref{BAWC-SO} and (\ref{equ-5.4}), (\ref{equ-5.5}), we also get $|\IBr(B_0(G))|=\omega_0(G)$ for the situation that $w$ is even and $\disc(V)=+$.
\end{proof}

Let $F_0$ be defined as in \S \ref{SS:prin-wei}.

\begin{lem}\label{prop:act-quasi-SO}
Let $s\in{\bH^*}^F$ be a semisimple element such that the image of $s$ in the adjoint quotient $\bH^*/Z(\bH^*)$ is quasi-isolated.
Then every character $\chi\in\cE(\bH^F,s)$ is $\langle F_0\rangle$-invariant.
\end{lem}

\begin{proof}
	Let $F_0^*$ be the automorphism of $\bH^*$ corresponding to $F_0$.
We first prove that the conjugacy class of ${\bH^*}^F$ containing $s$ is $\langle F_0^*\rangle$-stable.
Quasi-isolated elements of adjoint simple groups were classified in \cite[\S5]{Bo05}.
Under the notation in the proof of Lemma~\ref{prop:2-defect-zero-SO}, $m_\Ga(s)>0$ only when $\Ga=x-1$, $x+1$ or $\Delta$, where $\Delta\in\cF$ is the polynomial with roots  $\zeta$ and $-\zeta$ where $\zeta$ is a primitive 4th root of unity in $\overline{\mathbb F}_q^\ti$.
Thus $F_0^*(s)$ is ${\bH^*}^F$-conjugate to $s$.
Moreover, we note that the action of automorphism groups on unipotent characters is given in \cite[Prop.~3.7 and 3.9]{Ma07}; from this we see that field automorphisms act trivially on unipotent characters.

Assume that $w$ is odd.	
Since $\bH$ has connected center, by \cite[Thm.~7.1]{DM90}, the action of $\langle F_0\rangle$ on $\Irr(H)$ is determined by its action on semisimple elements as well as on unipotent characters, via an equivariant Jordan decomposition.
Thus every character $\chi\in\cE(\bH^F,s)$ is $\langle F_0\rangle$-invariant.
Now assume that $w$ is even.
We take a regular embedding $\bH\to\tilde\bH$, so that $\tilde\bH$ is the special conformal orthogonal group $\textup{CSO}(V)$ over $V$.
Let $\tilde s\in\tilde\bH^*$ such that its image is $s$ under $\tilde\bH^*\twoheadrightarrow\bH^*$ and let $\tilde\chi\in\cE(\tilde\bH^F,\tilde s)\cap \Irr(\tilde\bH^F\mid\chi)$.
Then $F_0^*(\tilde s)$ is $(\tilde\bH^*)^F$-conjugate to $z\tilde s$ with $z\in Z(\tilde\bH^*)^F$ and from \cite[Thm.~7.1]{DM90} we deduce that $\tilde\chi^{F_0}=\hat z \tilde\chi$.
Thus $\chi\in\cE(\bH^F,s)$ is $\langle F_0\rangle$-invariant by \cite[Thm.~A]{Sp23}.
\end{proof}	

\begin{cor}\label{cor:act-Br-SO}
Every Brauer character in $\IBr(B_0(H))$ is $\langle F_0\rangle$-invariant.
\end{cor}

\begin{proof}	
By \cite[Thm.2.9]{Ch20}, there is unitriangular basic set $I$ for the principal block $B_0(H)$ of $H$ such that every character in $I$ lies in a $\chi\in\cE(\bH^F,s)$ such that the image of $s$ in the adjoint quotient $\bH^*/Z(\bH^*)$ is quasi-isolated.
So this assertion follows from Lemma~\ref{prop:act-quasi-SO} immediately.
\end{proof}	

\begin{thm}\label{thm:IBAW-B}
	Suppose that $\dim(V)$ is odd and $\Omega(V)$ is simple.
If every element in $\IBr(B_0(G))$ is $\langle F_0\rangle$-invariant, then the inductive BAW condition holds for the principal 2-block of $\Omega(V)$.
\end{thm}	

\begin{proof}	
	To verify the inductive BAW condition for the principal block of $G$, we use the criterion of Brough--Sp\"ath \cite{BS22}. For groups of type $\tp B$, we refer to \cite[Thm.~3.3]{FM22} for convenience, and under the notation there, we let $X=S=\Omega_{2n+1}(q)$, $\tilde X=H=\SO_{2n+1}(q)$ and $D=\langle F_0\rangle$.
	Then the condition (1) of \cite[Thm.~3.3]{FM22} holds.
	Suppose that $B$ is the principal 2-block of $X$ and $\tilde{\mathcal B}$ is the principal 2-block of $\tilde X$.
	The conditions (3) and (4) of \cite[Thm.~3.3]{FM22} follow by Theorem~\ref{A-infinity} and Corollary~\ref{cor:act-field-cla} respectively, as every element of $\IBr(B)\cup\Alp(B)$ is $\langle F_0\rangle$-invariant.
	
	By the arguments preceding Lemma~\ref{lem:tw-conju}, $|\IBr(\tilde{\mathcal B})|=|\Alp(\tilde{\mathcal B})|$, and by Corollary~\ref{cor:act-field-cla} and \ref{cor:act-Br-SO}, every element of $\IBr(\tilde{\mathcal B})\cup\Alp(\tilde{\mathcal B})$ is $\langle F_0\rangle$-invariant.
	As $|\tilde X/X|=2$, one element in $\IBr(\tilde{\mathcal B})$ (resp. $\Alp(\tilde{\mathcal B})$) covers one or two elements of $\IBr(B)$ (resp. $\Alp(B)$).
	Therefore, to prove condition (2) of \cite[Thm.~3.3]{FM22}, it suffices to show that the number of irreducible Brauer characters of $\tilde{\mathcal B}$ covering two irreducible Brauer characters of $B$ is equal to the number of weights (up to conjugacy) of $\tilde{\mathcal B}$ covering two weights of $B$, which was proved in the proof of Proposition~\ref{BAWC-Spin}.
	Thus we complete the proof.
\end{proof}

\begin{cor}\label{cor-BAW-Spin7}
	The inductive BAW condition holds for the principal 2-block of $\Spin_7(q)$ (with odd $q$).
\end{cor}

\begin{proof}
	Keep the notation and setup as above.
	Let $\bG\hookrightarrow\tilde\bG$ be a regular embedding. 	Following H\'ezard \cite{He04} and Lusztig \cite{Lu09}, Taylor \cite[\S5.2]{Tay12} describes a map $\mathbf C_\la\mapsto \chi_\la$ from the set of unipotent classes of $\bG$ to $\Irr(G)$ with the following properties: $\chi_\la$ lies in the Lusztig series of a
	quasi-isolated 2-element $\tilde s_\la\in\tilde \bG^*$ with centralizer of type $\tp C_{a} \tp C_b$ with $a+b=n$ if $\kappa(\la)\ge2$, and of type $\tp C_a \tp A_b\tp C_a$ with $2a+b=n$ if $\kappa(\la)=1$.
	Let $\cF_\la$ denote the Lusztig family of the $Z_{\tilde\bG^*}(\tilde s_\la)$ containing the Jordan correspondent of $\chi_\la$.
	
	According to \cite[Prop.~4.3]{GH08} for each class $\mathbf C_\la$ there exist
	$|A_{\tilde\bG}(u)|$ (for $u\in\mathbf C_\la$) characters $\tilde\rho_{\la,i}$ in the Lusztig series $\cE(\tilde\bG^F,\tilde s_\la)$ with Jordan
	correspondents in the family $\cF_\la$, such that the matrix of multiplicities
	of their Alvis--Curtis--Kawanaka--Lusztig duals in the generalised
	Gelfand--Graev characters corresponding to $\mathbf C_\la$ is the identity matrix.
	Further, Taylor \cite[Prop.~5.4]{Tay13} shows that for each class $\mathbf C_\la$
	there exist $|A_{\la}|$ characters $\rho_{\la,i}\in\Irr(\bG^F)$ (with $1\le i\le |A_{\la}|$), which are just the characters lying below the $\tilde\rho_{\la,i}$, such that again the matrix of multiplicities of their duals in the generalised Gelfand--Graev characters corresponding to $\mathbf C_\la$
	is the identity matrix.  Let $\Xi$ be the set of characters $\rho_{\la,i}$, where $\mathbf C_\la$ runs through the unipotent classes of $\bG$, and $1\le i\le |A_{\la}|$. 
	By a result of Chaneb \cite[Thm.2.9]{Ch20}, $\Xi$ is a unitriangular basic set of the principal 2-block of $\Spin_7(q)$.
	Note that, in \cite[Thm.2.9]{Ch20} spin and half spin groups are excluded as $A_{\la}$ may be non-abelian in general. However, for our situation $n=3$, the group $A_{\la}$ is an elementary 2-group of rank $\delta(\la)$ as $a(\la)=1$ when $\ka(\la)=1$, and the proof of \cite[Thm.2.9]{Ch20} also applies.
	
	By Corollary~\ref{cor:act-Br-SO}, those unitriangular basic sets are $\langle F_0\rangle$-stable.
	This means, $B_0(H)$ has an $\langle F_0\rangle$-stable unitriangular basic set, such that the irreducible constituents of the restriction of those irreducible characters to $S$ form an $(H\rtimes\langle F_0\rangle)$-stable unitriangular basic set for $B_0(S)$.
	Therefore, by Corollary~\ref{cor:act-Br-SO}, every element in $\IBr(B_0(S))$ is $\langle F_0\rangle$-invariant and thus this assertion follows from Theorem~\ref{thm:IBAW-B}.
\end{proof}	

Let $J$ (resp. $\tH$) be the conformal orthogonal group $\textup{CO}(V)$ (resp. $\textup{CSO}(V)$) over $V$.

\begin{prop}\label{prop:act-IBr-SO}
	Suppose that $\dim(V)$ is even.
There is a $(J\rtimes\langle F_0\rangle)$-equivariant bijection between $\IBr(B_0(H))$ and the unipotent conjugacy classes of $H$.
\end{prop}

\begin{proof}
	Following H\'ezard \cite{He04} and Lusztig \cite{Lu09}, Taylor \cite[\S6]{Tay12} describes a map $\mathbf C_\la\mapsto \chi_\la$ from the set of unipotent classes of $\bH$ to $\Irr(H)$ such that $\chi_\la$ lies below the Lusztig series of a semisimple 2-element $\tilde s_\la\in\tilde \bH^*$ satisfying that the image of $\tilde s_\la$ in the adjoint quotient $\tilde\bH^*/Z(\tilde\bH^*)$ is quasi-isolated and the following properties hold:
$\tilde s_\la=1$ and $\chi_\la$ is a unipotent character parametrised by a non-degenerate Lusztig symbol if $a(\la)>0$, $\kappa(\la)>1$ and $\iota(\la)=0$, 
$\tilde s_\la$ has centralizer of type $\tp D_\mu \tp D_\nu$ with $\mu+\nu=w/2$ and the image of $\chi_\la$ under Lusztig’s Jordan decomposition is a special unipotent character parametrised by non-degenerate Lusztig symbols if $a(\la)>0$, $\kappa(\la)>1$ and $\iota(\la)=1$,
$\tilde s_\la$ has centralizer of type $\tp D_\mu \tp A_\nu \tp D_\mu$ with $2\mu+\nu=w/2$ and the image of $\chi_\la$ under Lusztig’s Jordan decomposition is a special unipotent character parametrised by non-degenerate Lusztig symbols if $a(\la)>0$ and $\kappa(\la)=1$,
$\tilde s_\la$ has centralizer of type $\tp A_{n-1}$ if $a(\la)=0$.
Here, $\iota(\la)=1$ if there exists an odd $i$ with odd $c_i$, and  $\iota(\la)=0$ otherwise.
	For the partition $\la\in\mathscr P_w$ with $a(\la)=0$, we also write $\chi_\la$ as $\chi_\la^\pm$ to distinguish $\mathbf C_\la^\pm$. 

	According to \cite[Prop.~4.3]{GH08} for each class $\mathbf C_\la$ there exist
		$|A_{\tilde\bH,\la}|$ irreducible characters $\tilde\rho_{\la,i}$ in the Lusztig series $\cE(\tilde\bH^F,\tilde s_\la)$ with Jordan
		correspondents in the same family as $\chi_\la$, such that the matrix of multiplicities
		of their Alvis--Curtis--Kawanaka--Lusztig duals in the generalised
		Gelfand--Graev characters corresponding to $\mathbf C_\la$ is the identity matrix.
		Further, Taylor \cite[Prop.~5.4]{Tay13} shows that for each class $\mathbf C_\la$
		there exist $|A_{\bH,\la}|$ characters $\rho_{\la,i}\in\Irr(\bH^F)$ (with $1\le i\le |A_{\bH,\la}|$), which are just the characters lying below the $\tilde\rho_{\la,i}$, such that again the matrix of multiplicities of their duals in the generalised
		Gelfand--Graev characters corresponding to $\mathbf C_\la$
		is the identity matrix.  Let $\Xi$ be the set of those characters $\rho_{\la,i}$, where $\mathbf C_\la$ runs through the unipotent classes of $\bH$, and $1\le i\le |A_{\bH,\la}|$. 
		By a result of Chaneb \cite[Thm.~2.9]{Ch20}, $\Xi$ is a unitriangular basic set of the principal 2-block of $H$.

	By Lemma~\ref{prop:act-quasi-SO}, every character of $\Xi$ is $\langle F_0\rangle$-invariant.
We partition \[\Xi=\Xi_1\coprod \Xi_2\coprod \Xi_3\coprod \Xi_4,\] where $\Xi_1$ consists of those $\rho_{\la,i}$ with $a(\la)>0$, $\kappa(\la)>1$ and $\iota(\la)=0$, $\Xi_2$ consists of those $\rho_{\la,i}$ with $a(\la)>0$, $\kappa(\la)>1$ and $\iota(\la)=1$, $\Xi_3$ consists of those $\rho_{\la,i}$ with $a(\la)>0$ and $\kappa(\la)=1$, and $\Xi_4$ consists of those $\rho_{\la,i}$ with $a(\la)=0$.
The action of $J$ on $\Irr(H)$ is determined in \cite[Thm.~7.1]{DM90} and \cite[Appendix~B]{FLZ19}.
Thus we deduce that every character in $\Xi_1$ is $J$-invariant, every character in $\Xi_2\coprod\Xi_3$ is $I$-invariant, every character in $\Xi_4$ is $\tH$-invariant.
For $a(\la)=0$, we have $\rho_{\la,i}=\tilde\rho_{\la,i}$ and $i=1$ and if we write $\rho_{\la,i}$ as $\rho_{\la}^\pm$ to distinguish $\mathbf C_\la^\pm$, then $\rho_{\la}^+$ and $\rho_{\la}^-$ are fused by $I$.
Moreover, the characters in $\Xi_2$ (or $\Xi_3$) occur in pairs such that for each pair those two characters are fused by $\tH$.
In particular, $\Xi$ is $(J\rtimes\langle F_0\rangle)$-stable and this implies that there exists a  $(J\rtimes\langle F_0\rangle)$-equivariant bijection between $\Xi$ and $\IBr(B_0(H))$.

Denote by $\mathscr U$ the set of unipotent conjugacy classes of $H$, and we also partition \[\mathscr U=\mathscr U_1\coprod \mathscr U_2\coprod \mathscr U_3\coprod \mathscr U_4,\] where $\mathscr U_1$ consists of conjugacy classes of $H$ contained in some $\mathbf C_\la$ with  $a(\la)>0$, $\kappa(\la)>1$ and $\iota(\la)=0$, $\mathscr U_2$ consists of conjugacy classes of $H$ contained in some $\mathbf C_\la$ with $a(\la)>0$, $\kappa(\la)>1$ and $\iota(\la)=1$, $\mathscr U_3$ consists of conjugacy classes of $H$ contained in some $\mathbf C_\la$ with $a(\la)>0$ and $\kappa(\la)=1$, and $\mathscr U_4$ consists of conjugacy classes of $H$ contained in some $\mathbf C_\la^\pm$ with  $a(\la)=0$.
We note that every unipotent class of $H$ is $(J\rtimes\langle F_0\rangle)$-stable unless the following situations: 
$\mathbf C_\la^+$ and $\mathbf C_\la^-$ are $I$-fused if $a(\la)=0$, and the elements in $\mathscr U_2$ (or $\mathscr U_3$) present in pairs such that for each pair those two elements are fused by $\tH$.

By construction $|\Xi_k|=|\mathscr U_k|$ for $k=1,2,3,4$.
Precisely, for each $k$, we have a bijection $\Xi_k\to\mathscr U_k$ such that the image of $\rho_{\la,i}$ is a conjugacy class of $H$ contained in $\mathbf C_\la$.
Therefore, we can get a bijection $\Xi\to\mathscr U$ which is $(J\rtimes\langle F_0\rangle)$-equivariant by the above arguments, and this completes the proof.
\end{proof}

\begin{cor}\label{cor:act-Brauer}
	Suppose that $\dim(V)$ is even. For any $\psi\in\IBr(H)$, $\sigma\in\tH$ and $\sigma'\in I\rtimes\langle F_0\rangle$, if $\psi^\sigma=\psi^{\sigma'}$, then $\psi^\sigma=\psi^{\sigma'}=\psi$ and $\psi$ extends to $I_\psi\rtimes\langle F_0\rangle_\psi$.
\end{cor}

\begin{proof}
Since $I\langle F_0\rangle_\psi/H$ has cyclic Sylow $\ell$-subgroups for any odd prime $\ell$, the extendibility of $\psi$ to $I_\psi\rtimes\langle F_0\rangle_\psi$ is clear (see e.g. \cite[Thm.~8.29]{Na98}).
Thanks to \cite[Thm.~A]{Sp23}, it suffices to show that $H$ has a $J\langle F_0\rangle$-stable unitriangular basic set.
By the proof of Proposition~\ref{prop:act-IBr-SO}, the principal 2-block of $H$ has a $J\langle F_0\rangle$-stable unitriangular basic set.
For non-principal 2-blocks of $H$, the proof of \cite[Cor.~4.6]{FM22} also applies for our situation; Jordan decomposition for Brauer character as in \cite[Prop.~4.5]{FM22} also holds by Bonnaf\'e--Rouquier Morita equivalence \cite{BR03}; 
using \cite[Thm.~16.4]{ST23} and \cite[Lemma~B.1]{FLZ19}, we can argue as in the proof of \cite[Prop.~4.5]{FM22}.
\end{proof}

Next, we determine the action of $J$ on $\Alp(B_0(\Omega(V)))$ and establish an equivariant bijection.

\begin{thm}\label{thm:eq-bij-spin}
	Suppose that $\dim(V)$ is even.
There exists a $J$-equivariant bijection between $\IBr(B_0(\Omega(V)))$ and $\Alp(B_0(\Omega(V)))$.
	\end{thm}
	
	\begin{rmk}
		Suppose that $\dim(V)$ is even.
		Then there is a $J$-equivariant bijection between $\IBr(B_0(G))$ and the unipotent conjugacy classes of $G$ (which follows from the proof below).
		\end{rmk}

	\begin{proof}
First we determine the action of $J$ on $\Alp(B_0(I))$.
In fact, the proof of Lemma~\ref{L:D4} also applies for arbitrary dimension, not only when $w=8$. In Table~\ref{Ta1}, the basic subgroups are $J$-stable except the group $R_{1,0,0,\bc}$ in line 2 and the group $R_{1,1,0,\bc}$ in line 3 are $J$-fused, and the groups in line 6 and 7 are $J$-fused.
Thus the $J$-invariant conjugacy classes of principal 2-weights of $I$ are in bijection with the tuples $(\la_1,\la_2,\la_3,\ka_1,\ka_2)$ where $\la_1$, $\la_2$ and $\la_3$ are partitions, $\ka_1$ and $\ka_2$ are 2-cores such that $\disc(V)=(-1)^{|\ka_2|}$ and
\[ 4(|\la_1|+|\la_2|)+8|\la_3|+2(|\ka_1|+|\ka_2|)=w.\]
We denote by $\omega_{w,+}'$ and $\omega_{w,-}'$ the numbers of $J$-invariant conjugacy classes of principal 2-weights of $I$ for the case $\disc(V)=+$ and $-$ respectively.
Then $\omega_{w,+}'+\omega_{w,-}'$ is equal to the coefficient of $t^w$ in 
\begin{equation}\label{equ-J-inv}
	((\sum_{k=0}^\infty t^{4k})(\sum_{k=0}^\infty t^{8k})\cdots)^2 (\sum_{k=0}^\infty t^{8k})(\sum_{k=0}^\infty t^{16k})\cdots\vartheta(t^2)^2=\prod_{k=1}^\infty \frac{(1+t^{2k})^2}{1-t^{8k}},
	\addtocounter{thm}{1}\tag{\thethm}
\end{equation}
while $\omega_{w,+}'-\omega_{w,-}'$ is equal to the coefficient of $t^w$ in 
\begin{equation}\label{equ-J-inv-1}
	((\sum_{k=0}^\infty t^{4k})(\sum_{k=0}^\infty t^{8k})\cdots)^2 (\sum_{k=0}^\infty t^{8k})(\sum_{k=0}^\infty t^{16k})\cdots\vartheta(t^2) \vartheta(-t^2) 
=\prod_{k=1}^\infty \frac{1}{1-t^{4k}}.
	\addtocounter{thm}{1}\tag{\thethm}
\end{equation}
Here we use identities $\vartheta(t)=\prod\limits_{k=1}^\infty\frac{(1-t^{2k})^2}{1-t^k}$ and $\prod\limits_{k=1}^\infty \frac{1}{(1-t^k)(1-(-t)^k)}=\prod\limits_{k=1}^\infty \frac{1+t^{2k}}{(1-t^{2k})^2}$.
In particular, $\omega_{w,+}'\ne\omega_{w,-}'$ occurs only when $4\mid w$.

The $J$-invariant unipotent conjugacy classes of $I$ are just those in $\mathbf C_\la$ such that ``$a(\la)=0$'' or ``$a(\la)\ge 1$, $\ka(\la)\ge 2$ and $\iota(\la)=0$'', i.e., $\la\in\mathscr P_w$ such that $c_i$ is even for all $i$.
Let $u'_{w,+}$ (resp. $u'_{w,-}$) denote the numbers of $J$-invariant unipotent conjugacy classes of $I$ if $\ty(V)=+$ (resp. $-$).
Then $u'_{w,+}+u'_{w,-}$ is equal to the coefficient of $t^w$ in 
\begin{equation}\label{equ-J-inv-2}
(1+2\sum_{k=1}^\infty t^{2k} )(\sum_{k=0}^\infty t^{4k})(1+2\sum_{k=1}^\infty t^{6k} )(\sum_{k=0}^\infty t^{8k})\cdots=\prod_{k=1}^\infty \frac{1}{(1-t^{4k-2})^2(1-t^{8k})},
	\addtocounter{thm}{1}\tag{\thethm}
\end{equation}
while $u'_{w,+}-u'_{w,-}$ is equal to the coefficient of $t^w$ in (\ref{equ-5.1-1}).
In particular, $u'_{w,+}\ne u'_{w,-}$ occurs only when $4\mid w$.
Note that $\prod\limits_{k=1}^\infty \frac{1}{1-t^{4k-2}}=\prod\limits_{k=1}^\infty (1+t^{2k})$, therefore, by (\ref{equ-J-inv}), (\ref{equ-J-inv-1}), (\ref{equ-J-inv-2}) and (\ref{equ-5.1-1}), we have that $\omega_{w,\eps}'=u'_{w,\eps}$ for $\eps\in\{\pm\}$.

Now we consider the principal 2-weights of the group $H$, which is determined in the proof of Lemma~\ref{BAWC-SO}.
Note that a principal 2-weight of $I$ either covers a unique weight of $H$ or covers four weights of $\Omega(V)$.
Let $\mathscr W_1$ denote the set of conjugacy classes of principal 2-weights of $H$ which is $I$-stable and covered by a $J$-invariant weight of $I$, let $\mathscr W_3$ denote the set of conjugacy classes of principal 2-weights of $H$ which is not $I$-stable, and let $\mathscr W_2=\Alp(B_0(H))\setminus(\mathscr W_1\coprod\mathscr W_3)$.
Then by above arguments, we have that $|\Xi_1\coprod \Xi_4|=|\mathscr W_1\coprod \mathscr W_3|$.
On the other hand, by the proof of Lemma~\ref{BAWC-SO}, one gets  $|\Xi_1|=|\mathscr W_1|$, and thus $|\Xi_4|=|\mathscr W_3|$ and $|\Xi_2\coprod\Xi_3|=|\mathscr W_2|$.

Up to conjugacy, every weight in $\mathscr W_1$ is $J$-invariant and every weight in $\mathscr W_1$ covers a unique weight of $\Omega(V)$, the weights in $\mathscr W_2$ present in pairs such that for each pair those two weights are fused by $J$ and every weight in $\mathscr W_2$ is $I$-stable, and the weights in $\mathscr W_3$ present in pairs such that for each pair those two weights are fused by $I$ and every weight in $\mathscr W_3$ covered two weights of $\Omega(V)$.

Let $\Theta$ be the canonical bijection between $\Xi$ and $\IBr(B_0(H))$ as in \cite[Lemma~2.1]{Na23}. Of course, $\Theta$ is $J$-equivariant.
Let $\mathscr B_1$, $\mathscr B_2$, $\mathscr B_3$  be the images of $\Xi_1$, $\Xi_2\coprod\Xi_3$, $\Xi_4$ under $\Theta$ respectively.
Then every Brauer character in $\mathscr B_1$ is $J$-invariant, every Brauer character in $\mathscr B_2$ is $I$-invariant, every Brauer character in $\mathscr B_3$ is $\tH$-invariant.
Moreover, the Brauer characters in $\mathscr B_2$  present in pairs such that for each pair those two characters are fused by $\tH$, and the Brauer characters in $\mathscr B_3$  present in pairs such that for each pair those two characters are fused by $I$.
By \cite[Remark~2.6]{De17}, $|\Irr(\Omega(V)\mid\rho)|$ divides $|\IBr(\Omega(V)\mid\Theta(\rho))|$ for any $\rho\in\Xi$.
Thus every Brauer character in $\mathscr B_3$ covers two irreducible Brauer characters of $\Omega(V)$.
Note that the outer automorphism group of $\Omega(V)$ induced by $J$ is isomorphic to a dihedral group of order 8 or a Klein four group according as $\disc(V)=+$ or $-$.
If $\disc(V)=-$, then the restriction of every irreducible Brauer character of $H$ to $\Omega(V)$ is also irreducible. If $\disc(V)=+$, then by Lemma~\ref{lem:dh} (1), the restriction of every irreducible Brauer character of $\mathscr B_1$ to $\Omega(V)$ is also irreducible.

It is clear that $|\mathscr B_i|=|\mathscr W_i|$ for every $i=1,2,3$. Now we establish bijections between those sets.
Let $f_1$ be arbitrary bijection between $\mathscr B_1$ and $\mathscr W_1$ and let $f_3$ be any $I$-equivariant bijection between $\mathscr B_3$ and $\mathscr W_3$.
By Proposition~\ref{BAWC-Spin}, the number of Brauer character in $\mathscr B_2$ which covers two irreducible Brauer characters of $\Omega(V)$ is equal to the number of conjugacy classes of weights in $\mathscr W_2$ that covers two weights of $\Omega(V)$.
Therefore, there exists a $\tH$-equivariant bijection $f_2:\mathscr B_2\to\mathscr W_2$ such that $\Res^{H}_{\Omega(V)}(\psi)$ is irreducible if and only if $f_2(\psi)$ covers a unique weight of $\Omega(V)$ for any $\psi\in\mathscr B_2$.
Combining $f_1$, $f_2$ and $f_3$, we get a bijection $f$ between $\IBr(B_0(H))$ and $\Alp(B_0(H))$.

Let $\disc(V)=-$. Then $H=\{\pm I_V\}\ti \Omega(V)$ and so it suffices to prove that $f$ is $J$-equivariant.
Note that $\mathscr B_3$ and $\mathscr W_3$ are empty, and by the proof of Proposition~\ref{prop:act-IBr-SO}, the group $I$ acts trivially on $\IBr(B_0(H))$ and $\Alp(B_0(H))$.
So $f$ is $J$-equivariant.

Now suppose that $\disc(V)=+$, then outer automorphism group of $\Omega(V)$ induced by $J$ is isomorphic to a dihedral group of order 8.
We will use Lemma~\ref{lem:dh} (2) to establish this assertion, where we take $W_1$ and $W_2$ as $\IBr(B_0(\Omega(V)))$ and $\Alp(B_0(\Omega(V)))$ respectively.
The outer automorphism group of $\Omega(V)$ induced by $H$ (resp. $I$) is interpreted as $\langle a^2\rangle$ (resp. $\langle a^2,b\rangle$).
Then conditions (a) and (b) of Lemma~\ref{lem:dh} (2) are clear.
For (c), we note that every element of  $\mathscr B_1\cup\mathscr B_2\cup\mathscr W_1\cup\mathscr W_2$ is $I$-invariant, while every element of $\mathscr B_3\cup\mathscr W_3$ is not $I$-invariant.
This completes the proof.
\end{proof}

\begin{thm}\label{thm:IBAW-D}
	Suppose that $\dim(V)$ is even and $P\Omega(V)$ is a simple group satisfying that $w\ne 8$ or $\disc(V)=-$.
If every element of $\IBr(B_0(G))$ is $\langle F_0\rangle$-invariant, then the inductive BAW condition holds for the principal 2-block of $P\Omega(V)$.
\end{thm}

\begin{proof}
Under the assumption, we have that $J\rtimes\langle F_0\rangle$ induces all automorphisms on $\Omega(V)$.
We use the version \cite[Thm.~2.2]{FLZ23b} of the criterion of inductive BAW condition of Brough--Sp\"ath \cite{BS22}. 
We remark that the condition (iv.b) can be replaced by: $(\varphi_0)^0$ extends to $(G\rtimes D)_{\varphi_0}$.
Take $G$, $\tG$, $D$, $\mathcal B$ in \cite[Thm.~2.2]{FLZ23b} to be $P\Omega(V)$, $\tH/Z(\tH)$, $\langle \tau, F_0\rangle$, $B_0(P\Omega(V))$ respectively, where $\tau\in I\setminus H$ is of order~2 so that $J=\langle \tH,\tau\rangle$. 
However, by the hypothesis and Corollary~\ref{cor:act-field-cla}, every element in $\IBr(B_0(P\Omega(V)))\cup\Alp(B_0(P\Omega(V)))$ is $\langle F_0\rangle$-invariant, so is every element in $\IBr(B_0(\tH))\cup\Alp(B_0(\tH))$.
Thus by Theorem~\ref{thm:eq-bij-spin}, there exists a $((\tH/Z(\tH))\rtimes\langle \tau,F_0\rangle)$-equivariant bijection between $\IBr(B_0(P\Omega(V)))$ and $\Alp(B_0(P\Omega(V)))$.

Now $\tilde{\mathcal B}=B_0(\tilde H/Z(\tilde H))$ is the unique 2-block of $\widetilde H/Z(\widetilde H)$ covering $B_0(P\Omega(V))$.
As the $J\langle F_0\rangle/Z(J)\Omega(V)$ has cyclic Sylow $\ell$-subgroups for any odd prime $\ell$, the maximal extendibility properties of Brauer characters hold (see e.g. \cite[Thm.~8.11 and Thm.~8.29]{Na98}).
Therefore, the condition (i) of \cite[Thm.~2.2]{FLZ23b} holds.
By similar arguments as in the proof of \cite[Lemma~2.6]{FLZ19}, every irreducible Brauer character $\psi$ of $B_0(P\Omega(V))$, the set $\IBr(\tH/Z(\tH)\mid\psi)\cap\IBr(B_0(\tH/Z(\tH)))$ is a singleton.
On the other hand, each weight of $B_0(P\Omega(V))$ is covered by a unique weight of $B_0(\tH/Z(\tH))$.
From this, there exists a $(\langle \tau,F_0\rangle)$-equivariant bijection between $\IBr(\tilde{\mathcal B})$ and $\Alp(\tilde{\mathcal B})$ which preserves blocks, and thus the condition (ii) of \cite[Thm.~2.2]{FLZ23b} is satisfied.
Moreover, by \cite[Prop.~2.5]{FLZ23b},  the condition (v) of \cite[Thm.~2.2]{FLZ23b} holds.

For the conditions (iii) and (iv) of \cite[Thm.~2.2]{FLZ23b}, we mention that the maximal extendibility properties of (Brauer) characters hold as above, and so it suffices to consider the conditions (iii.a) and (iv.a).
As the actions of $\langle F_0\rangle$ on $\IBr(B_0(\Omega(V)))$ and $\Alp(B_0(\Omega(V)))$ are trivial, we may omit $F_0$ and just let $D=\langle\tau\rangle$ and $\disc(V)=+$.
The actions  of $\langle\tau\rangle$ on $\IBr(B_0(\Omega(V)))$ and $\Alp(B_0(\Omega(V)))$ are determined in the proof of Theorem~\ref{thm:eq-bij-spin}.
It follows from the results there that the conditions (iii.a) and (iv.a) of \cite[Thm.~2.2]{FLZ23b} are satisfied.
\end{proof}

\begin{prop}\label{prop-BAW-Spin8}
The inductive BAW condition holds for the principal 2-block of $P\Omega_8^+(q)$ (with odd $q$).
\end{prop}

\begin{proof}
	Let $w=8$ and $\disc(V)=+$.  Then $\Out(P\Omega(V))\cong \fS_4\ti \langle F_0\rangle$.
We use the criterion of inductive BAW condition in the proof of Theorem~\ref{thm:IBAW-D}. We remark that by the arguments of \cite[\S4]{BS22} the assumption that $D$ is abelian in \cite[Thm.~2.2]{FLZ23b} can be replaced by the following statement: $D$ is isomorphic to the direct product of a cyclic group and the symmetric group $\fS_3$.
The condition (iii) of \cite[Thm.~2.2]{FLZ23b} follows by Theorem~\ref{A-infinity}. 
So by the proof of Theorem~\ref{thm:IBAW-D}, we only need to  establish an $\Aut(P\Omega(V))$-equivariant bijection between $\IBr(B_0(P\Omega(V)))$ and $\Alp(B_0(P\Omega(V)))$, as well as verify the condition (iv) of \cite[Thm.~2.2]{FLZ23b}.

In the notation of \cite[Thm.~2.2]{FLZ23b}, let $G=P\Omega(V)$ and $D\cong\fS_3\times\langle F_0\rangle$.
By Corollary~\ref{cor:act-field-cla} and the action of $\Ga$ on $\Alp(B_0(P\Omega(V)))$ determined below, the weight $(R,\varphi_0)$ in condition (iv) can be taken such that condition (iv.a) holds and
\(
(G\rtimes D)_{R,\varphi_0}/N_G(R)
\cong\fS_3\times\langle F_0\rangle\)
or \(Z_2\times\langle F_0\rangle.
\)
Moreover, $(\varphi_0)^0$ extends to $(G\rtimes D)_{R,\varphi_0}$.
Thus the condition (iv) of \cite[Thm.~2.2]{FLZ23b} holds.

Next, we establish an $\Aut(P\Omega(V))$-equivariant bijection between the sets $\IBr(B_0(P\Omega(V)))$ and  
$\Alp(B_0(P\Omega(V)))$. In the following we will use the notation in \S\ref{SS:Spin8}. By the proofs of 
Proposition~\ref{prop:act-IBr-SO} and Theorem~\ref{thm:eq-bij-spin}, the action of $D\cong D_8$ on $\IBr(B_0(P\Omega(V)))$, 
and we will describe this action as follows. 

Note that $A_{\bG}(u)$ is abelian for all unipotent elements of $\bG$ as $w=8$, and thus the result of Chaneb \cite[Thm.~2.9]{Ch20} also holds for $G=\Spin_8^+(q)$, which is analogous as in the proof of Corollary~\ref{cor-BAW-Spin7}.
We recall the unitriangular basic set $\Xi$ of $H$ from the proof of Proposition~\ref{prop:act-IBr-SO}, and let $\Xi'_i$ be the set of irreducible constituents of the restrictions of the characters in $\Xi_i$ to $\Omega(V)$ for $1\le i\le 4$.
Then $\Xi'=\coprod\limits_{i=1}^4\Xi_i'$ is a unitriangular basic set of $\Omega(V)$.
So by \cite[Thm.~A]{Sp23} and Corollary~\ref{cor:act-Br-SO}, $\langle F_0\rangle$ acts trivially on $\IBr(B_0(\Omega(V)))$.
By Corollary~\ref{cor:act-field-cla}, $\langle F_0\rangle$ acts trivially on $\Alp(B_0(P\Omega(V)))$.

For $1\le i\le 4$, let $\mathscr B'_i$ be the subset of $\IBr(B_0(\Omega(V)))$ corresponding to $\Xi_i'$ under the canonical bijection $\Xi'\to\IBr(B_0(\Omega(V)))$. 
Direct calculations show that $|\mathscr B_1'|=4$, $|\mathscr B_2'|=4$, $|\mathscr B_3'|=12$, $|\mathscr B_4'|=8$.
Then every Brauer character in $\mathscr B'_1$ is $D$-invariant, the set $\mathscr B_2'$ consists of two $D$-orbits of length two and $E$ acts trivially on $\mathscr B_2'$, the set $\mathscr B_3'$ consists of three $D$-orbits of length 4 and $N$ acts transitively on each $D$-orbits, and the set $\mathscr B_4'$ consists of two $D$-orbits of length 4 and $E$ acts transitively on each $D$-orbit.

Now regard $\mathscr B_i'$ (with $1\le i\le 4$) as Brauer characters of $\IBr(B_0(P\Omega(V)))$.
Similar as in the proof of Proposition~\ref{P:weight-r2} (c) we can show that every Brauer character in $\mathscr B'_1$ is $\Ga$-invariant.
Similar as in the proof of Proposition~\ref{P:weight-r2} (a) we know that $\mathscr B_2'\cup\mathscr B_4'$ consists of two $\Ga$-orbits of length 6, and in each orbit there is a Brauer character $\psi$ with $\Ga_\psi=E$.
Therefore, the $D$-orbits in $\mathscr B_3'$ must be $\Ga$-stable, as otherwise $\mathscr B_3'$ consists of one $\Ga$-orbit and thus there exists at most two Brauer characters $\psi$ in $\mathscr B_3'$ with $\Ga_\psi=\langle(1,2)\rangle$ which contradicts the fact that in each $D$-orbit in $\mathscr B_3'$ there is at least one such $\psi$.
Since there is a unique conjugacy class of subgroups of $\Ga$ of order 6, the action of $\Ga$ in $\mathscr B_3'$ is clear.
When $a=2$, $R_{30}$ provides $6$ principal weights, which form two
$\Ga$-orbits of length one and one $\Ga$-orbit of length four.
Therefore, by Theorem~\ref{thm:eq-bij-spin} and the same arguments for the action of $\Ga$ on $\Alp(B_0(P\Omega(V)))$, we get an $\Aut(P\Omega(V))$-equivariant bijection between $\IBr(B_0(P\Omega(V)))$ and $\Alp(B_0(P\Omega(V)))$.
\end{proof}

\begin{cor}\label{que:action-Br-uni}
Every Brauer character in $\IBr(B_0(G))$ is $\langle F_0\rangle$-invariant.
Moreover, there is an $\Aut(G)$-equivariant bijection between $\IBr(B_0(G))$ and the unipotent conjugacy classes of $G$.
\end{cor}

\begin{proof}
By Theorem~\ref{A-infinity} and Corollary~\ref{cor:act-Br-SO}, every Brauer character in $\IBr(B_0(G))$ is $\langle F_0\rangle$-invariant.
By above arguments and analyzing the action of $\Aut(G)$ on the unipotent conjugacy classes of $G$, we can deduce this assertion.
\end{proof}

We propose the following question.

\begin{ques}\label{que:unitriangular}
Let $\Xi'$ be the set of irreducible constituents of the restrictions of the characters in $\Xi$ to $\Omega(V)$.
Is $\Xi'$ a unitriangular basic set for $B_0(\Omega(V))$?
\end{ques}

Now we are able to prove Theorem~\ref{mainthm:BAW-BD}.

\begin{proof}[Proof of Theorem~\ref{mainthm:BAW-BD}]
	Thanks to \cite[Thm.~C]{Sp13}, we may assume that $q$ is odd.
	According to \cite[Theorem~5.7]{FLZ22} and Theorem~\ref{A-infinity}, we only need to deal with quasi-isolated 2-blocks, as the inductive BAW condition has been verified for groups of type $\tp A$ in \cite{FLZ23}.
	By \cite{Bo05}, quasi-isolated 2-blocks of $G$ are just unipotent 2-blocks, and by \cite[Thm.~21.14]{CE04}, the principal 2-block of $G$ is the unique unipotent 2-block.
	
We exclude the simple groups $P\Omega_8^+(q)$, as they are dealt with in Proposition~\ref{prop-BAW-Spin8}.
By Theorem~\ref{thm:IBAW-B} and~\ref{thm:IBAW-D}, it suffices to show that every Brauer character in $\IBr(B_0(\Omega(V)))$ is $\langle F_0\rangle$-invariant, which follows by Corollary~\ref{que:action-Br-uni} directly.
Thus we complete the proof.	
\end{proof}

\begin{cor}
	If $\dim(V)$ is even and $\disc(V)=-$, then the inductive BAW condition holds for the simple groups $\Omega(V)$ and the prime~2.	
\end{cor}	

\begin{proof}
Suppose that $\dim(V)$ is even and $\disc(V)=-$. Then $\SO(V)=\{\pm I_V\}\ti\Omega(V)$ and the non-generic Schur multipliers do not occur. 
According to Corollary~\ref{cor:act-Br-SO}, every element in $\IBr(B_0(\Omega(V)))$ is $\langle F_0\rangle$-invariant. It follows from Theorem~\ref{thm:IBAW-D} that the inductive BAW condition holds for the principal 2-block of $\Omega(V)$.
Moreover, by Theorem~\ref{mainthm:BAW-BD}, Theorem~\ref{A-infinity} completes the proof.
\end{proof}

\subsection{The principal 2-block of $\tp F_4(q)$}\label{SS:prin}

Now assume further that $\bG$ is of type $\tp F_4$ so that $G=\bG^F=\tp F_4(q)$ and assume that $p=2$.  Here we suppose that $q$ is odd. 
Let $q=r^f$ where $r$ is an odd prime, and let $F_0$ be a field automorphism of $\bG$ such that $\bG^{F_0}=\tp F_4(r)$ and $F=F_0^f$.
Then \[\Aut(G)\cong G\rtimes \langle F_0\rangle.\] Note that $\bG^*=\bG$. 

We will write the principal 2-block $B_0(G)$ of $G$ simply $B_0$ when no confusion can arise.
Keep the notation of \S\ref{S:F4}. If $(R,\varphi)$ is a principal weight of $G$, then by Lemma \ref{P:r=0},  $\rank B(R)=1$ or $2$.
For $r=1,2$, let $\Alp(B_0)_r$ be the subset of $\Alp(B_0)$ consisting of $\overline{(R,\varphi)}$ such that the rank of $B(R)$ is $r$.

\subsubsection{The case $r=1$}
First let $R$ be a radical 2-subgroup of $G$ with $\rank B(R)=1$ and $H:=N_G(B(R))$. From \S\ref{SS:Spin9}, we know that $H=Z_G(B(R))\cong\Spin_{9,+}(q)$, $RZ_H(R)\le H$ and $N_G(R)=N_H(R)$.
Let $\varphi\in\dz(N_G(R)/R)$. Then by Lemma~\ref{lem:B_0-wei}, $(R,\varphi)$ is a principal weight of $G$ if and only if it is a principal weight of $H$.
From this we deduce that up to conjugacy $\Alp(B_0)_1$ is a subset of $\Alp(B_0(H))$ consisting of $\overline{(R,\varphi)}$ such that the rank of $B(R)$ is $1$.
By the arguments preceding \S\ref{SS:Spin8}, $R$ is conjugate to one of $R_i$ with $1\le i\le 21$ and $i\notin\{2,4,5,6,12\}$.
We identify the principal weights of $H$ with the principal weights of $H/Z\cong\Omega_{9,+}(q)$. By \cite[Lemma~2.7 and Cor.~2.12]{BS22}, the principal weights of $\Omega_{9,+}(q)$ are the weights of $\Omega_{9,+}(q)$ covered by the principal weights of $\rO_{9,+}(q)$. According to \cite[\S6]{An93a}, every principal weight subgroup of $\rO_{9,+}(q)$ provide only one principal weight.
So we can count the number of conjugacy classes of principal weights in $\Alp(B_0)_1$ by Corollary \ref{cor:wei-cov-solquo} and Table \ref{Ta2}, i.e., $|\Alp(B_0)_1|=19$.

\subsubsection{The case $r=2$}
Now let $R$ be a principal weight 2-subgroup of $G$ with $\rank B(R)=2$ and let $H:=N_G(B(R))$ and $L:=Z_G(B(R))$. Then $R$ is conjugate to one of $R_{22}$--$R_{37}$ in Table \ref{Ta3}.
From \S\ref{SS:Spin8}, we know that $L\cong \Spin_{8,+}(q)$, $H/L\cong \fS_3$, $RZ_H(R)\le L$ and $N_G(R)=N_H(R)$. Let $\varphi\in\dz(N_G(R)/R)$. Then by Lemma~\ref{lem:B_0-wei}, $(R,\varphi)$ is a principal weight of $G$ if and only if it is a principal weight of $H$. Therefore up to conjugacy $\Alp(B_0)_2$ is the subset of $\Alp(B_0(H))$ consisting of $\overline{(R,\varphi)}$ with $R\subset L$.
According to \cite[Lemma~2.7]{BS22}, if $R$ is a principal weight subgroup of $H$, then it is also a principal weight subgroup of $L$.
More precisely, if $(R,\varphi)$ is a principal weight of $H$, then by Lemma~\ref{ext-rdz}, $(R,\vartheta)$ is a principal weight of $L$ for any $\vartheta\in\Irr(N_L(R)\mid\varphi)$.

Conversely, let $(R,\vartheta)$ be a principal weight of $L$. If $(R,\varphi)$ is a weight of $H$ covering $(R,\vartheta)$, then $\varphi\in\rdz(N_H(R)\mid \vartheta)$, and moreover, by Lemma~\ref{ext-rdz}, $O_2(N_H(R)_\vartheta/N_L(R))=1$.   
As $RZ_H(R)\le L$, if such a weight $(R,\varphi)$ exists, then it is necessarily a principal weight of $H$.

The following Lemma \ref{ext-S3S3} will be also used in the proof of Proposition~\ref{prop:bij-quasi-isolated}.

\begin{lem}\label{ext-S3S3}
	Let $Y\unlhd X$ be arbitrary finite groups such that $X/Y\cong\fS_3$ or $\fS_3\ti \fS_3$, and let $\chi$ be an $X$-invariant irreducible character of $Y$.
Then $|\rdz(X\mid\chi)|=1$.
\end{lem}

\begin{proof}	
First we assume that $X/Y\cong\fS_3$.
Then by \cite[Thm. 35.10]{Hu98}, $\chi$ extends to $\tilde\chi\in\Irr(X)$ and then $\rdz(X\mid\chi)$ consisting of the unique character $\tilde\chi \ka$ where $\ka$ is the unique character of $\fS_3$ of degree 2.

Now let  $X/Y\cong\fS_3\ti \fS_3$. We let $V$ be a normal subgroup of $X$ containing $Y$ such that $V/Y\cong\fS_3$ and $X/V\cong\fS_3$.
Then by the above argument, $|\rdz(V\mid\chi)|=1$ and we write $\rdz(V\mid\chi)=\{ \eta\}$. Since $\chi$ is $X$-invariant, it follows that $\eta$ is also $X$-invariant.
Thus by the above argument again, $|\rdz(X\mid\eta)|=1$. This completes the proof as $\rdz(X\mid\eta)=\rdz(X\mid\chi)$ by construction.
\end{proof}

\begin{lem}\label{lem:num-S3}
	Let $(R,\vartheta)$ be a principal weight of $L$. Then $|\rdz(N_H(R)\mid \vartheta)|=1,0,3,1$ if $|N_H(R)_\vartheta/N_L(R)|=1,2,3,6$ respectively.
\end{lem}

\begin{proof}
As $N_H(R)_\vartheta/N_L(R)$ is isomorphic to a subgroup of $\fS_3$. If $|N_H(R)_\vartheta/N_L(R)|=6$, the this follows by Lemma~\ref{ext-S3S3}.
If $|N_H(R)_\vartheta/N_L(R)|<6$, then $N_H(R)_\vartheta/N_L(R)$ is cyclic and $\vartheta$ extends to $N_H(R)_\vartheta$. So this lemma follows by Gallagher's Theorem.
\end{proof}	

For the representatives $R$ in Proposition~\ref{P:weight-r2}, each $R$ provides
a unique principal weight $(R,\vartheta)$ of $L$, and
$N_H(R)_\vartheta=N_H(R)$, except when $a=2$ and $R=R_{30}$.
When $a=2$, $R_{30}$ provides $6$ principal weights of $L$.
The action of $S$ on these weights has three orbits of length one and one orbit of length three.
By \cite[Lemma~2.7]{BS22} and Lemma~\ref{lem:num-S3}, we have the following.

\begin{prop}
Let $R$ be one of the representatives in Proposition~\ref{P:weight-r2}.
When $a>2$, $R$ provides a unique principal weight of $H$ if
$R\in\{R_{22},R_{24},R_{28},R_{29},R_{30},R_{31},R_{35}\}$,
and no principal weight of $H$ otherwise.
When $a=2$, each of $R_{22},R_{24},R_{28},R_{29}$ provides a unique principal weight of $H$,
and $R_{30}$ provides three principal weights of $H$.
All other representatives provide no principal weight of $H$.
\end{prop}

From this, we can count the number of conjugacy classes of principal weights in $\Alp(B_0)_2$ by \S\ref{SS:Spin8}, i.e., $|\Alp(B_0)_2|=7$.

Now we can establish the inductive BAW condition for the block $B_0$. 

\begin{thm}\label{thm-BAW-principal}
The	inductive BAW condition holds for the principal 2-block $B_0$.
\end{thm}

\begin{proof}
Note that $B_0$ is $\Aut(G)$-stable and $\Out(G)$ is cyclic. 
To verify the inductive BAW condition holds for the block $B_0$, it suffices to establish an $\Aut(G)$-equivariant bijection between $\IBr(B_0)$ and $\Alp(B_0)$ (see for example \cite[Lemma~2.10]{Sc16}). 

According to \cite[Prop.~7.14]{Ge18}, the number of irreducible Brauer characters in the unipotent blocks $\cE_2(G,1)$ is $28$. 
By Th\'eor\`eme A and the table on page 349 of \cite{En00}, the group $G$ has three unipotent 2-blocks: the principal block and two unipotent blocks of defect zero. So $|\IBr(B_0)|=26$.
According to \cite[Prop.~3.7 and 3.9]{Ma07} and \cite[Prop.~7.14]{Ge18}, we see that $\langle F_0\rangle$ acts trivially on $\IBr(B_0)$.

By above arguments, $|\Alp(B_0)|=|\Alp(B_0)_1|+|\Alp(B_0)_2|=26$, and thus $|\IBr(B_0)|=|\Alp(B_0)|$, i.e., the blockwise Alperin weight conjecture holds for the block $B_0$. To verify the inductive BAW condition, we only need to prove that $\langle F_0\rangle$ acts trivially on $\Alp(B_0)$.
First, $\langle F_0\rangle$ acts trivially on $\Alp(B_0)_1$ by Corollary~\ref{cor:act-field-cla}, as $\Alp(B_0)_1$ can be regarded as a subset of $\Alp(B_0(\Spin_{9,+}(q)))$ by construction.
Let $(R,\varphi)$ be a principal weight of $G$ with $\rank B(R)=2$.
Then $(R,\vartheta)$ is a principal weight of $L$ covered by $(R,\varphi)$ for any $\vartheta\in\Irr(N_L(R)\mid\varphi)$.
By Corollary~\ref{cor:act-field-cla}, $\overline{(R,\vartheta)}$ is $\langle F_0\rangle$-invariant.
On the other hand, $\overline{(R,\varphi)}$ is the unique element of $\Alp(B_0(H))$ covering $\overline{(R,\vartheta)}$ by construction. From this $\overline{(R,\varphi)}$ is also $\langle F_0\rangle$-invariant.
Then $\langle F_0\rangle$ acts trivially on $\Alp(B_0)_2$, and this completes the proof.
\end{proof}

\subsection{Non-unipotent quasi-isolated 2-blocks of $\tp F_4(q)$}

The group $\bG^F=\tp F_4(q)$ possesses a non-unipotent quasi-isolated 2-block only when $3\nmid q$.

Assume that $3\nmid q$. Then there exists an element $s_0\in\bG^F$ of order 3 such that $\C_{\bG}(s_0)^F$ is an extension of 
$(\SL_{3,long}(\eps q)\times\SL_{3,short}(\eps q))/\langle(\delta_{0}I,\delta_{0}^{-1}I)\rangle$ by $Z_3$, where $\eps=\pm 1$ 
is the sign satisfying $3\mid (q-\eps)$, 
$\delta_0$ is an element of order $3$ in $\GL_1(\epsilon q)$,
and the coset containing
\[
(\diag\{\eta,\eta,\eta^{-2}\},
 \diag\{\eta^{-1},\eta^{-1},\eta^2\})
\]
is a generator of the quotient $Z_3$, where
$\eta^{q-\epsilon}=\delta_0$.
By the classification of quasi-isolated semisimple elements of reductive groups of Bonnaf\'e \cite{Bo05}, up to conjugacy, $s_0$ is the unique nonidentity
quasi-isolated semisimple element of $\bG$ of odd order. 

Write $\bH=Z_{\bG}(s_0)$, $H=\bH^F$ and $R_0\in\Syl_2(H)$.
Let $B=\cE_2(G,s_0)$. Then by \cite[Table~2]{KM13}, $B$ is a single block of $G$, that is, the unique non-principal quasi-isolated 2-block of $G$, and by \cite[Prop.~6.1]{Ru22}, $R_0$ is a defect group of $B$. Let $B'=\cE_2(H,s_0)$.
Note that $B'=\cE_2(H,s_0)$ is a single block by \cite[Thm.~21.14]{CE04}, as any component of $\bH$ is of classical type.

\subsubsection{The $B'$-weights of $H$}

\begin{lem}\label{lem:rad-SL3}
Every radical 2-subgroup of $\SL_3(\eps q)$ is conjugate to one of the following.
	\begin{enumerate}[\rm(1)]
		\item $P_1=1$,
		\item $P_2=\iota(R_{2,0,0}^1)$,
		   \item $P_3=\iota(R_{1,0,0}^1\ti R_{1,0,0}^1)$,
		   \item $P_4=\iota(R_{1,1,0}^1)$,
		   \item $P_5=\iota(R_{1,0,1}^{1,-})$,
		   \item $P_6=\iota(R_{1,0,0,(1)}^1)$ or $\iota(R_{1,0,1}^2)$ according as $\eps=\varepsilon$ or $\eps\ne\varepsilon$. In this situation, $P_6$ is a Sylow 2-subgroup of $\SL_3(\eps q)$.
	   \end{enumerate}	
Here $\iota$ is the embedding $\GL_2(\eps q)\embed\SL_3(\eps q)$, $g\mapsto\diag(g,\det(g)^{-1})$.
\end{lem}
 
\begin{proof}
This follows from Lemma~\ref{L:R3} and the construction of radical subgroups of  $\GL_3(\eps q)$ in \S \ref{SS:GL2}. 
See also \cite[\S4]{FLZ21}.
\end{proof}

Let $L=\SL_3(\eps q)$. Then by the explicit description of the weights of special linear and unitary groups $\SL_n(\eps q)$ in \cite[\S5]{FLZ21}, one gets $|\Alp(B_0(L))|=3$. Precisely, each of the radical subgroups $P_3$, $P_5$ and $P_6$ (in Lemma~\ref{lem:rad-SL3}) provides a unique principal weight of $L$.	

Note that $B'=\hat s_0\otimes \cE_2(H,1)$ and $\cE_2(H,1)$ is the principal block of $H$.  In particular, $R_0=P_6\ti P_6$ is a defect group of $B'$.
We have $J\unlhd H$ and $H/J\cong Z_3$, where
$J\cong\SL_3(\eps q)\circ_3\SL_3(\eps q)$.
Hence every radical $2$-subgroup of $H$ is contained in $J$.

\begin{lem}\label{lem:wei-B'-iso}
$|\Alp(B')|=9$. Precisely, each of the radical subgroup $R=R_1\ti R_2$, where $R_1,R_2\in\{P_3, P_5, P_6\}$, provides a $B'$-weight of $H$.

In particular, if $(R,\varphi)$ and $(R',\varphi')$ are two $B'$-weights of $H$ that are not $H$-conjugate, then the weight subgroups $R$ and $R'$ are not $H$-conjugate.
\end{lem}	

\begin{proof}
Note that $J$ is contained in the kernel of the linear character $\hat s_0$. So the block $B'$ covers the principal block $B_0(J)$ of $J$.
In addition, there are exactly three blocks of $H$ covering $B_0(J)$.
So $(R,\Res^{N_H(R)}_{N_J(R)}(\varphi))$ form a complete set of representatives of the $J$-conjugacy classes of $B_0(J)$-weights of $J$ if $(R,\varphi)$ runs through a complete set of representatives of the $H$-conjugacy classes of $B'$-weights of $H$.
In particular, $|\Alp(B')|=|\Alp(B_0(J))|$.
Therefore, this lemma follows from Lemma~\ref{lem:rad-SL3} immediately.
\end{proof}	

\begin{rmk}\label{rmk-wei-B'-iso}
In Lemma~\ref{lem:wei-B'-iso}, one shows that  if $(R,\varphi)$ and $(R',\varphi')$ are two $B'$-weights of $H$ that are not $H$-conjugate, then $R$ and $R'$ are not $\Aut(G)$-conjugate.
\end{rmk}	

\subsubsection{Classifying the $B$-weights of $G$}

Take $R_{0}$ a Sylow 2-subgroup of \[\GL_{2,long}(\epsilon q)\cdot\GL_{2,short}(\epsilon q)\subset\SL_{3,long}(\epsilon q)\cdot
\SL_{3,short}(\epsilon q)\subset Z_{G}(s_{0}).\] Let \[A_{0}=\Omega_{1}(Z(R_{0})),\quad D_{0}:=Z_{G}(R_{0}),\quad D_{1}=
Z_{G}(Z(R_{0})).\]    

Note that the invariants $r$ and $s$ for $A_{0}$ are both equal to 1, i.e., $A_{0}\sim F_{1,1}$ (cf. beginning of \S \ref{S:F4}). 
We have \[Z_{G}(A_{0})\cong((\Sp_{4}(q)\times\Sp_{2}(q)\times\Sp_{2}(q))/\langle(-I_{4},-I_{2},-I_{2})\rangle)
\rtimes\langle\tau\rangle,\] where \[\tau=(\left(\begin{array}{cc}&\eta I_{2}\\-\eta^{-1}I_{2}&\\\end{array}\right),
\left(\begin{array}{cc}&\eta\\-\eta^{-1}&\\\end{array}\right),\left(\begin{array}{cc}&\eta\\-\eta^{-1}&\\\end{array}\right))\] 
with $\eta\in\GL_{1}(q^{2})_{2}$ and $\eta^{2}$ a generator of $\GL_{1}(q)_{2}$. Thus, $\tau^{2}=1$ and 
\[(\Sp_{4}(q)\times\Sp_{2}(q)\times\Sp_{2}(q))^{\tau}=\GL_{2}(-\varepsilon q)\times\GL_{1}(-\varepsilon q)\times
\GL_{1}(-\varepsilon q).\] Note that $s_{0}\in Z_{G}(A_{0})$ and $Z_{Z_{G}(A_{0})}(s_{0})=Z_{Z_{G}(s_{0})}(A_{0})$ is a central 
extension of $\GL_{2,long}(\epsilon q)\cdot\GL_{2,short}(\epsilon q)$ by $Z_{3}$. Up to a permutation of the second and the 
third components, we may assume that $s_{0}$ is conjugate to \[(\left(\begin{array}{cc}&I_{2}\\-I_{2}&-I_{2}\\\end{array}\right),
\left(\begin{array}{cc}&1\\-1&-1\\\end{array}\right),I_{2}).\]  

\begin{lem}\label{lem:cen-ele}
Let $A$ be an elementary abelian 2-subgroup with $A_0\subset A\subset R_0$.
Then $Z_{\bH}(A)$ is connected, and $Z_{\bH}(A)$ is a Levi subgroup of $Z_{\bG}^\circ(A)$.
\end{lem}

\begin{proof}
Under the notation of \S\ref{S:F4}, the centralizer of $F_{r,s}$ (with $0\le r\le 2$ and $0\le s\le 3$) in $\bG$ is connected if and only if $s=0$ or 1. 
Moreover, $A_0\sim F_{1,1}$. 

Let $A$ be an elementary abelian 2-subgroup with $A_0\subset A\subset R_0$.
Then $Z_{\bH}(A)$ is connected, and $A\sim F_{r,s}$ with $r,s\in\{1,2\}$.
This lemma can be checked case by case.
\end{proof}

\begin{lem}\label{lem:weight-subgp}
	If $R$ is a $B$-weight subgroup of $G$ with $R\subset R_0$, then $R$ is a $B'$-weight subgroup of $H$. 
\end{lem}

\begin{proof}
Let $(R_0,b_{R_0})$ be a maximal $B$-Brauer pair of $G$. By Lemma~\ref{lem:center-weisub}, we have $Z(R_0)\subset R\subset R_0$.
Let $A=\Omega_1(Z(R))$. 
We have unique Brauer pairs $(A, b_{A})$ and $(R,b_R)$ such that 
\[ (1,B)\le (A, b_{A})\le (R,b_R)\le  (R_0,b_{R_0}).\]
Then by \cite[Thm.~4.14]{BDR17}, up to conjugacy, we may assume that $b_{A}$ is the unique block of $Z_G(A)$ covering $\cE_2(Z_{\bG}^\circ(A)^F,s_0)$.
Note that $\cE_2(Z_{\bG}^\circ(A)^F,s_0)$ is a single block since any component of $Z_{\bG}^\circ(A)$ is of classical type.

By construction, $N_{Z_G(A)}(R)\unlhd N_G(R)$, and thus by Lemma~\ref{lem:weight-subgp-sub}, $R$ is also a $b_A$-weight subgroup of $Z_G(A)$.
By Lemma~\ref{lem:weight-subgp-sub} again, $R$ is also a $\cE_2(Z_{\bG}^\circ(A)^F,s_0)$-weight subgroup of $Z_{\bG}^\circ(A)^F$.
According to the Jordan decomposition of weights (cf. \cite[Thm.~4.15]{FLZ22}) and Lemma~\ref{lem:cen-ele}, this implies that $R$ is a $\cE_2(Z_H(A),s_0)$-weight subgroup of $Z_H(A)$.
Now $A_0\subset A$, and $Z_H(A_0)$ is an extension of $\GL_{2,long}(\epsilon q)\cdot\GL_{2,short}(\epsilon q)$ by $Z_3$.
The 2-weights of general linear and unitary groups were classified in \cite{An92,An93b}, from which one gets the 2-weight subgroups of $Z_H(A)$.
Thus this lemma follows from Lemma~\ref{lem:wei-B'-iso} and direct calculation.
\end{proof}

Let  $(R_0,b_{R_0})$ be a maximal $B$-Brauer pair of $G$, and let $(R,b_R)$ be any $B$-Brauer pair of $G$ with $A_0\subset R$ and $(R,b_R)\le  (R_0,b_{R_0})$.
We have unique Brauer pair $(A_0, b_{A_0})$ such that 
\[ (1,B)\le (A_0, b_{A_0})\le (R,b_R)\le  (R_0,b_{R_0}).\]
As in the proof of Lemma~\ref{lem:weight-subgp}, we may assume that $b_{A_0}=\cE_2(Z_{G}(A_0),s_0)$ up to $G$-conjugacy.
Thus $(R,b_R)$ is a $B$-Brauer pair of $G$ if and only if it is a $\cE_2(Z_{G}(A_0),s_0)$-Brauer pair of $Z_{G}(A_0)$. 
By \cite[Thm.~7.7]{BDR17}, the block algebras of $\cE_2(Z_{G}(A_0),s_0)$ and $\cE_2(Z_H(A_0),s_0)$ are splendid Rickard equivalent.
Applying Brauer functor and by \cite[Thm.~1.7]{Ha99}, we deduce that for any $\cE_2(Z_{G}(A_0),s_0)$-Brauer pair $(R,b_R)$ of $Z_{G}(A_0)$, there exists a unique $\cE_2(Z_H(A_0),s_0)$-Brauer pair $(R, b_R')$ of $Z_H(A_0)$ such that the block algebras of $b_R$ and $b_R'$ are Rickard equivalent.
Also, $(R,b_R')$ is a $B'$-Brauer pair of $H$.

\begin{lem}\label{lem:nor-brauer-pair}
Keep the hypotheses and setup as above.
\begin{enumerate}[\rm(1)]
	\item 	$(R,b_R)$ is a self-centralizing $B$-Brauer pair of $G$ if and only if $(R,b_R')$ is a self-centralizing $B'$-Brauer pair of $H$.
	\item If $(R,b_R)$ is a self-centralizing $B$-Brauer pair of $G$ such that $R$ is a $B$-weight subgroup of $G$, then $N_G(R,b_R)=Z_G(R)N_{H}(R)$.
\end{enumerate}
\end{lem}

\begin{proof}
(1) follows from the fact that $b_R$ and $b_R'$ are Rickard equivalent.
Now we consider (2) and suppose that $(R,b_R)$ is a self-centralizing and $R$ is a $B$-weight subgroup of $G$ with $R\subset R_0$.
By Lemma~\ref{lem:weight-subgp}, $R$ is a $B'$-weight subgroup of $H$, and thus we see that  $R$ is listed in Lemma~\ref{lem:wei-B'-iso}.

First let $\varepsilon=\epsilon$. 
Then $Z_{\bG}(R)\subset Z_{\bG}(Z(R_0))\subset \bH$.
Now assume that $(R,b_R)$ is a $B$-Brauer pair of $G$; we also regard it as a $B'$-Brauer pair of $H$.
One checks that $Z_{\bG}(R)$ is a torus case by case.
Let $\hat b_R$ be the block of $Z_J(R)$ covered by $b_R$. Then by \cite[Lemma~2.3]{KS15}, $(R,\hat b_R)$ is a $B_0(J)$-Brauer pair of $J$ by Brauer Third Main Theorem, which implies that $\hat b_R$ is the principal block of $Z_J(R)$.
Since $b_R$ is not the principal block of $Z_H(R)$, we can assume that $b_R=\cE_2(Z_G(R),s_0)$ and the linear character $\Res^{H}_{Z_{H}(R)}(\hat s_0)$ is the canonical character of $b_R$.
So using the fact that $Z_{\bG}(R)$ is a torus, we deduce that $N_G(R,b_R)=N_{H}(R)$. 

Next, assume that $\varepsilon\ne\epsilon$.  
The Lusztig induction $\textup{R}^{Z_{\bG}(R)}_{Z_{\bH}(R)}$ induces a Morita equivalence between the block algebras of $b_R$ and $b_R'$ (cf. \cite[Prop.~3.9]{Ru22a}).
In addition, the linear character $\hat s_{0,R}=\Res^{H}_{Z_{H}(R)}(\hat s_0)$ is the canonical character of $b_R'$ and $\pm\textup{R}^{Z_{\bG}(R)}_{Z_{\bH}(R)}(\hat s_{0,R})$ is the canonical character of $b_R$.
We can check case by case that $Z_{\bH}(R)$ is a maximal torus of $Z_{\bG}(R)$.
Similarly, in the proof of the case $\varepsilon=\epsilon$, one shows that $b_R'=\cE_2(Z_{H}(R),s_0)$ and thus $b_R=\cE_2(Z_{G}(R),s_0)$. 
Note that the linear character $\hat s_0$ is stabilized by all elements of $H$.
So $N_{H}(R)\subset N_G(R,b_R)$ and we have $Z_G(R)N_{H}(R)\subset N_G(R,b_R)$.
Conversely, for any $g\in N_G(R,b_R)$, we have that $g$ stabilizes the $Z_G(R)$-conjugacy class of the pair $(Z_{\bH}(R),\hat s_{0,R})$. 
So there exists $c\in Z_G(R)$ such that $gc^{-1}$ stabilizes the pair $(Z_{\bH}(R),\hat s_{0,R})$. 
Let $\mathbf{L}=Z_{\bG}(Z_{\bH}(R))$.
Then $\mathbf{L}$ is a Levi subgroup of $\bG$ contained in $\bH$.
By \cite[(8.19)]{CE04}, $\hat s_0$ defines a linear character of $\mathbf{L}^F$ as $s_0\in Z(\mathbf{L}^*)^F$.
Also, $gc^{-1}$ stabilizes $\hat s_0$ by construction.
Therefore, according to \cite[Prop.~1.9 (i)]{CE99}, $gc^{-1}\in H$ since $\mathbf{L}\subset\bH$.
This implies that $N_G(R,b_R)\subset Z_G(R)N_{H}(R)$, which completes the proof.
\end{proof}

\begin{prop}\label{prop:bij-quasi-isolated}
	$|\Alp(B)|=9$.  Precisely, up to conjugacy, each of the $B'$-weight subgroup of $H$ provides one $B$-weight of $G$.
\end{prop}

\begin{proof}	
	Let $R$ be a $B$-weight subgroup of $G$.
	Then by Lemma~\ref{lem:weight-subgp}, we can assume that $R$ is listed in Lemma~\ref{lem:wei-B'-iso}. 
By Lemma~\ref{lem:nor-brauer-pair}, $(R,b_R)$ is self-centralizing if and only if $(R,b_R')$ is self-centralizing.
Each of these radical 2-subgroups provides a unique self-centralizing $B$-Brauer pair $(R,b_R)$ of $G$.  Denote by $\theta$ the canonical character of $b_R$ and by $\vartheta$ the inflation of $\theta\in\Irr(RZ_G(R)/R)$ to $RZ_G(R)$. By Lemma~\ref{lem:nor-brauer-pair}, $N_G(R)_\vartheta=N_G(R,b_R)=Z_G(R)N_{H}(R)$ which implies that $N_G(R)_\vartheta/RZ_G(R)\cong N_{H}(R)/RZ_H(R)$.
Direct calculations show that $N_H(R)/RZ_H(R)\cong 1$, $\fS_3$ or $\fS_3\ti\fS_3$, and then by Lemma~\ref{ext-S3S3}, one sees that $\rdz(N_G(R)_\vartheta\mid\vartheta)$ has exactly one element, which implies that $R$ provides exactly one $B$-weight of $G$ by Construction~\ref{construction-of-weights}.  This completes the proof.
\end{proof}

\begin{rmk}\label{rmk:Jordan-weights}
In the proof of Proposition~\ref{prop:bij-quasi-isolated}, we indeed get a correspondence between the weights of $B$ and $\cE_2(H,1)$.
Comparing with the Jordan decomposition for characters (\cite[\S15]{CE04}) and for blocks (\cite{BDR17}, \cite{En08}), this can be viewed as a Jordan decomposition for weights of quasi-isolated 2-blocks of groups of type $\tp F_4$.
\end{rmk}

Now we are ready to establish the inductive BAW condition for the quasi-isolated block~$B$.

\begin{thm}\label{thm-BAW-quasi-isolated}
The	inductive BAW condition holds for the block $B$.
\end{thm}	

\begin{proof}	
	Note that $B$ is $\Aut(G)$-stable.
As in the proof of Theorem \ref{thm-BAW-principal}, it is sufficient to show that there is an $\Aut(G)$-equivariant bijection between $\IBr(B)$ and $\Alp(B)$.

By \cite[Prop.~5.1]{AHL21}, $\cE(G,s_0)$ is a basic set of the block $B$ and so $|\IBr(B)|=9$.
In addition, by \cite[Prop.~5.9]{AHL21}, every character in $\cE(G,s_0)$ is $\Aut(G)$-invariant. 
Thus every Brauer character in $\IBr(B)$ is $\Aut(G)$-invariant. 

By Proposition~\ref{prop:bij-quasi-isolated}, $|\Alp(B)|=9$.
Moreover, if $(R,\varphi)$ and $(R',\varphi')$ are $B$-weights that are not $G$-conjugate, then by Remark~\ref{rmk-wei-B'-iso}, $R$ and $R'$ are not $\Aut(G)$-conjugate.
Thus every element in $\Alp(B)$ is $\Aut(G)$-invariant. 

Therefore, there exists an $\Aut(G)$-equivariant bijection between $\IBr(B)$ and $\Alp(B)$ and thus $B$ satisfies the inductive BAW condition.
\end{proof}

Finally, we are able to prove Theorem \ref{mainthm:BAW-F4}.

\begin{proof}[Proof of Theorem \ref{mainthm:BAW-F4}]	
Let $G=\tp F_4(q)$ (with odd $q$).	
By \cite[Thm.~5.7]{FLZ22}, to establish the inductive BAW condition for 2-blocks of $G$, it suffices to consider 
the quasi-isolated 2-blocks for groups of types $\tp F_4$, $\tp A_1$, $\tp A_2$, $^2\tp A_2$, $\tp C_2$, $\tp C_3$, 
$\tp B_3$. Note that the papers \cite{FLZ23}, \cite{FM22} treat the types $\tp A_n$, $^2\tp A_n$ and $\tp C_n$, while groups of type $\tp B_3$ are dealt with in Corollary~\ref{cor-BAW-Spin7}.
We mention that for classical groups of types $\tp B$, $\tp C$ and $\tp D$, the quasi-isolated 2-blocks are just the principal 2-blocks.

Therefore, we are left with the task of establishing the inductive BAW condition for the quasi-isolated 2-blocks of $G$.
By Th\'eor\`eme A and the table on page 349 of \cite{En00}, the group $G$ has three unipotent 2-blocks: the principal block and two unipotent blocks of defect zero.
Thus by Theorem \ref{thm-BAW-principal}, the inductive BAW condition holds for all unipotent 2-blocks.
By \cite[\S3]{KM13}, the group $G$ possesses a non-unipotent quasi-isolated 2-block only when $3\nmid q$, and in which situation, the non-unipotent quasi-isolated 2-block is unique and dealt with in Theorem~\ref{thm-BAW-quasi-isolated}.
Thus we complete the proof of Theorem~\ref{mainthm:BAW-F4}.
\end{proof}


\section*{Acknowledgements} 
The authors are strongly indebted to Gunter Malle and Damiano Rossi for comments on an earlier version.


\end{document}